%% file: SimpliciesAndSimplicialComplex.tex
\documentclass[12pt,oneside,reqno]{amsart}
\usepackage{amssymb,tikz,amsmath,bm,color,graphicx,hyperref,anysize,comment}
\usepackage{lineno} %For number to each line in documents
\usepackage{soul}  % Import soul package for highlighting
\usepackage{multicol}
\usepackage{mathtools}
\usepackage{listings}
\usepackage{subcaption}
\usepackage{amsfonts}
\usetikzlibrary{circuits.logic.US,circuits.logic.IEC,fit}

\usetikzlibrary{cd}
\usepackage[numbers,sort&compress]{natbib}
\theoremstyle{plain}
\newtheorem{theorem}{Theorem}[section]
\newtheorem{lemma}[theorem]{Lemma}
\newtheorem{corollary}[theorem]{Corollary}

\newenvironment{solution*}{%
  \par\noindent\textbf{Solution. }\normalfont
}{\qed}
\theoremstyle{remark}
\newtheorem{workingrule}{Working Rule}[section]
\newtheorem{program}{Program}[section]
\theoremstyle{definition}
\newtheorem{example}{Example}[section]
\newtheorem{definition}[theorem]{Definition}
\theoremstyle{remark}
\newtheorem{remark}[theorem]{Remark}
\newtheorem{question}[]{\textbf{Question}}
\newtheorem{answer}{\textbf{Answer}}
\usetikzlibrary{3d}
\usetikzlibrary{shapes.geometric, arrows}
\tikzstyle{startstop} = [rectangle, rounded corners, minimum width=3cm, minimum height=1cm,text centered, draw=black, fill=red!20]
\tikzstyle{process} = [rectangle, minimum width=3cm, minimum height=1cm, text centered, draw=black, fill=blue!10]
\tikzstyle{decision} = [diamond, minimum width=3cm, minimum height=1cm, text centered, draw=black, fill=green!20]
\tikzstyle{arrow} = [thick,->,>=stealth]
\usetikzlibrary{3d}
\usetikzlibrary{shapes.geometric, arrows}
\tikzstyle{startstop} = [rectangle, rounded corners, minimum width=3cm, minimum height=1cm,text centered, draw=black, fill=red!20]
\tikzstyle{process} = [rectangle, minimum width=3cm, minimum height=1cm, text centered, draw=black, fill=blue!10]
\tikzstyle{decision} = [diamond, minimum width=3cm, minimum height=1cm, text centered, draw=black, fill=green!20]
\tikzstyle{arrow} = [thick,->,>=stealth]
\input{ShortcutSymbolsListForLatex}
\makeatletter
\def\@setcopyright{}
\def\serieslogo@{}
\makeatother
\lstdefinestyle{pythonstyle}{
    language=Python,
    basicstyle=\ttfamily\small,
    keywordstyle=\color{blue}\bfseries,
    stringstyle=\color{red},
    commentstyle=\color{gray},
    numbers=left,
    numberstyle=\tiny\color{gray},
    stepnumber=1,
    numbersep=5pt,
    backgroundcolor=\color{white},
    showspaces=false,
    showstringspaces=false,
    showtabs=false,
    frame=single,
    tabsize=4,
    captionpos=b,
    breaklines=true,
    breakatwhitespace=false,
    escapeinside={\%*}{*)}
}

\begin{document}
%\linenumbers % Start line numbering
\marginsize{2.54 cm}{2.54 cm}{2.54 cm}{2.54 cm}
% First we specify the top matter (author, title, etc).
% Note: All of the top matter items are optional and can be omitted.
% But you will probably want to specify at least the author and title
% and perhaps an abstract.
% author information
% first author
\author{Sanjay Mishra}
\address{Department of Mathematics \newline \indent Amity School of Applied Sciences, Amity University Lucknow Campus, UP India}
\email{drsmishraresearch@gmail.com}
%\author{Suhaila Barekzai}
%\address{Department of Mathematics \newline \indent Lovely Professional University, Phagwara-144411, Punjab, India}
%\email{s.barekzai1995@gmail.com}
% title
\title[Simplicial Complexes]{Foundations of Simplicial Complexes:From Geometric Independence to Realizations}
%%%%%%%%%%%%%%%%%%%%%%%%%%%%%%%%%%%%%%%%%%%%%%%%%%%%%%%%%%%%%%%%%%%%%%%%%%%%%%%%%%%%%%%%%%%%%%%%%%%%%%%%%%%%%%%%%
%                                                  Abstract
%%%%%%%%%%%%%%%%%%%%%%%%%%%%%%%%%%%%%%%%%%%%%%%%%%%%%%%%%%%%%%%%%%%%%%%%%%%%%%%%%%%%%%%%%%%%%%%%%%%%%%%%%%%%%%%%%
\begin{abstract}
This paper develops a complete foundational treatment of simplicial complexes from Euclidean spaces through geometric realizations, emphasizing concrete computations, examples, and practical verification methods. Beginning with finite point sets in finite and infinite-dimensional Euclidean spaces, geometric independence is established via linear independence of relative vectors, with explicit matrix rank tests. $n$-simplices arise as convex hulls of such independent points, proven convex, compact, uniquely spanned, and homeomorphic to unit balls, with detailed barycentric coordinate. Simplicial complexes form through collections closed under faces and with simplex intersections either empty or common faces, verified by necessary and sufficient disjoint interior conditions, illustrated across dimensions from lines to tetrahedra plus non-examples. Derived structures including subcomplexes, $p$-skeletons, vertices, stars, and links lead to geometric realizations as continuous spaces with weak topology, proven Hausdorff and locally compact, alongside ray characterizations of convexity and continuity via simplicial maps.

\end{abstract}
%%%%%%%%%%%%%%%%%%%%%%%%%%%%%%%%%%%%%%%%%%%%%%%%%%%%%%%%%%%%%%%%%%%%%%%%%%%%%%%%%%%%%%%%%%%%%%%%%%
% AMS subject classifications (used in AMS journals)
% AMS keywords (used in AMS journals)
% acknowledge support, etc
%\thanks{This research was partially supported by Lovely Professional University}
%\thanks{We would like to thank our colleagues for their helpful criticism.}
% dedication
%\dedicatory{ABC}
% today's date, or fill in whatever date you prefer
%\date{\today}
% This ends the top matter information.
% We can now tell LaTeX to display the top matter.
\maketitle
%\tableofcontents % Print the table of contents itself
% Having displayed the top matter, we now proceed to the body of the article.
\tableofcontents % Print the table of contents itself
%*****************************************************************************************************************
%                                         Introduction
%*****************************************************************************************************************
\section{Introduction}

Simplicial complexes provide a powerful combinatorial framework for studying the geometry and topology of spaces using finite building blocks called simplices. A simplex is the convex hull of a finite set of geometrically independent points in Euclidean space, such as a line segment in \(\mathbb{R}^1\), a triangle in \(\mathbb{R}^2\), or a tetrahedron in \(\mathbb{R}^3\).[web:228] By gluing simplices together along their faces in a controlled way, one obtains a simplicial complex, which encodes both the combinatorial structure (how simplices intersect) and the geometric shape of the underlying space.

From a geometric and topological point of view, simplicial complexes serve as discrete models of continuous spaces. Their geometric realizations are built by viewing each abstract simplex as a genuine Euclidean simplex and then identifying faces according to the combinatorial data. This construction yields a topological space, often called the polytope or underlying space of the complex, on which one may study fundamental notions such as continuity, connectedness, compactness, and local finiteness.

Simplicial complexes also play a central role in modern applications. In computational topology and topological data analysis, finite subsets of Euclidean space are used as vertex sets from which one constructs complexes that approximate the shape of data. In numerical analysis and the finite element method, triangulations of domains are modeled as simplicial complexes to support piecewise linear approximations of functions and solutions of differential equations. These examples illustrate how the abstract theory of simplicial complexes connects discrete structures with continuous phenomena.

The goal of this article is to develop an expository treatment of simplicial complexes starting from Euclidean spaces and finite data sets, and then progressing to geometric independence, simplices, simplicial complexes, and their geometric realizations. Along the way, emphasis is placed on explicit examples, computational viewpoints, and structural results such as necessary and sufficient conditions for a collection of simplices to form a simplicial complex, the relationship between geometric and linear independence, and the topological properties of the underlying space.

The article is structured as follows. Section \eqref{s:Euclidean Space} provides a comprehensive foundation on Euclidean spaces, covering both finite-dimensional spaces like the familiar plane and three-dimensional room, and infinite-dimensional spaces, with detailed definitions of finite point sets as the basic building blocks for geometric constructions. It proves these spaces form vector spaces for addition and scaling, normed spaces with length measurements, metric spaces for distances between points, and topological spaces with notions of openness and continuity, illustrated through examples. Section \eqref{s:Simplices} introduces simplices as the convex hulls formed by geometrically independent points, rigorously proving their key properties: convexity showing line segments between any two points stay inside, compactness ensuring every sequence has a convergent subsequence within the simplex, uniqueness of the spanning vertex set up to ordering, properties of the interior as open and convex within its affine plane, and a homeomorphism to the unit ball that maps the boundary precisely to the unit sphere. Section \eqref{s:Simplicial Complex in RN} defines simplicial complexes as collections of simplices closed under taking faces where any two simplices intersect either emptily or along a common face, establishing necessary and sufficient conditions through disjoint interiors of distinct simplices, defining complex dimension as the highest simplex dimension present, and providing concrete examples from one-dimensional line segments through two-dimensional filled triangles to three-dimensional tetrahedra, alongside counterexamples of invalid collections.

\section{Euclidean Space}\label{s:Euclidean Space}
Let us begin our study with some essential preliminaries. Euclidean space is a fundamental concept in mathematics, originating from Euclidean geometry. It refers to a space in which familiar notions such as distance, angle, and geometric shapes are well-defined.

In computational topology, data often originates as a finite collection of points in Euclidean space. These points may represent physical locations, pixel coordinates, measurements, or abstract features in the space. From a topological perspective, such a finite subset of Euclidean space serves as the foundational input for constructing simplices and, ultimately, simplicial complexes. This section introduces the interpretation of a finite subset of Euclidean space as a vertex set, which forms the basis for building higher-dimensional geometric objects. We discuss how to treat such point sets as discrete data, describe the conditions under which they may generate simplices, and prepare this data for computational processing in later sections.

\begin{definition}[Finite-Dimensional Euclidean Space as Data Set]\label{d:[Finite-Dimensional Euclidean Space as Data Set}
The finite-dimensional Euclidean space (denoted by $\R^m$) of dimension $m \in \N$, is defined as
\begin{equation}\label{eq:Finite Dimensional Euclidean Space}
\mbb{R}^m = \{ \btext{x} =  (x_1, x_2, \ldots, x_m) \colon x_i \in \R \text{ for } i = 1, 2, \ldots, m \}    
\end{equation}
\end{definition}

\begin{definition}[Infinite-Dimensional Euclidean Space]\label{d:Infinite Dimensional Euclidean Space}
The infinite dimensional Euclidean space (denoted by $\R^{\N}$) of infinite dimension is defined 
\begin{equation}\label{eq:Infinite Dimensional Euclidean Space}
\mbb{R}^{\N} = \{ \btext{x} =  (x_1, x_2, \ldots) \colon x_i \in \R, \, i \in \N \}  
\end{equation}
\end{definition}

\begin{remark}
Euclidean spaces $\R^m$ and $\R^{\N}$ are both infinite spaces but their dimensions are finite and infinite respectively.
\end{remark}

\begin{definition}[Finite Subset of $\mathbb{R}^m$]\label{d:Finite Subset in Rm}
A subset $A$ of $\mathbb{R}^m$ is said to be a finite subset if it contains only finitely many elements. That is,
\begin{equation}
A = \{ \mathbf{a}_1, \mathbf{a}_2, \ldots, \mathbf{a}_n \} \subseteq \mathbb{R}^m \quad \text{for some } n \in \mathbb{N},
\end{equation}
where $\mathbf{a}_i = (x_{i1}, x_{i2}, \ldots, x_{im}) \in \mathbb{R}^m$ for each $i = 1, 2, \ldots, n$.
\end{definition}

\begin{remark}
For convenience, we consider a finite subset of $(n + 1)$ elements in the space $\R^m$, denoted by $A = \{\btext{a}_0, \btext{a}_1, \ldots, \btext{a}_n\}$.
\end{remark}

\begin{definition}[Finite Subset in $\R^N$]\label{d:Finite Subset in RN}
Let $J = \{0, 1, 2, \ldots, n\}, n \in \N$ is indexed set. The finite subset $A$ of $\R^N$ is 
\begin{equation}\label{eq:Finite Subset in RN}
A = \{\btext{a}_{j} \in \R^N \colon \btext{a}_{j} = (x_1, x_2, \ldots), x_1, x_2, \ldots \in \R, \forall j \in J\}    
\end{equation}
\end{definition}

It can be seen that each element in the set $A$ is an infinite-tuple of real numbers and the number of elements is $n+1$ which is finite. For convenient we use $A = \{\btext{a}_{0}, \ldots \btext{a}_{n}\}$.

\begin{definition}[Generalized Euclidean space]\label{d:Generalized Euclidean space}
The generalized Euclidean space, $\mathbb{E}^J$, is defined as the subset of $\mathbb{R}^J$ consisting of all points $(x_\alpha)_{\alpha \in J}$ such that $x_\alpha = 0$ for all but finitely many values of $\alpha$; it is a vector space called the generalized Euclidean space and is endowed with the metric $|x - y| = \max \{|x_{\alpha} - y_{\alpha}|\}_{\alpha \in J}$.
\end{definition}

So far,  the Euclidean space $\R^n$ has been defined only as a set, but other structure like algebraic, geometric and topological can be imposed on it.

Let us see some elementary example of Euclidean space of various dimensions.

\begin{example}
The 1-dimension Euclidean space $\R^1$ (or $\R$) is the real line. The 2-dimension Euclidean space $\R^2 = \{(x_1, x_2) \colon x_i \in \R, i = 1, 2\}$ is the Cartesian plane. The 3-dimension Euclidean space $\R^3 = \{(x_1, x_2, x_3) \colon x_i \in \R, i = 1, 2, 3\}$ is the 3-dimensional space. 
\end{example}

\begin{example}[Representation of Finite Data Set]\label{eg:Representation of Finite Data Set}
How can the academic performance data of 30 students across 10 subjects, each graded on a scale from 0 to 100, be represented as a subset of the Euclidean space $\mathbb{R}^{10}$, both in set-builder notation and tabular form?
\end{example}

\begin{solution*}
The academic performance data can be viewed as a set of 30 points in $\mathbb{R}^{10}$, where each point corresponds to a student's scores across the 10 subjects. Formally, 
\[
S = \{ \mathbf{a} = (x_1, x_2, \ldots, x_{10}) \in \mathbb{R}^{10} : 0 \leq x_i \leq 100, \, i=1,2,\ldots,10 \},
\]
where each $\mathbf{x}$ represents a student's performance vector. This set $S$ contains exactly 30 such vectors. The tabular representation is given in Table~\ref{tab:academic_performance}.
\begin{table}[h!]
\centering
\small
\begin{tabular}{c|cccccccccc}
\text{Student} & \text{Subj 1} & \text{Subj 2} & \text{Subj 3} & \text{Subj 4} & \text{Subj 5} & \text{Subj 6} & \text{Subj 7} & \text{Subj 8} & \text{Subj 9} & \text{Subj 10} \\
\hline
1 & 75 & 88 & 91 & 69 & 84 & 79 & 85 & 90 & 78 & 82 \\
2 & 83 & 76 & 85 & 90 & 74 & 81 & 69 & 86 & 80 & 77 \\
3 & 79 & 82 & 73 & 87 & 91 & 78 & 84 & 75 & 88 & 90 \\
4 & 92 & 74 & 80 & 85 & 77 & 83 & 75 & 79 & 86 & 81 \\
5 & 88 & 90 & 76 & 80 & 78 & 84 & 72 & 82 & 74 & 85 \\
6 & 70 & 85 & 89 & 77 & 83 & 75 & 80 & 71 & 88 & 79 \\
7 & 85 & 79 & 82 & 74 & 88 & 91 & 77 & 83 & 90 & 72 \\
8 & 83 & 77 & 84 & 80 & 75 & 88 & 79 & 85 & 74 & 81 \\
9 & 76 & 90 & 71 & 83 & 78 & 82 & 84 & 79 & 86 & 80 \\
10 & 88 & 74 & 85 & 78 & 90 & 76 & 83 & 72 & 79 & 84 \\
11 & 84 & 81 & 76 & 87 & 75 & 90 & 88 & 74 & 82 & 79 \\
12 & 78 & 89 & 83 & 80 & 84 & 77 & 85 & 73 & 90 & 75 \\
13 & 90 & 75 & 78 & 82 & 79 & 88 & 73 & 80 & 77 & 85 \\
14 & 83 & 80 & 84 & 75 & 86 & 90 & 72 & 85 & 79 & 77 \\
15 & 77 & 84 & 75 & 89 & 80 & 85 & 74 & 90 & 83 & 78 \\
16 & 85 & 78 & 90 & 73 & 82 & 79 & 88 & 75 & 84 & 77 \\
17 & 79 & 83 & 87 & 80 & 75 & 85 & 90 & 78 & 72 & 84 \\
18 & 88 & 74 & 82 & 86 & 79 & 77 & 83 & 80 & 85 & 90 \\
19 & 84 & 80 & 75 & 79 & 90 & 73 & 85 & 77 & 88 & 74 \\
20 & 75 & 85 & 79 & 84 & 77 & 90 & 82 & 78 & 73 & 86 \\
21 & 90 & 78 & 83 & 75 & 84 & 79 & 87 & 80 & 72 & 85 \\
22 & 83 & 79 & 86 & 72 & 90 & 85 & 74 & 78 & 80 & 77 \\
23 & 76 & 87 & 75 & 80 & 78 & 90 & 83 & 85 & 74 & 82 \\
24 & 88 & 74 & 80 & 85 & 79 & 72 & 84 & 77 & 90 & 78 \\
25 & 85 & 79 & 83 & 78 & 84 & 75 & 87 & 80 & 72 & 90 \\
26 & 78 & 83 & 80 & 75 & 85 & 79 & 90 & 82 & 77 & 85 \\
27 & 72 & 85 & 74 & 79 & 83 & 90 & 75 & 84 & 78 & 80 \\
28 & 90 & 77 & 85 & 80 & 78 & 74 & 82 & 79 & 85 & 83 \\
29 & 84 & 79 & 90 & 75 & 80 & 83 & 77 & 72 & 85 & 79 \\
30 & 79 & 85 & 77 & 83 & 90 & 80 & 84 & 78 & 75 & 82 \\
\end{tabular}
\vspace{0.5cm} % Adds vertical space above table
\caption{Academic performance data for 30 students across 10 subjects.}
\label{tab:academic_performance}
\vspace{0.3cm} % Adds vertical space after caption before table
\end{table}
\end{solution*}

Before presenting fundamental theorems about Euclidean space, it is important to understand that Euclidean space $\R^{n}$ serves as a foundational example connecting various mathematical structures. It is not only a vector space, supporting operations of addition and scalar multiplication, but also naturally equipped with a norm, metric, and topology derived from its algebraic and geometric properties. These layers of structure enable rigorous definitions and proofs across diverse areas such as linear algebra, analysis, and topology. The following theorems summarize these central viewpoints on Euclidean space, laying the groundwork for advanced study and applications in mathematics.

\begin{theorem}[Euclidean Space as Vector Space]\label{t:Euclidean Space as Vector Space}
The Euclidean space $\R^n$ is a vector space over the field $\R$ with vector addition and scalar multiplication 
\[\btext{x} + \btext{y} = (x_1 + y_1, x_2 + y_2, \ldots, x_n + y_n), \, 
\alpha \btext{x} = (\alpha x_1, \alpha x_2, \ldots, \alpha x_n)\]
for all $\btext{x},\btext{y} \in \R^n$ and $\alpha \in \R$.
\end{theorem}

\begin{theorem}[Euclidean Space as Norm Space]\label{t:Euclidean Space as Norm Space}
The Euclidean space $\R^n$ is norm space over the field $\R$ with norm 
\[\| \btext{x} \| = \sqrt{x_1^2 + x_2^2 + \cdots + x_n^2}\]
for all $\btext{x} \in \R^n$.
\end{theorem}

\begin{theorem}[Euclidean Space as Metric Space]\label{t:Euclidean Space as Metric Space}
The Euclidean space $\R^n$ is metric space with metric
\[d(\btext{x}, \btext{y}) = \| \btext{x} - \btext{y} \| = \sqrt{(x_1 - y_1)^2 + (x_2 - y_2)^2 + \cdots + (x_n - y_n)^2} \]
for all $\btext{x}, \btext{y} \in \R^n$.
\end{theorem}

\begin{theorem}[Euclidean Space as Topological Space]\label{t:Euclidean Space as Topological Space}
The Euclidean space $\R^n$ is topological space with standard topology that is the collection of open sets. The set $U \in \R^n$ is open in this standard topology, if for every point $\btext{x} \in U$ there exists a radius $\epsilon > 0$ such that open ball 
\[B_{\epsilon}(\btext{x}) = \{\btext{y} \in \R^n \colon d(\btext{x}, \btext{y}) < \epsilon\}\]
is entirely contained in $U$. 
\end{theorem}

\subsection{Analytic geometry of Euclidean space} \label{ss:Analytic geometry of Euclidean space}
Analytic geometry, combines algebra and geometry to study geometric objects and their relationships using coordinate systems and algebraic equations. In the context of Euclidean space $\R^n$, it provides a structured way to represent points, lines, planes, and curves in spaces of any dimension using ordered tuples of real numbers. Now we describe some elementary subsets of $\R^n$ from an analytic geometric perspective. We begin with the concept of a geometrically  (or affinely\footnote{Geometric independence and affine independence refer to the same concept; the former is used in topology and geometry, while the latter is used in linear algebra.}) independent sets of points of $\R^N$, which plays a central role in the construction of simplices and simplicial complexes.

\begin{definition}[Geometrically Independent Set]\label{d:Geometrically Independent Set}
A set $A = \{\btext{a}_{0}, \ldots \btext{a}_{n}\}$ of points of $\R^{N}$ is said to be geometrically independent if, for any real scalars $t_{i}$, holds following implication: 
\begin{equation}\label{eq:Geometrically Independent Set of Point of RN}
\sum_{i = 0}^{n}t_{i} = 0 \, \text{and} \, \sum_{i = 0}^{n}t_{i}\btext{a}_{i} = \btext{0} \Rightarrow t_{i} = 0, \, \forall i
\end{equation}
\end{definition}

As we saw from the standard result that Euclidean space is also a vector space, so let us see in the next result how geometrically independence is related to linear independence which will help us in verifying geometrically independence of subsets of $\R^n$.  Geometric independence and linear independence are related but distinct concepts, especially in the context of affine geometry and linear algebra. To better understand first we recall linear independent set in $\R^N$. 

\begin{definition}[Linearly Independent Set]\label{d:Linearly Independent Set}
A set of vectors $V = \{\btext{v}_1, \ldots, \btext{v}_m\}$ in $\R^N$ is said to be linearly independent if, for any real scalars $\lambda_{i}$ holds following implication:  
\begin{equation}\label{eq:Linearly Independent Set}
 \sum_{i = 1}^{m}\lambda_{i}\btext{v}_{i} = \btext{0} \Rightarrow \lambda_{i} = 0, \, \forall i   
\end{equation}
\end{definition}
Let us look at some examples that clarifies the difference between geometric and linear independence.

\begin{example}\label{eg:Geometric and linear independence of the sets1}
Is the set $ A = \{\btext{a}_{0} = (0,0), \btext{a}_{1} = (1,0), \btext{a}_{2} = (2,0)\} $ in $\mathbb{R}^2$ geometrically independent or linearly independent?
\end{example}

\begin{solution*}
A set $ A = \{\btext{a}_{0} = (0,0), \btext{a}_{1} = (1,0), \btext{a}_{2} = (2,0)\} $ is geometrically independent if, for any real scalars $ t_0, t_1, t_2 $, the following implication holds:
\[
\sum_{i=0}^{2} t_i = 0 \quad \text{and} \quad \sum_{i=0}^{2} t_i \btext{a}_i = \btext{0} \quad \Rightarrow \quad t_0 = t_1 = t_2 = 0.
\]

Assume $ t_0 = 1 $, $ t_1 = -2 $, $ t_2 = 1 $; then 
\[
\sum_{i=0}^{2} t_i = 1 + (-2) + 1 = 0
\]
and
\[
\sum_{i=0}^{2} t_i \btext{a}_i = (1)(0,0) + (-2)(1,0) + (1)(2,0) = (0,0) = \btext{0}.
\]

However, not all $ t_0, t_1, t_2 $ are zero; therefore, the set $ A $ is not geometrically independent.

On the other hand, if the points of $ A $ are regarded as position vectors, then the set of vectors $ V = \{\btext{v}_1 = \btext{a}_1 - \btext{a}_0 = (1,0), \btext{v}_2 = \btext{a}_2 - \btext{a}_0 = (2,0)\} $ is said to be linearly independent if, for any real scalars $ \lambda_1, \lambda_2 $, the following implication holds:
\[
\sum_{i=1}^{2} \lambda_i \btext{v}_i = \btext{0} \quad \Rightarrow \quad \lambda_1 = \lambda_2 = 0.
\]

Assume $ \lambda_1 = 2 $, $ \lambda_2 = -1 $; then
\[
\sum_{i=1}^{2} \lambda_i \btext{v}_i = 2(1,0) + (-1)(2,0) = (0,0) = \btext{0}.
\]

But not all $ \lambda_1, \lambda_2 $ are zero; therefore, the set $ A $ is not linearly independent. Hence, $ A $ is neither geometrically independent nor linearly independent.
\end{solution*}

\begin{example}\label{eg:Geometric and linear independence of the sets2}
Is the set $ A = \{\btext{a}_{0} = (0,0), \btext{a}_{1} = (1,0), \btext{a}_{2} = (0,1)\} $ in $\mathbb{R}^2$ both geometrically independent and linearly independent? 
\end{example}

\begin{solution*}
A set $ A = \{\btext{a}_{0} = (0,0), \btext{a}_{1} = (1,0), \btext{a}_{2} = (0,1)\} $ is geometrically independent since the only scalars $ t_0, t_1, t_2 $ satisfying 
\[
\sum_{i=0}^{2} t_i = 0
\]
and
\[
\sum_{i=0}^{2} t_i \btext{a}_i = 0 \cdot (0,0) + 0 \cdot (1,0) + 0 \cdot (0,1) = (0,0) = \btext{0}
\]
are $ t_0 = t_1 = t_2 = 0 $.

Therefore, the set $ A $ is geometrically independent.

On the other hand, the set of vectors $ V = \{\btext{v}_1 = \btext{a}_1 - \btext{a}_0 = (1,0), \btext{v}_2 = \btext{a}_2 - \btext{a}_0 = (0,1)\} $ is linearly independent since the only scalars $ \lambda_1, \lambda_2 $ satisfying 
\[
\sum_{i=1}^{2} \lambda_i \btext{v}_i = 0 \cdot (1, 0) + 0 \cdot (0, 1) = (0,0) = \btext{0}
\]
are $ \lambda_1 = \lambda_2 = 0 $.

Therefore, the set $ A $ is linearly independent. Hence, $ A $ is both geometrically independent and linearly independent.
\end{solution*}

By observing the above two Examples \eqref{eg:Geometric and linear independence of the sets1} and \eqref{eg:Geometric and linear independence of the sets2}, we can establish a useful connection between geometric and linear independence of sets in $\mathbb{R}^N$. The following Theorem \eqref{t:Necessary and Sufficient Condition for GI of Set} provides necessary and sufficient conditions for the geometric independence of a set in terms of linear independence. This result will help solve problems related to geometric independence by applying linear algebra techniques, making the process more time-efficient and computationally convenient.

\begin{theorem}[Necessary and Sufficient Condition for Geometric Independent Set]\label{t:Necessary and Sufficient Condition for GI of Set}
A set $A = \{\btext{a}_{0}, \ldots \btext{a}_{n}\}$ of points of $\R^{N}$ is geometrically independent if and only if the set of vectors 
\[V = \{\btext{a}_{1} - \btext{a}_{0}, \ldots, \btext{a}_{n} - \btext{a}_{0}\}\]
is linearly independent in the sense of linear algebra. 
\end{theorem}

\begin{proof}
First we prove that geometric independence implies linear independence. As given set $A = \{\btext{a}_{0}, \ldots \btext{a}_{n}\}$ of points of $\R^{N}$ is geometrically independent. Now we prove that set of vectors $V = \{\btext{a}_{1} - \btext{a}_{0}, \ldots, \btext{a}_{n} - \btext{a}_{0}\}$ is linear independence. 

By definition, a set of points $A = \{\btext{a}_{0}, \ldots \btext{a}_{n}\}$ in $ \R^N $  is geometrically independent if the only way to express any point in $A$  as an affine combination of the other points in $A$  is the trivial case where one of the coefficients is $1$  and the rest are $0$ . That is, if
\[\sum_{i=0}^{n} t_i \btext{a}_i = \btext{0}, \quad \text{with} \quad \sum_{i=0}^{n} t_i = 0,\]
then necessarily $ t_i = 0 $  for all $ i $ .

Rewriting the sum in terms of the vectors $ \btext{a}_i - \btext{a}_0 $ , we get
\[\sum_{i=1}^{n} t_i (\btext{a}_i - \btext{a}_0) = -t_0 \btext{a}_0.\]
Since $ \sum_{i=0}^{n} t_i = 0 $ , we substitute $ t_0 = -\sum_{i=1}^{n} t_i $ , yielding:
\[\sum_{i=1}^{n} t_i (\btext{a}_i - \btext{a}_0) - \left(\sum_{i=1}^{n} t_i\right) \btext{a}_0 = \btext{0}.\]
This simplifies to:
\[\sum_{i=1}^{n} t_i (\btext{a}_i - \btext{a}_0) = \btext{0}.\]
Since $ A $  is geometrically independent, the only solution to this equation is $ t_1 = t_2 = \dots = t_n = 0 $ .  
Thus, the vectors $ \btext{a}_1 - \btext{a}_0, , \dots, \btext{a}_n - \btext{a}_0 $  are linearly independent.

Conversely, we prove that linear independence implies geometric independence. Now, suppose that the set $V = \{\btext{a}_1 - \btext{a}_0, \dots, \btext{a}_n - \btext{a}_0 \}$ is linearly independent. We need to prove that the set $A = \{\btext{a}_{0}, \ldots \btext{a}_{n}\}$ is geometrically independent. Suppose we have an affine dependence relation among the points in $A$
\[\sum_{i=0}^{n} t_i \btext{a}_i = \btext{0}, \quad \text{with} \quad \sum_{i=0}^{n} t_i = 0.\]

Rewriting,
\[t_0 \btext{a}_0 + \sum_{i=1}^{n} t_i \btext{a}_i = \btext{0}.\]

Rearranging in terms of the vectors,
\[\sum_{i=1}^{n} t_i (\btext{a}_i - \btext{a}_0) = -t_0 \btext{a}_0.\]

Since $ \sum_{i=0}^{n} t_i = 0 $ , we get $ t_0 = -\sum_{i=1}^{n} t_i $ , giving
\[\sum_{i=1}^{n} t_i (\btext{a}_i - \btext{a}_0) - \left(\sum_{i=1}^{n} t_i\right) \btext{a}_0 = \btext{0}.\]

On simplifying, we get
\[\sum_{i=1}^{n} t_i (\btext{a}_i - \btext{a}_0) = \btext{0}.\]

Since the vectors $ \btext{a}_1 - \btext{a}_0, \dots, \btext{a}_n - \btext{a}_0 $  are linearly independent, the only solution is $ t_1 = t_2 = \dots = t_n = 0 $ . From $ t_0 = -\sum_{i=1}^{n} t_i $ , it follows that $ t_0 = 0 $  as well. Thus, the only possible affine combination satisfying the condition is the trivial one, proving that $ A $  is geometrically independent.
\end{proof}

\begin{corollary}\label{cor:One point set is GI}
One point set of $\R^{N}$ is always geometrically independent. 
\end{corollary}

\begin{corollary}\label{cor:More than one point set is GI}
In $\R^{N}$ two distinct points, three non-collinear points, four non-coplanar points and so on form a geometrically independent set.
\end{corollary}

Using the theory of geometric independence of sets as we discussed earlier, we now present a systematic working rule for determining whether a given finite set in $\R^n$ is geometrically independent.

\begin{workingrule}[Determining the Geometric Independence of a Finite Set]\label{wr:Determining the Geometric Independence of a Finite Set} 
Let $A = \{\btext{a}_{0}, \btext{a}_{1}, \ldots \btext{a}_{n}\}$ be the set of $(n + 1)$ points in $\R^m$. To determine geometrically independence of $A$ follow these steps:
\begin{enumerate} 
\item \textbf{Formulate the position vectors:} Suppose the elements $ \btext{a}_0, \btext{a}_1,  \ldots, \btext{a}_n $ of $A$ as position vectors  in $\R^m$.
\item \textbf{Define the set of relative vectors:} Define the set $V$ of relative vectors with respect to $\btext{a}_0$ as
\[V = \{\textbf{v}_{i} \colon \btext{v}_i = \btext{a}_i - \btext{a}_0, \, \text{for} \,  i = 1, 2, \ldots, n\} \]
The total $n$ vectors $\btext{v}_1, \ldots, \btext{v}_n$ will be used to determine independence.
\item \textbf{Form a matrix of relative vectors:} Construct a matrix $M_V$ whose columns are the vectors $ \btext{v}_1, \btext{v}_2, \ldots, \btext{v}_n$
\[
M_V = \begin{bmatrix}
v_{11} & v_{12} & \cdots & v_{1n} \\
v_{21} & v_{22} & \cdots & v_{2n} \\
    \vdots & \vdots & \ddots & \vdots \\
v_{n1} & v_{n2} & \cdots & v_{mn}
\end{bmatrix}_{m \times n}
\]
\item \textbf{Determine the rank of $M_V$:} Calculate the rank of the matrix $M_V$ by any suitable method such as the Echelon method or the normal method. 
 \item \textbf{Apply the Geometric Independence Criterion:}
\begin{enumerate}
\item If $\operatorname{rank}(M_V) = n$, then  $A$ is geometrically independent.
\item If $\operatorname{rank}(M_V) < n$, then $A$ is not geometrically independent.
\end{enumerate}
\end{enumerate}
\end{workingrule}

\begin{example}\label{eg:Example of a Geometrically Independent Set in R2}
Is the set $ A = \{\btext{a}_0 = (0, 0), \btext{a}_1 = (1, 1)\} $ in $\mathbb{R}^2$ geometrically independent?
\end{example}

\begin{solution*}
Since $A$ has distinct elements, it is geometrically independent by Corollary \eqref{cor:More than one point set is GI}.
\end{solution*}

\begin{example}\label{eg:Example of a Geometrically Independent Set in R2}
Is the set $ A = \{\btext{a}_0 = (0, 0), \btext{a}_1 = (1, 0), \btext{a}_2 = (0, 1) \} $ in $\mathbb{R}^2$ geometrically independent?
\end{example}

\begin{solution*} The set $A = \{\btext{a}_0 = (0, 0), \btext{a}_1 = (1, 0), \btext{a}_2 = (0, 1) \}$ in $\R^2$.
\begin{enumerate}
\item  Suppose the elements $\btext{a}_0, \btext{a}_1, \btext{a}_2$ of $A$ as position vectors, so 
\[    
\btext{a}_0 = \begin{bmatrix} 0 \\ 0 \end{bmatrix}, \quad
\btext{a}_1 = \begin{bmatrix} 1 \\ 0 \end{bmatrix}, \quad
\btext{a}_2 = \begin{bmatrix} 0 \\ 1 \end{bmatrix}    
\]
\item The set of relative vectors with respect to position vector $\btext{a}_0$ is $V = \{\btext{v}_1, \btext{v}_2\}$, where 
\[    
\btext{v}_1 = \btext{a}_1 - \btext{a}_0 = \begin{bmatrix} 1 \\ 0 \end{bmatrix}, \quad
\btext{v}_2 = \btext{a}_2 - \btext{a}_0 = \begin{bmatrix} 0 \\ 1 \end{bmatrix}    
\]
\item The matrix of relative vectors $\btext{v}_1, \btext{v}_2$ is 
\[ M_V = \begin{bmatrix}
1 & 0 \\
0 & 1
\end{bmatrix}   
\]
\item Since the number of non-zero rows in the matrix $M_V$ is 2 so by Echelon form $\operatorname{rank}(M_V) = 2$. 
\item Since $\operatorname{rank}(M_V) = 2 = n$, therefore the set $A$ is geometrically independent.
\end{enumerate}
\end{solution*} 

\begin{example}\label{eg:Example of a Geometrically Independent Set in R3}
Is the set $ A = \{\btext{a}_0 = (0, 0, 0), \btext{a}_1 = (1, 0, 0), \btext{a}_2 = (0, 1, 0)\} $ in $\mathbb{R}^3$ geometrically independent?
\end{example}

\begin{solution*}
The set 
$A = \{\btext{a}_0 = (0, 0, 0), \btext{a}_1 = (1, 0, 0), \btext{a}_2 = (0, 1, 0)\} $
in $\R^3$.  
\begin{enumerate}
\item  Suppose the elements $\btext{a}_0, \btext{a}_1, \btext{a}_2$ of $A$ as position vectors, so 
\[ \btext{a}_0 = \begin{bmatrix} 0 \\ 0 \\ 0 \end{bmatrix}, \quad
\btext{a}_1 = \begin{bmatrix} 1 \\ 0 \\ 0 \end{bmatrix}, \quad
    \btext{a}_2 \begin{bmatrix} 0 \\ 1 \\ 0 \end{bmatrix} 
    \]
\item The set of relative vectors with respect to position vector $\btext{a}_0$ is $V = \{\btext{v}_1, \btext{v}_2\}$, where 
\[\btext{v}_1 = \btext{a}_1 - \btext{a}_0 = \begin{bmatrix} 1 \\ 0 \\ 0 \end{bmatrix}, \quad
    \btext{v}_2 = \btext{a}_2 - \btext{a}_0 = \begin{bmatrix} 0 \\ 1 \\ 0 \end{bmatrix}  \]
\item The matrix of relative vectors $\btext{v}_1, \btext{v}_2$ is
\[M_V = \begin{bmatrix}
    1 & 0 \\
    0 & 1 \\
    0 & 0
    \end{bmatrix}_{3 \times 2}\]
\item Since the number of non-zero rows in the matrix $M_V$ is 2 so by Echelon form $\operatorname{rank}(M_V) = 2 = n$, therefore the set $A$ is geometrically independent.
\end{enumerate}
\end{solution*}

\begin{example}\label{eg:Non Example of a Geometrically Independent Set in R3}
Is the set $ A = \{\btext{a}_0 = (0, 0, 0), \btext{a}_1 = (1, 0, 0), \btext{a}_2 = (2, 0, 0)\} $ in $\mathbb{R}^3$ geometrically independent?
\end{example}

\begin{solution*}
The set $A = \{\btext{a}_0 = (0, 0, 0) , \btext{a}_1 = (1, 0, 0), \btext{a}_2 = (2, 0, 0)\}$ in $\R^3$. 
\begin{enumerate}
\item  Suppose the elements $\btext{a}_0, \btext{a}_1, \btext{a}_2$ of $A$ as position vectors, so 
\[ \btext{a}_0 = 
\begin{bmatrix} 
0 \\ 0 \\ 0 
\end{bmatrix}, \quad
\btext{a}_1 = 
\begin{bmatrix}
1 \\ 0 \\ 0 
\end{bmatrix}, \quad
\btext{a}_2 
\begin{bmatrix} 
2 \\ 0 \\ 0 
\end{bmatrix} \]
\item The set of relative vectors with respect to position vector $\btext{a}_0$ is $V = \{\btext{v}_1, \btext{v}_2\}$, where 
\[\btext{v}_1 = \btext{a}_1 - \btext{a}_0 = \begin{bmatrix} 1 \\ 0 \\ 0 \end{bmatrix}, \quad
    \btext{v}_2 = \btext{a}_2 - \btext{a}_0 = \begin{bmatrix} 2 \\ 0 \\ 0 \end{bmatrix}  \]
    \item The matrix of relative vectors $\btext{v}_1, \btext{v}_2$ is
\[M_V = \begin{bmatrix}
    1 & 2 \\
    0 & 0 \\
    0 & 0
    \end{bmatrix}_{3 \times 2}\]
\item Since the number of non-zero rows in the matrix $M_V$ is 1 so by Echelon form $\operatorname{rank}(M_V) = 1  \neq 2 =  n$, therefore the set $A$ is not geometrically independent.
\end{enumerate}
\end{solution*}

\begin{example}\label{eg:Example of a Geometrically Independent Set in R4}
Is the set $ A = \{\btext{a}_0, \btext{a}_1, \btext{a}_2, \btext{a}_3, \btext{a}_4\} $, where the points 
\[
\btext{a}_0 = (1, 1, 1, 1), \quad \btext{a}_1 = (2, 3, 1, 4), \quad \btext{a}_2 = (3, 5, 2, 1), \quad \btext{a}_3 = (4, 6, 3, 7), \quad \btext{a}_4 = (5, 9, 6, 3)
\]
in $\mathbb{R}^4$, geometrically independent?
\end{example}

\begin{solution*}
Consider the set $ A = \{\btext{a}_0, \btext{a}_1, \btext{a}_2, \btext{a}_3, \btext{a}_4\} $, where the points are 
\[
\btext{a}_0 = \begin{bmatrix} 1 \\ 1 \\ 1 \\ 1 \end{bmatrix}, \quad 
\btext{a}_1 = \begin{bmatrix} 2 \\ 3 \\ 1 \\ 4 \end{bmatrix}, \quad
\btext{a}_2 = \begin{bmatrix} 3 \\ 5 \\ 2 \\ 1 \end{bmatrix}, \quad 
\btext{a}_3 = \begin{bmatrix} 4 \\ 6 \\ 3 \\ 7 \end{bmatrix}, \quad
\btext{a}_4 = \begin{bmatrix} 5 \\ 9 \\ 6 \\ 3 \end{bmatrix}.
\]

\begin{enumerate}
\item Define the set of relative vectors with respect to $\btext{a}_0$ as 
\[
V = \{\btext{v}_1, \btext{v}_2, \btext{v}_3, \btext{v}_4\},
\]
where 
\[
\btext{v}_1 = \btext{a}_1 - \btext{a}_0 = \begin{bmatrix} 1 \\ 2 \\ 0 \\ 3 \end{bmatrix}, \quad 
\btext{v}_2 = \btext{a}_2 - \btext{a}_0 = \begin{bmatrix} 2 \\ 4 \\ 1 \\ 0 \end{bmatrix}, \quad
\btext{v}_3 = \btext{a}_3 - \btext{a}_0 = \begin{bmatrix} 3 \\ 5 \\ 2 \\ 6 \end{bmatrix}, \quad 
\btext{v}_4 = \btext{a}_4 - \btext{a}_0 = \begin{bmatrix} 4 \\ 8 \\ 5 \\ 2 \end{bmatrix}.
\]

\item Form the matrix $M_V$ whose columns are the vectors $\btext{v}_1, \btext{v}_2, \btext{v}_3, \btext{v}_4$:
\[
M_V = \begin{bmatrix}
1 & 2 & 3 & 4 \\
2 & 4 & 5 & 8 \\
0 & 1 & 2 & 5 \\
3 & 0 & 6 & 2
\end{bmatrix}.
\]

\item Compute the row-echelon form of $M_V$:
\[
\approx
\begin{bmatrix}
1 & 0 & -1 & 6 \\
0 & 1 & 2 & 5 \\
0 & 0 & -1 & 0 \\
0 & 0 & 0 & 20
\end{bmatrix}.
\]

\item Since the number of nonzero rows in the echelon form of $M_V$ is 4, the rank of $M_V$ is $4 = n$. Therefore, the set $A$ is geometrically independent.
\end{enumerate}
\end{solution*}

\begin{comment}
\begin{example}
Let us take a look at the data of ten common clinical parameters used in various medical tests that help in diagnosing diseases. Each parameter generally corresponds to a specific blood test. See Table \eqref{fig:Diagnostic Test Parameters for 20 Individuals}.
\begin{figure}[h!]
        \centering
        \includegraphics[scale=0.9]{CloudData1.png}
        \caption{Diagnostic Test Parameters for 20 Individuals}
        \label{fig:Diagnostic Test Parameters for 20 Individuals}
    \end{figure}  
Write this cloud data in Euclidean form and further investigate its geometric independence.  
\end{example}

\begin{solution*}
The cloud data as given in the table \eqref{fig:Diagnostic Test Parameters for 20 Individuals} can write in Euclidean form as
\begin{equation}\label{eq:Diagnostic Test Parameters for 20 Individuals}
\R^{10} = \{\btext{a} \colon \btext{a} = (x_1, x_2, \ldots, x_{10}), x_i \in \R \, \forall i = 1, 2, \ldots, 10\}    
\end{equation}
where each $x_{i} \in \R$ represent numerical measurements (parameters) of blood. 
\end{solution*}
\end{comment}

\subsection{An $n$-Dimensional Plane Spanned by a Geometrically Independent Set}\label{ss:An n Dimensional Plane Spanned by a Geometrically Independent Set}
To further understand the geometric structure of Euclidean spaces, it is useful to consider the concept of an $n$-plane in $\R^N$. These are the natural generalizations of lines and planes to higher dimensions, and they provide the foundational geometric setting for defining simplices and simplicial complexes. An $n$-plane  (denoted by $P_n$) in $\R^N$ is the affine subspace of dimension $n$ determined by $n + 1$ geometrically independent points in $\R^N$.

\begin{definition}[$n$-plane spanned by a geometrically independent set]\label{d:n-plane spanned by a geometrically independent set}
Let $A = \{\btext{a}_0, \btext{a}_1, \ldots, \btext{a}_n\}$ be a geometrically independent set of points in $\R^N$. The $n$-plane $P_n$ in $\R^N$ spanned by $A$ is the set
\begin{equation}\label{eq:n-plane spanned by a geometrically independent set}
P_{n} = \left\{ \btext{a} \in \R^N \colon \btext{a} = \sum_{i = 0}^{n} t_i \btext{a}_i \ \text{for some scalars} \ t_i \in \R \ \text{such that} \sum_{i = 0}^{n} t_i = 1 \right\}.
\end{equation}
\end{definition}

\begin{remark}
It is observed that, since the points $\btext{a}_0, \ldots, \btext{a}_n$ are geometrically independent, the scalars $t_i$ are uniquely determined by the point $\btext{a}$. That is, for each $\btext{a} \in P_n$, the coefficients $t_i$ are unique.    
\end{remark}
 
\begin{definition}[$n$-plane passing through a point]\label{d:n-plane passing through a point}
Let $A = \{\btext{a}_0, \btext{a}_1, \ldots, \btext{a}_n\}$ be a geometrically independent set of points in $\R^N$. Then the $n$-plane, denoted by $P_n(\btext{a}_0)$, that passes through the point $\btext{a}_0$ and is parallel to the vectors $\btext{v}_i = \btext{a}_i - \btext{a}_0$ for $i = 1, 2, \ldots, n$, is the set
\begin{equation}\label{eq:n-plane passing through a point}
P_n(\btext{a}_0) = \left\{ \btext{a} \in \R^N \colon \btext{a} = \btext{a}_0 + \sum_{i = 1}^{n} t_i(\btext{a}_i - \btext{a}_0),\ \text{for some scalars } t_i \in \R \right\}.
\end{equation}
\end{definition}

\begin{remark}
In case of finite Euclidean space $\R^m$ instead of the space $\R^{N}$, both definitions  \eqref{d:n-plane spanned by a geometrically independent set} and \eqref{d:n-plane passing through a point} of an $n$-plane whether it is spanned by a geometrically independent set or passes through a point the integers $n$ and $m$ must satisfy the condition $1 \leq n \leq m$. This ensures that the set $A$ of $(n+1)$ points in $\R^m$ can be geometrically independent. Geometric independence requires that the $n$ relative vectors formed from these points be linearly independent, which is only possible when $n \leq m$ in the $m$-dimensional Euclidean space $\R^m$.
\end{remark}

\begin{theorem}[Properties of an $n$-Plane]\label{t:Properties of an n-Plane}
Let an $n$-plane in $\mathbb{R}^N$ be defined as an affine subspace spanned by $n+1$ geometrically independent points. Then the following properties hold:
\begin{enumerate}
    \item \textbf{Dimension:} The $n$-plane has dimension $n$, meaning it is parameterized by $n$ independent variables (the coefficients $t_1, \dots, t_n$).
    \item \textbf{Affine Subspace:} The $n$-plane is an affine subspace of $\mathbb{R}^N$, not necessarily passing through the origin unless $\btext{a}_0 = \btext{0}$.
    \item \textbf{Uniqueness:} For $n+1$ geometrically independent points, there exists a unique $n$-plane passing through them.
    \item \textbf{Linear Independence:} The relative vectors $\btext{v}_i = \btext{a}_i - \btext{a}_0$ for $i = 1, 2, \dots, n$ must be linearly independent for the points to span an $n$-plane.
\end{enumerate}
\end{theorem}

\begin{workingrule}[Determining whether a point lies on an $n$-plane and finding its equation]\label{wr:Determining whether a point lies on an $n$-plane and finding its equation}
Let $A = \{\btext{a}_0, \btext{a}_1, \ldots, \btext{a}_n\}$ be a geometrically independent set in $\R^m$, and let $\btext{w}$ be a point in $\R^m$. To check whether the point $\btext{w}$ lies on the $n$-plane $P_n$ spanned by $A$ and further find the equation of $P_n$, apply the following steps:

\begin{enumerate}
\item \textbf{Choose reference vector:} Let us choose $\btext{a}_0 \in A$ as a fixed reference vector. 

\item \textbf{Compute the set of relative vectors:} Compute the set of relative vectors with respect to $\btext{a}_0$ as follows:
\begin{equation}\label{eq:Set of relative vectors}
V = \{\btext{v}_i \colon \btext{v}_i = \btext{a}_i - \btext{a}_0,\; i = 1, 2, \ldots, n \}.
\end{equation}

\item  \textbf{Compute displacement vector:} Compute displacement vector of given point $\btext{w}$ with respect to $\btext{a}_0$ as follows: 
\begin{equation}\label{eq:Displacement vector}
\btext{u} = \btext{w} - \btext{a}_0   \end{equation}

\item \textbf{Form system of linear equations:} Write displacement vector $\btext{u}$ as a linear combination of the relative vectors $\btext{v}_i \in V$ for all $i = 1, 2, \ldots, n$ that represent the system of linear equations  
\begin{equation}\label{eq:Form system of linear equations}
\btext{u} = s_1 \btext{v}_1 + s_2 \btext{v}_2 + \cdots + s_n \btext{v}_n    
\end{equation}
where $\btext{s} = (s_1, \ldots, s_n)$ is the unknown column vector and $s_{i} \in \R$ for all $i = 1, \ldots, n$. 

\item \textbf{Construct the matrix form of the system of linear equations:} The matrix form of above system of linear equations \eqref{eq:Form system of linear equations} is
\begin{equation}\label{eq:Matrix representation of the system of linear equations}
 \btext{u} = M_V \btext{s} 
\end{equation}
where 
\begin{equation}\label{eq:Matrix M V}
\btext{u} = 
\begin{bmatrix}
u_1 & \cdots & u_m
\end{bmatrix}^{T}, \,
M_V  = 
\begin{bmatrix}
\btext{v}_1 & \btext{v}_2 & \cdots & \btext{v}_n    
\end{bmatrix}_{m \times n} \, \text{and} \, 
\btext{s} = 
\begin{bmatrix}
s_1 & \cdots & s_n
\end{bmatrix}^{T}
\end{equation}

\item \textbf{Solution of system of linear equations:}  Solve the system of equations \eqref{eq:Matrix representation of the system of linear equations} by suitable method.
\item \textbf{Check whether $\btext{w}$ lies on $P_n$:} Apply following:
\begin{enumerate} 
 \item If system of equations \eqref{eq:Matrix representation of the system of linear equations} is consistent and have unique solution $\btext{s} = (s_1, \ldots, s_n)$, then $\btext{w}$ lies on $P_n$.

\item If the system of equations \eqref{eq:Matrix representation of the system of linear equations} is inconsistent, then $\btext{w}$ does not lie on $P_n$.
\end{enumerate}

\item \textbf{Compute affine coefficients:} If $\btext{w}$ lies on $P_n$, then compute the affine coefficients as follows:
\begin{equation}\label{eq:Affine coefficients}
t_i = s_i \text{ for } i = 1, \ldots, n, \quad t_0 = 1 - \sum_{i=1}^n s_i       
\end{equation}

\item \textbf{Verify affine condition:} 
To verify that $\btext{w}$ is an affine combination of points in $A$, check
\begin{equation}\label{eq:Verify affine condition}
\sum_{i=0}^n t_i = 1
\end{equation}

\item \textbf{Find the equation of $P_n$ on which $\btext{w}$ lies:}  
If $\btext{w}$ lies on $P_n$, therefore, $P_n$ is given by
\begin{equation}\label{eq:n-plane}
P_n = \left\{ \btext{a} \in \R^m : \btext{a} = \btext{a}_0 + \sum_{i=1}^{n} s_i \btext{v}_i,\; s_i \in \R \right\}
\end{equation}
\end{enumerate}
\end{workingrule}

\begin{example}
Consider the concept of $n$-planes in $\mathbb{R}^N$:
\begin{enumerate}
    \item A \textbf{0-plane} $P_0$ in $\mathbb{R}^N$ is a single point, i.e., a set containing only one element.
    \item A \textbf{1-plane} $P_1$ in $\mathbb{R}^N$ is a one-dimensional affine subspace, i.e., a straight line parameterized by one parameter.
    \item A \textbf{2-plane} $P_2$ in $\mathbb{R}^N$ is a two-dimensional affine subspace, commonly referred to as a flat surface (for example, a plane embedded in $\mathbb{R}^3$).
    \item A \textbf{3-plane} $P_3$ in $\mathbb{R}^4$ or higher dimensions generalizes the notion of a plane to an affine subspace of dimension three.
\end{enumerate}
\end{example}

\begin{example}[Geometric Independence and Plane Containment in $\mathbb{R}^3$]
Is the set $ A = \{\btext{a}_0, \btext{a}_1, \btext{a}_2\} $ a geometrically independent set of points in $\mathbb{R}^3$, where 
\[
\btext{a}_0 = (1, 0, 0), \quad \btext{a}_1 = (0, 1, 0), \quad \btext{a}_2 = (0, 0, 1),
\]
and does the point 
\[
\btext{w} = \left(\frac{1}{2}, \frac{1}{2}, 0 \right) \in \mathbb{R}^3
\]
lie in the 2-plane 
\[
P_2 = \{ \btext{a} = (1 - s_1 - s_2, s_1, s_2) : s_1, s_2 \in \mathbb{R} \}
\]
spanned by $ A $?
\end{example}

\begin{solution*}
Let $A = \{\btext{a}_0, \btext{a}_1, \btext{a}_2\}$ be a geometrically independent set of points in $\R^3$ given by
\[
\btext{a}_0 = (1, 0, 0), \quad \btext{a}_1 = (0, 1, 0), \quad \btext{a}_2 = (0, 0, 1).
\]

Let $\btext{w} = \left( \frac{1}{2}, \frac{1}{2}, 0 \right)$ be a point in $\R^3$. We apply the following steps to check whether $\btext{w}$ lies in the $2$-plane $P_2$ spanned by $A$, and find the equation of that plane if it does.

\begin{enumerate}
\item \textbf{Choose reference vector:} Let us choose $\btext{a}_0 = (1, 0, 0) \in A$ as a fixed reference vector.

\item \textbf{Compute the set of relative vectors:} The set of relative vectors with respect to $\btext{a}_0$ is $V = \{\btext{v}_1, \btext{v}_2\}$, where
\begin{equation}\label{eg:Set of relative vectors}
\btext{v}_1 = \btext{a}_1 - \btext{a}_0 = (-1, 1, 0), \quad \btext{v}_2 = \btext{a}_2 - \btext{a}_0 = (-1, 0, 1)
\end{equation}

\item  \textbf{Compute displacement vector:} The displacement vector of given point $\btext{w}$ with respect to $\btext{a}_0$ is  
\begin{equation}\label{eg:Displacement vector}
\btext{u} = \btext{w} - \btext{a}_0 = \left( \frac{1}{2}, \frac{1}{2}, 0 \right) - (1, 0, 0) = \left( -\frac{1}{2}, \frac{1}{2}, 0 \right)  \end{equation}

\item \textbf{Form the system of linear equations:}
The displacement vector $\btext{u}$ is a linear combination of the vectors $\btext{v}_1, \btext{v}_2$ of $V$, so the system of linear equations is
\begin{equation}\label{eg:Form system of linear equations}
\btext{u} = s_1 \btext{v}_1 + s_2 \btext{v}_2   
\end{equation}

\item \textbf{Construct the matrix form of the system of linear equations:} The matrix form of above system of linear equations \eqref{eg:Form system of linear equations} is
\begin{equation}\label{eg:Matrix representation of the system of linear equations}
 \btext{u} = M_V \btext{s} 
\end{equation}
where 
\begin{equation}\label{eg:Matrix M V}
\btext{u} = 
\begin{bmatrix}
-\tfrac{1}{2} \\ \tfrac{1}{2} \\ 0    
\end{bmatrix}, \,
M_V  = 
\begin{bmatrix}
-1 & -1 \\ 1 & 0 \\ 0 & 1   
\end{bmatrix} \, \text{and} \,
\btext{s} = 
\begin{bmatrix}
s_1 \\ s_2     
\end{bmatrix}
\end{equation}

\item \textbf{Solution of system of linear equations:}  On solving the system of equations \eqref{eg:Matrix representation of the system of linear equations} by suitable method, we get system is consistent with a unique solution 
\begin{equation}\label{eq:Solution of system}
\btext{s} = (s_1, s_2) = (\tfrac{1}{2}, 0)
\end{equation}

\item \textbf{Check whether $\btext{w}$ lies on $P_n$:} Since the system of linear equations \eqref{eg:Matrix representation of the system of linear equations} is consistent with a unique solution therefore, the given point $\btext{w}$ lies on $P_2$ spanned by $A$.   

\item \textbf{Compute affine coefficients:} Since $\btext{w}$ lies on $P_2$ spanned by $A$, therefore, the affine coefficients are
\begin{equation}\label{eg:Affine coefficients}
t_1 = s_1 = \frac{1}{2}, \quad t_2 = s_2 = 0, \quad t_0 = 1 - (t_1 + t_2) = \frac{1}{2}      
\end{equation}

\item \textbf{Verify affine condition:} From the value of the affine coefficients $t_0, t_1, t_2$ we can see that
\begin{equation}\label{eg:Verify affine condition}
t_0 + t_1 + t_2 = \frac{1}{2} + \frac{1}{2} + 0 = 1 
\end{equation}
therefore, $t_0 = \tfrac{1}{2}, t_1= \tfrac{1}{2}, t_2 = 0$ are affine coefficients.  

\item \textbf{Find the equation of $P_n$ in which $\btext{w}$ lies:}  
Since $\btext{w}$ lies on $P_2$, therefore, $P_2$ is given by
\begin{align}\label{eg:n-plane}
P_2 & = \left\{ \btext{a} \in \R^3 \colon \btext{a} = \btext{a}_0 + s_1 \btext{v}_1 + s_2 \btext{v}_2,\; s_1, s_2 \in \R \right\} \notag \\
& =  \left\{ \btext{a} \in \R^3 \colon \btext{a} = (1, 0, 0) + s_1(-1, 1, 0) + s_2(-1, 0, 1) \right\} \notag \\
& = \left\{ \btext{a} \in \R^3 \colon \btext{a} = (1 - s_1 - s_2,\; s_1,\; s_2), \ s_1, s_2 \in \R \right\} 
\end{align}

This is the explicit description of $P_2$ as a parametric affine subspace of $\R^3$.
\end{enumerate}
\end{solution*}

\begin{theorem}[Extending a Geometrically Independent Set]\label{t:Extending a Geometrically Independent Set}
If $A = \{\btext{a}_0, \ldots, \btext{a}_n\}$ is geometrically independent set of points in $\R^m$ and $\btext{w} \in \R^m$ is arbitrary point lies outside the $n$-plane $P_n$  which is spanned by the points of $A$, then set $B = \{\btext{w}, \btext{a}_0, \ldots, \btext{a}_n\}$ is also geometrically independent. 
\end{theorem}

\begin{proof}
As given $A = \{\btext{a}_0, \ldots, \btext{a}_n\}$ is geometrically independent in $\R^N$ and a arbitrary point $\btext{w} \in \R^N$ such that $\btext{w} \notin P_{n}$, where $P_n$ the $n$-plane spanned by set $A$. Now our aim to show that $B = \{\btext{w}, \btext{a}_0, \ldots, \btext{a}_n\}$ is also geometrically independent.

Since $A$ is geometrically independent, then by Theorem \eqref{t:Necessary and Sufficient Condition for GI of Set} the vector $\btext{v}_{i} = \btext{a}_{i} - \btext{a}_{0}$ (for $i = 1, \ldots, n$) are linear independent i.e. there exists scalars $t_{i}$ such 
\begin{equation*}\label{eq:Generation of New Geometrical Independent Set form Old}
\sum_{i = 1}^{n} t_{i} \btext{v}_{i} = \btext{0} \Rightarrow t_0 = \ldots = t_{n}  
\end{equation*}

Now, we consider the set $B = \{\btext{w}, \btext{a}_0, \ldots, \btext{a}_n\}$. We check whether the vectors
\[\btext{a}_1 - \btext{a}_0, \btext{a}_2 - \btext{a}_0, \dots, \btext{a}_n - \btext{a}_0, \btext{w} - \btext{a}_0
\]
are linearly independent. Suppose there exist scalars  $t_1, t_2, \dots, t_n, s$ such that
\begin{equation}\label{eq:Generation of New Geometrical Independent Set form Old1}
t_1 (\btext{a}_1 - \btext{a}_0) + t_2 (\btext{a}_2 - \btext{a}_0) + \dots + t_n (\btext{a}_n - \btext{a}_0) + s (\btext{w} - \btext{a}_0) = \btext{0}    
\end{equation}

Rewriting, we obtain
\begin{equation}\label{eq:Generation of New Geometrical Independent Set form Old2}
\sum_{i=1}^{n} t_i (\btext{a}_i - \btext{a}_0) = -s (\btext{w} - \btext{a}_0)    
\end{equation}

The left-hand side of Equation \eqref{eq:Generation of New Geometrical Independent Set form Old2} is a linear combination of the vectors $ \btext{a}_1 - \btext{a}_0, \dots, \btext{a}_n - \btext{a}_0 $ , which lie in the $ n $ -plane $ P_n $ . The right-hand side of Equation \eqref{eq:Generation of New Geometrical Independent Set form Old2} is a scalar multiple of $ \btext{w} - \btext{a}_0 $ , which does not belong to $ P_n $  since $ \btext{w} \notin P_n $ .

Since $ P_n $  is an $ n $ -dimensional affine plane, the only way the above equation can hold is if
\[t_1 = t_2 = \dots = t_n = s = 0\]

This proves that the vectors $ \btext{a}_1 - \btext{a}_0, \btext{a}_2 - \btext{a}_0, \dots, \btext{a}_n - \btext{a}_0, \btext{w} - \btext{a}_0 $  are linearly independent. Hence, the set $ B $  is geometrically independent.
\end{proof}

\begin{workingrule}[Extending a Geometrically Independent Set]\label{wr:Extending a Geometrically Independent Set}
Let $A = \{\btext{a}_0, \btext{a}_1, \ldots, \btext{a}_n\}$ be a geometrically independent set of points in $\R^m$, and let $\btext{w} \in \R^m$ be a point not lying in the $n$-plane $P_n$ spanned by $A$. To check whether the extended set  
\[
B = \{\btext{w}, \btext{a}_0, \ldots, \btext{a}_n\}
\]  
is geometrically independent, follow these steps:

\begin{enumerate}
\item \textbf{Choose a reference point:} Choose $\btext{a}_0 \in A$ as a fixed reference point.

\item \textbf{Compute the set of relative vectors:} Compute the set of relative vectors with respect to $\btext{a}_0$ as follows:
\begin{equation}\label{eq:Set of relative vectors}
V = \{\btext{v}_i \colon \btext{v}_i = \btext{a}_i - \btext{a}_0,\; i = 1, 2, \ldots, n \}.
\end{equation}

\item \textbf{Compute the relative vector for $\btext{w}$:}  
Compute the relative vector for $\btext{w}$ with respect to $\btext{a}_0$ as follows:
\begin{equation}\label{eq:Relative vector for w}
\btext{v}_{\btext{w}} = \btext{w} - \btext{a}_0.    
\end{equation}
 
\item \textbf{Construct the matrices:} 
\begin{enumerate}
\item Construct the matrix $M_V$ whose columns are the vectors $\btext{v}_1, \btext{v}_2, \ldots, \btext{v}_n$ of $V$:
\begin{equation}\label{eq:Matrix MV}
M_V = \begin{bmatrix}
\btext{v}_1 & \btext{v}_2 & \ldots & \btext{v}_n
\end{bmatrix}_{m \times n}
\end{equation}
\item Construct the matrix $M_{V_{\btext{w}}}$ whose columns are $\btext{v}_1, \btext{v}_2, \ldots, \btext{v}_n, \btext{v}_{\btext{w}}$:
\begin{equation}\label{eq:Matrix MVw}
M_{V_{\btext{w}}} = \begin{bmatrix}
\btext{v}_1 & \btext{v}_2 & \ldots & \btext{v}_n & \btext{v}_{\btext{w}}
\end{bmatrix}_{m \times (n + 1)}
\end{equation}
\end{enumerate}  
\item \textbf{Compute the ranks of the matrices:} Compute the ranks of $M_V$ and $M_{V_{\btext{w}}}$ using any suitable method.
\item \textbf{Check geometric independence of the set $B$:}  
If 
\[
\operatorname{rank}(M_{V_{\btext{w}}}) = \operatorname{rank}(M_V) + 1,
\]
then the set $B$ is geometrically independent. 
\end{enumerate}
\end{workingrule}

\begin{example}[Generation of a New Geometrically Independent Set from an Old Set]\label{eg:generation of New Geometrical Independent Set form Old}
Let $ A = \{\btext{a}_0 = (2, 3, 1), \btext{a}_1 = (3, 5, 2),  \btext{a}_2 = (4, 4, 3)\} $ be a geometrically independent set of points in $\mathbb{R}^3$. If the point $\btext{w} = (5, 6, 7)$ lies outside the 2-plane $P_2$ spanned by $A$, is the new set 
$B = \{\btext{w}, \btext{a}_0, \btext{a}_1, \btext{a}_2 \}$ also geometrically independent?
\end{example}

\begin{solution*}
Let $A = \{\btext{a}_0 = (2, 3, 1), \btext{a}_1 = (3, 5, 2),  \btext{a}_2 = (4, 4, 3)\}$ be a geometrically independent set of points in $\R^3$. Let $\btext{w} = (5, 6, 7)$ be a point that lies outside the 2-plane $P_2$ spanned by $A$. Apply following steps to check whether the new set $B = \{\btext{w}, \btext{a}_{0}, \btext{a}_{1}, \btext{a}_{2}\}$ is also geometrically independent.
\begin{enumerate}
\item \textbf{Choose a reference point:}  
We choose $\btext{a}_0 = (2, 3, 1)$ as the fixed reference point.

\item \textbf{Compute the set of relative vectors:} The set of relative vectors with respect to $\btext{a}_0$ is $V = \{\btext{v}_1, \btext{v}_2\}$, where   
\[\btext{v}_1 = \btext{a}_1 - \btext{a}_0 = (1, 2, 1), \, \btext{v}_2 = \btext{a}_2 - \btext{a}_0 = (2, 1, 2) \]

\item \textbf{Compute the relative vector for $\btext{w}$:}  The relative vector for $\btext{w}$ with respect to $\btext{a}_0$ is
\[
\btext{v}_{\btext{w}} = \btext{w} - \btext{a}_0 = (3, 3, 6)
\]

\item \textbf{Construct the matrices:} 
\begin{enumerate}
\item The matrix $M_V$ whose columns are the vectors $\btext{v}_1, \btext{v}_2$ of $V$ is
\[M_V = 
\begin{bmatrix}
1 & 2 \\ 2 & 1 \\ 1 & 2
\end{bmatrix}\]

\item The matrix $M_{V_{\btext{w}}}$ whose columns are the vectors $\btext{v}_1, \btext{v}_2, \btext{v}_{\btext{w}}$ is
\[M_{V_{\btext{w}}} = 
\begin{bmatrix}
1 & 2 & 3 \\ 2 & 1 & 3 \\ 1 & 2 & 6
\end{bmatrix} \]
\end{enumerate}
\item \textbf{Compute the ranks of the matrices:} On applying elementary row operations for Reduced Row Echelon Form (RRFE), we get 
\begin{gather*}
M_V \xrightarrow{\text{RREF}} 
\begin{bmatrix}
1 & 2 \\
0 & -3 \\
0 & 0
\end{bmatrix}
\Rightarrow \operatorname{rank}(M_V) = 2 \\
M_{V_{\btext{w}}} \xrightarrow{\text{RREF}} 
\begin{bmatrix}
1 & 2 & 3 \\
0 & -3 & -3 \\
0 & 0 & 3
\end{bmatrix}
\Rightarrow \operatorname{rank}(M_{V_{\btext{w}}}) = 3
\end{gather*}

\item \textbf{Check geometric independence:} Since
\[
\operatorname{rank}(M_{V_{\btext{w}}}) = \operatorname{rank}(M_V) + 1 = 2 + 1 = 3,
\]
therefore, $B$ is geometrically independent.
\end{enumerate}
\end{solution*}

\subsection{Affine Transformation on $\R^{N}$}\label{ss:Affine Transformation on RN}
Affine transformations are a fundamental class of mappings in geometry and linear algebra that generalize linear transformations by including translations. Unlike purely linear maps, which always fix the origin, affine transformations can shift, rotate, scale, or shear objects while preserving the structure of straight lines and parallelism. More formally, an affine transformation on $\R^m$ is a function that maps points to points in a way that preserves affine combinations—that is, combinations of the form $t\btext{x} + (1 - t)\btext{y}$ for $t \in \R$. These transformations play a crucial role in various areas such as computer graphics, robotics, topology, and optimization, as they capture geometric transformations of figures without altering their fundamental shape or relative configuration.

\begin{definition}[Matrix as a Linear Transformation]\label{d:Matrix as a Linear Transformation}
Let $A \in \R^{m \times m}$ be a real square matrix. Then $A$ defines a linear transformation 
\[
T_A \colon \R^m \to \R^m, \, \btext{x} \mapsto A\btext{x},
\]
where the output $A\btext{x}$ is obtained by standard matrix-vector multiplication. The transformation $T_A$ is called a linear transformation because it satisfies the following properties for all $\btext{x}, \btext{y} \in \R^m$ and scalars $\alpha, \beta \in \R$:
\begin{equation}\label{eq:Linear tranformation properties}
T_A(\alpha \btext{x} + \beta \btext{y}) = \alpha T_A(\btext{x}) + \beta T_A(\btext{y}),\ \text{and} \
T_A(\btext{0}) = \btext{0}
\end{equation}
\end{definition}

The above Definition \eqref{d:Matrix as a Linear Transformation} tells us that any real $m \times m$ matrix $A$ defines a function $T_A \colon \R^m \to \R^m$ via matrix multiplication. This function is called a linear transformation because it preserves vector addition and scalar multiplication. That is, the transformation $T_A$ respects the algebraic structure of $\R^m$. Linear transformations map the origin to the origin and preserve the straightness of lines, although they may change lengths and angles. Examples include rotations, scalings, reflections, and shear transformations. To illustrate this, we now consider a concrete example in $\R^2$.

\begin{example}[Linear Transformation Induced by a Matrix]
Let $ A $ be the $ 2 \times 2 $ matrix
\[
A = \begin{bmatrix}
2 & 1 \\
0 & 3
\end{bmatrix}.
\]
Consider the function $ T_A : \mathbb{R}^2 \to \mathbb{R}^2 $ defined by $ T_A(\mathbf{x}) = A\mathbf{x} $.

Is $ T_A $ a linear transformation, and what are its algebraic and geometric effects on vectors in $\mathbb{R}^2$?
\end{example}

\begin{solution*}
Let $\mathbf{x} = \begin{bmatrix} x_1 \\ x_2 \end{bmatrix} \in \mathbb{R}^2$. Then the image of $\mathbf{x}$ under $T_A$ is:
\[
T_A(\mathbf{x}) =
\begin{bmatrix}
2 & 1 \\
0 & 3
\end{bmatrix}
\begin{bmatrix}
x_1 \\
x_2
\end{bmatrix}
=
\begin{bmatrix}
2x_1 + x_2 \\
3x_2
\end{bmatrix}.
\]
\begin{itemize}
\item \textbf{Linearity check:}  Let $\mathbf{y} = \begin{bmatrix} y_1 \\ y_2 \end{bmatrix} \in \mathbb{R}^2$ and let $\alpha, \beta \in \mathbb{R}$. Then:
\[
T_A(\alpha \mathbf{x} + \beta \mathbf{y}) =
A(\alpha \mathbf{x} + \beta \mathbf{y}) =
\alpha A\mathbf{x} + \beta A\mathbf{y} =
\alpha T_A(\mathbf{x}) + \beta T_A(\mathbf{y}),
\]
confirming that $T_A$ preserves linear combinations.

\item \textbf{Zero vector check:}  Let $\mathbf{0} = \begin{bmatrix} 0 \\ 0 \end{bmatrix}$. Then:
\[
T_A(\mathbf{0}) = A\mathbf{0} =
\begin{bmatrix}
2 & 1 \\
0 & 3
\end{bmatrix}
\begin{bmatrix}
0 \\
0
\end{bmatrix}
=
\begin{bmatrix}
0 \\
0
\end{bmatrix} = \mathbf{0}.
\]
So, $T_A$ maps the zero vector to itself, as required by linearity.

\item \textbf{Geometric interpretation:}  
The transformation $T_A$ scales the $y$-coordinate by a factor of 3, and applies a shear in the $x$-direction by adding $x_2$ to $2x_1$. The origin remains fixed, and straight lines and parallelism are preserved, as expected from a linear transformation.
\end{itemize}
\end{solution*}

\begin{definition}[Affine Transformation on $\R^{N}$]\label{d:Affine Transformation on RN}
A function  $T \colon \R^N \to \R^N$  define by $T(\btext{a}) = A\btext{a} + \btext{b}$ is called affine transformation on $\R^N$, where $A$ is $N \times N$ is a real matrix (representing a linear transformation), $\btext{b}$ is a fixed a vector (representing a translation vector).
\end{definition}

\begin{theorem}[Properties of Affine Transformation]\label{t:Properties of Affine Transformation}
Affine transformation preserves collinearity, it preserves ratio i.e. it maintains the ratios of distances along a straight line. Unlike linear transformations, affine transformations do not necessarily fix the origin unless a fix vector $\btext{b}$.
\end{theorem}

\begin{example}[Affine Transformation on \(\mathbb{R}^2\)]
Is the function \( T: \mathbb{R}^2 \to \mathbb{R}^2 \) defined by
\[
T(x, y) = 
\begin{bmatrix} 2 & 1 \\ 1 & 3 \end{bmatrix} 
\begin{bmatrix} x \\ y \end{bmatrix} 
+ \begin{bmatrix} 1 \\ -2 \end{bmatrix}
\]
an affine transformation on \(\mathbb{R}^2\)? 

How does this transformation, consisting of a linear transformation by the matrix 
\[
A = \begin{bmatrix} 2 & 1 \\ 1 & 3 \end{bmatrix},
\]
which scales, rotates, and shears the space, combined with a translation by the vector \(\btext{b} = (1, -2)\), affect specific points such as the origin and the unit vectors along the axes?
\end{example}

\begin{solution*}
The function \(T\) is indeed an affine transformation on \(\mathbb{R}^2\) as it can be expressed as a linear transformation followed by a translation. For example, the points transform as follows:
\begin{enumerate}
    \item The origin \((0,0)\) is mapped to \((1,-2)\):
    \[
    T(0,0) = A(0,0) + (1,-2) = (1,-2).
    \]
    \item The unit vector \((1,0)\) along the x-axis is mapped to \((3,-1)\), showing scaling and shearing:
    \[
    T(1,0) = A(1,0) + (1,-2) = 
    \begin{bmatrix} 2 \\ 1 \end{bmatrix} + \begin{bmatrix} 1 \\ -2 \end{bmatrix} = \begin{bmatrix} 3 \\ -1 \end{bmatrix}.
    \]
    \item The unit vector \((0,1)\) along the y-axis is mapped to \((2,1)\), indicating distortion and rotation:
    \[
    T(0,1) = A(0,1) + (1,-2) = 
    \begin{bmatrix} 1 \\ 3 \end{bmatrix} + \begin{bmatrix} 1 \\ -2 \end{bmatrix} = \begin{bmatrix} 2 \\ 1 \end{bmatrix}.
    \]
\end{enumerate}
\end{solution*}

\begin{example}[Examples of Affine Transformations]\label{eg:Examples of Affine Transformations}
There are following example of affine transformation $T \colon \R^N \to \R^N$ defined by 
\begin{enumerate}
\item \textbf{Translation:} $T(\btext{a}) = \btext{a} + \btext{b}$
\item \textbf{Scaling:}
$T(\btext{a}) = c\btext{a}$, where $c$ is scalar. 
\item \textbf{Rotation with Translation:} $T(\btext{a}) = R\btext{a} + \btext{b}$, $R$ is rotation matrix. 
\item \textbf{Shear Transformation:} A transformation that skews the coordinate system while preserving parallelism.
\end{enumerate}
\end{example}

\begin{theorem}[Preservation of Geometrically Independency during Affine Transformation]\label{t:Preservation of Geometrically Independency during Affine Transformation}
Let $A = \{\btext{a}_{0}, \ldots, \btext{a}_n\}$ is geometrically independent set in $\R^N$ and $T$
is any affine transformation on $A$, then $T(A)$ is also  geometrically independent. That is, the Affine transformation preserves geometrically independency of set.
\end{theorem}
\begin{proof}

By definition a set of points $A = \{\btext{a}_0, \dots, \btext{a}_n\}$ in $\R^N$ is said to be geometrically independent if the vectors  
\[
\btext{a}_1 - \btext{a}_0, \btext{a}_2 - \btext{a}_0, \dots, \btext{a}_n - \btext{a}_0
\]
are linearly independent in $\R^N$. Therefore by Definition of linear independency of vectors
\begin{equation}\label{eq:Preservation of Geometrically Independency during Affine Transformation1}
\sum_{i=1}^{n} t_i (\btext{a}_i - \btext{a}_0) = 0 \Rightarrow t_i =  0, \forall i    
\end{equation}
An affine transformation $T \colon \R^N \to \R^N$ is given by
\[
T(\btext{a}) = A' \btext{a} + \btext{b},
\]
where $A'$ is an $N \times N$ invertible matrix (a linear transformation), and $\btext{b}$ is a translation vector. 
Applying $T$ to the set $A$, we obtain the set of transformed points
\[T(A) = \{T(\btext{a}_i)\} = \{A' \btext{a}_i + \btext{b}\colon \btext{a}_{i} \in A \, \forall i\}\]
Now our aim is to show that set $T(A)$ is geometrically independent for this we prove that it is linearly independent. 

The transformed difference vectors are:
\[T(\btext{a}_i) - T(\btext{a}_0) = (A' \btext{a}_i + \btext{b}) - (A' \btext{a}_0 + \btext{b})
= A' (\btext{a}_i - \btext{a}_0)\]

Now, if the original vectors $\btext{a}_1 - \btext{a}_0, \dots, \btext{a}_n - \btext{a}_0$ are linearly independent, then the equation
\[
\sum_{i=1}^{n} t_i A'(\btext{a}_i - \btext{a}_0) = 0
\]
implies
\[
A' \left( \sum_{i=1}^{n} t_i (\btext{a}_i - \btext{a}_0) \right) = 0.
\]

Since $A'$ is invertible, the only solution is:
\[
\sum_{i=1}^{n} t_i (\btext{a}_i - \btext{a}_0) = 0,
\]
then by Equation \eqref{eq:Preservation of Geometrically Independency during Affine Transformation1} implies $t_i = 0,  \, \forall i$. Therefore, vectors $T(\btext{a}_1) - T(\btext{a}_0), \ldots, T(\btext{a}_n) - T(\btext{a}_0)$ are linear independent.  Thus, the transformed set $T(A) = \{T(\btext{a}_0), T(\btext{a}_1), \dots, T(\btext{a}_n)\}$ remains geometrically independent.

Since an affine transformation is a composition of a linear transformation and a translation, and we have shown that it preserves linear independence of difference vectors, it follows that an affine transformation preserves the geometrical independence of a set.
\end{proof}

\begin{theorem}[Preservation of Plane during Affine Transformation]\label{t:Preservation of Plane during Affine Transformation}
Let $P_n$ is the $n$-plane spanned by the geometrically independent set $A = \{\btext{a}_0, \ldots, \btext{a}_n\}$. Let $T \colon \R^N \to \R^N$ is affine transformation $T$ carries the plane $P_n$ spanned by $\btext{a}_0, \ldots, \btext{a}_n$ onto the plane spanned by $T(\btext{a}_0), \ldots, T(\btext{a}_n)$.
\end{theorem}

\begin{proof}
Let $ P_n $  be the $ n $ -plane spanned by the geometrically independent set $ A = \{\btext{a}_0, \btext{a}_1, \dots, \btext{a}_n\} $  in $ \R^N $ .  
By definition, every point $ \btext{a} \in P_n $  can be written as an affine combination of the points in $ A $  
\begin{equation}\label{eq:Preservation of Plane during Affine Transformation1}
\btext{a} = \btext{a}_0 + \sum_{i=1}^{n} t_i (\btext{a}_i - \btext{a}_0), t_{i} \in \R, \forall i    
\end{equation}

Since $ T\colon \R^N \to \R^N $  is an affine transformation, it has the form:  
\begin{equation}\label{eq:Preservation of Plane during Affine Transformation2}
T(\btext{a}) = A' \btext{a} + \btext{b}    
\end{equation}
where $ A' $  is an $ N \times N $  invertible matrix (a linear transformation), and $ \btext{b} $  is a translation vector.

By Equation \eqref{eq:Preservation of Plane during Affine Transformation1}, applying $ T $  to both sides of the affine combination, we get 
\[T(\btext{a}) = T \left( \btext{a}_0 + \sum_{i=1}^{n} t_i (\btext{a}_i - \btext{a}_0) \right)\]

Using the property of affine transformations
\[T(\btext{a}) = T(\btext{a}_0) + \sum_{i=1}^{n} t_i (T(\btext{a}_i) - T(\btext{a}_0))\]

This equation shows that every point $ T(\btext{a}) $  is an affine combination of the transformed points $ T(\btext{a}_0), T(\btext{a}_1), \dots, T(\btext{a}_n) $, which means the set $ T(P_n) $  is precisely the plane spanned by the transformed points. Since $ T(P_n) $  consists of all affine combinations of $ T(\btext{a}_0), \dots, T(\btext{a}_n)$, we conclude that $ T $  maps the plane $ P_n $  onto the plane spanned by $ T(\btext{a}_0), \dots, T(\btext{a}_n)$.  
\end{proof}

\section{Simplices}\label{s:Simplices}
Having established the essential background in Euclidean geometry and affine transformations, we are now prepared to introduce the fundamental building blocks of simplicial complexes—simplices. The word simplices is used as the plural of simplex.  A simplex is the simplest possible polytope in any given dimension: a point in dimension 0, a line segment in dimension 1, a triangle in dimension 2, a tetrahedron in dimension 3, and so on. Simplices generalize the notion of ``flat shapes" to arbitrary dimensions and play a central role in the study of both geometric and algebraic topology. 

There are two approaches to studying simplex - geometric and abstract (axiom based). In this section we will first study its geometrical approach and later consider the abstract approach. In what follows, we define simplex formally, illustrate them with examples, and explore their basic properties essential for constructing simplicial complexes.

\begin{definition}[$n$-simplex]\label{d:n-simplex}
Let $A = \{\btext{a}_0, \ldots, \btext{a}_n\}$ is geometrically independent set in $\R^N$. The $n$-simplex (denoted by $\sigma_n$) spanned by $\btext{a}_0, \ldots, \btext{a}_n$ is the set of all points $ \btext{x} \in \R^N$ such that 
\[\btext{x} = \sum_{i = 0}^{n}t_{i}\btext{a}_{i}, \, \text{where} \, \sum_{i = 0}^{n}t_{i} = 1 \, \text{and} \, t_{i} \geq 0, \, \forall i\]
Symbolically, 
\begin{equation}\label{eq:n-simplex}
 \sigma_n = \Big \{\btext{x} \in \R^N \colon \btext{x} = \sum_{i = 0}^{n}t_{i}\btext{a}_{i}, \, \text{where} \, \sum_{i = 0}^{n}t_{i} = 1 \, \text{and} \, t_{i} \geq 0, \, \forall i \Big\} 
\end{equation}
\end{definition}

Now we are going to define some important terms related to simplex which will help us understand some other properties of simplex.

\begin{definition}[Set of Vertices of Simplex]\label{d:Set of Vertices of Simplex}
If simplex $\sigma_n$ is spanned by the geometrically independent set $A = \{\btext{a}_0, \btext{a}_1, \ldots, \btext{a}_n\}$ in $\R^N$, then set $A$ is called set of vertices of the simplex $\sigma_n$.     
\end{definition}

\begin{definition}[Dimension of Simplex]\label{d:Dimension of Simplex}
The dimension of the simplex $\sigma_n$ spanned by $A = \{\btext{a}_0, \ldots, \btext{a}_n\}$ is defined by $\dim (\sigma_n) = n$, where $n$ is one less than the number of geometrically independent points in $A$.
\end{definition}

\begin{definition}[Barycentric Coordinate of Point in Simplex]\label{d:Barycentric Coordinate of Point in Simplex}
The simplex $\sigma_n$ spanned by the geometrically independent set $A = \{\btext{a}_0, \btext{a}_1, \ldots, \btext{a}_n\}$ in $\R^N$ is 
\[ \sigma_n = \Big \{\btext{x} \in \R^N \colon \btext{x} = \sum_{i = 0}^{n}t_{i}\btext{a}_{i}, \, \text{where} \, \sum_{i = 0}^{n}t_{i} = 1 \, \text{and} \, t_{i} \geq 0, \, \forall i \Big\} \]
The real numbers $t_i$ that is uniquely determined by $\btext{x}$ is called \textbf{barycentric coordinate} of the point $\btext{x}$ of the simplex $\sigma_n$ w.r.t. vertices $\btext{a}_0, \ldots, \btext{a}_n$. 
\end{definition}

After understanding the definition of an $n$-simplex, the set of its vertices, and the concept of barycentric coordinates, it is natural to examine the topological structure of a simplex. Beyond its geometric description, a simplex can be viewed as a topological space with specific properties that make it a fundamental object in topology. In particular, we are interested in how a simplex behaves as a subset of Euclidean space, its topological dimension, its boundary and interior, and its role in constructing more complex topological spaces. This perspective is essential for developing simplicial complexes and understanding their use in algebraic topology.

\begin{theorem}[Simplex as Topological Space]\label{t:Simplex as Topological Space}
If $\sigma_n$ is a $n$-simplex spanned by the geometrically independent set  $A = \{\btext{x}_0, \ldots, \btext{x}_n\}$ in $\R^N$, then $\sigma_n$ is topological space with subspace topology.   
\end{theorem}

\begin{proof}
If $\sigma_n$ is a $n$-simplex spanned by the geometrically independent set  $A = \{\btext{x}_0, \ldots, \btext{x}_n\}$ in $\R^N$, then by Definition
\[\sigma_n = \Big \{\btext{x} \in \R^N \colon \btext{x} = \sum_{i = 0}^{n}t_{i}\btext{a}_{i}, \, \text{where} \, \sum_{i = 0}^{n}t_{i} = 1 \, \text{and} \, t_{i} \geq 0, \, \forall i \Big\} \]
Since $ \sigma_n $ is a subset of $ \R^N $, it inherits the subspace topology from the standard topology on $ \R^N $. The subspace topology on $ \sigma_n $ is defined as follows
\[\mathcal{T}_{\sigma_n} = \{ U \cap \sigma_n \colon U \, \text{is open in} \, \R^N \}\]
This means that a set $ V \subset \sigma_n $ is open in $ \sigma_n $ if and only if there exists an open set $ U $ in $ \R^N $ such that $V = U \cap \sigma_n$. To show that $ (\sigma_n, \mathcal{T}_{\sigma_n}) $ is a topological space, we check the axioms of a topology:

\begin{enumerate}
    \item \textbf{Empty Set and Full Space}:
    \begin{enumerate}
        \item The empty set $ \emptyset $ is open because $ \emptyset = \emptyset \cap \sigma_n $.
        \item The entire simplex $ \sigma_n $ is open because $ \sigma_n = \mathbb{R}^N \cap \sigma_n $, where $ \R^N $ is open in itself.
    \end{enumerate}
    \item \textbf{Arbitrary Unions:} If $ \{V_\alpha\} $ is a collection of open sets in $ \sigma_n $, then each $ V_\alpha = U_\alpha \cap \sigma_n $ for some open set $ U_\alpha $ in $ \R^N $. Their union is
        \[ \bigcup_{\alpha} V_\alpha = \bigcup_{\alpha} (U_\alpha \cap \sigma_n) = \left( \bigcup_{\alpha} U_\alpha \right) \cap \sigma_n        \]
        Since $ \bigcup_{\alpha} U_\alpha $ is open in $ \mathbb{R}^N $, the resulting set is open in $ \sigma_n $.
    \item \textbf{Finite Intersections:} If $ V_1, V_2 $ are open in $ \sigma_n $, then $ V_1 = U_1 \cap \sigma_n $ and $ V_2 = U_2 \cap \sigma_n $ for some open sets $ U_1, U_2 $ in $ \mathbb{R}^N $. Their intersection is
        \[V_1 \cap V_2 = (U_1 \cap \sigma_n) \cap (U_2 \cap \sigma_n) = (U_1 \cap U_2) \cap \sigma_n        \]
        Since $ U_1 \cap U_2 $ is open in $ \mathbb{R}^N $, the resulting set is open in $ \sigma_n $.
\end{enumerate}
Since all three conditions hold, $ \mathcal{T}_{\sigma_n} $ is a valid topology, meaning $ \sigma_n $ is a topological space with the \textbf{subspace topology}.
\end{proof}

%\subsection{Different Types of Simplex}\label{ss:Different Types of Simplex}

To gain a deeper understanding of the structure of a simplex, it is important to explore how its subsets are organized and interpreted within topology and geometry. Concepts such as faces, boundaries, interiors, and closures of a simplex provide crucial insights into how simplices connect and combine to form more complex geometric structures like simplicial complexes. These notions not only describe the geometric composition of a simplex but also play a central role in defining chains, boundaries, and homology groups in algebraic topology. We now formally define and illustrate each of these concepts in order to establish a solid foundation for the study of simplicial complexes.

\begin{definition}[Face of Simplex]\label{d:Face of Simplex}
Let $n$-simplex $\sigma_n$ is spanned by the geometrically independent set $A = \{\btext{a}_0, \btext{a}_1, \ldots, \btext{a}_n\}$ in $\R^N$. Any $k$-simplex $\sigma_k$ ($0 \leq k < n$) spanned by the subset $B$ of $A$ is called face of $\sigma_n$. Symbolically, 
\begin{equation}\label{eq:Face of Simplex}
\sigma_k = \Big \{\btext{x} \in \R^N \colon \btext{x} = \sum_{j = 0}^{k}t_{j}\btext{a}_{j}, \, \text{where} \, \sum_{j = 0}^{k}t_{j} = 1 \, \text{and} \, t_{j} \geq 0, \, \forall j \Big\} 
\end{equation}
\end{definition}
In particular, the face $\sigma_k$ of $\sigma_n$ spanned by the set $B = \{\btext{a}_1, \ldots, \btext{a}_n\} \subset A$ is called \textbf{face opposite to $\btext{a}_0$}. The faces different from $\sigma_n$ itself are called \textbf{proper faces} of $\sigma_n$.

\begin{definition}[Boundary of Simplex]\label{d:Boundary of Simplex}
The boundary of $\sigma_n$, denoted by $\operatorname{Bd}(\sigma_n)$ (or $\operatorname{Bd} \sigma_{n}$), is the union of all its $(n-1)$-faces. Symbolically, 
\begin{equation}\label{eq:Boundary of Simplex}
\operatorname{Bd}(\sigma_n)= \bigcup_{i = 0}^{n}\sigma_{n - 1}^{(i)}
\end{equation}   
where $\sigma_{n - 1}^{(i)}$ is the $(n-1)$-face obtained by omitting the vertex $\btext{a}_i$ from $A$.
\end{definition}

\begin{definition}[Interior of Simplex]\label{d:Interior of Simplex}
The interior of $\sigma_n$, denoted by $\operatorname{Int}(\sigma_n)$, consists of all points $\mathbf{x} \in \sigma_n$ whose barycentric coordinates are strictly positive. Symbolically, 
\begin{equation}\label{eq:Interior of Simplex1}
\operatorname{Int}(\sigma_n) = \Big \{\btext{x} \in \sigma_n \colon \btext{x} = \sum_{i = 0}^{n}t_{i}\btext{a}_{i}, \, \text{where} \, \sum_{i = 0}^{n}t_{i} = 1 \, \text{and} \, t_{i} > 0, \, \forall i \Big\} 
\end{equation}
\end{definition}
Interior of the simplex can also be defined by the equation
\begin{equation}\label{eq:Interior of Simplex2}
\operatorname{Int}(\sigma_n) = \sigma_{n} -  \operatorname{Bd}(\sigma_n)
\end{equation}

Geometrically, the interior excludes the boundary of the simplex. The set $\operatorname{Int}(\sigma_n)$ is also called \textbf{open simplex}.

\begin{definition}[Closure of Simplex]\label{d:Closure of Simplex}
The closure of $\sigma_n$, denoted by $\operatorname{Cl}(\sigma_n)$, is the smallest closed set in $\R^N$ that contains $\sigma_n$.
\end{definition}

\begin{remark}\hfill
\begin{enumerate}
\item The set $\operatorname{Bd}(\sigma_n)$ contains all point $\btext{x}$ such that at least one of the barycentric coordinates $t_{i}$ of $\btext{x}$ is zero. 
\item The set $\operatorname{Int}(\sigma_n)$ contains all point $\btext{x}$ such that  the barycentric coordinates $t_{i}(x) > 0$ for all $i$.
\item From above two observation, we can say that given $\btext{x} \in \sigma_n$, there is exactly one face $\sigma_k$ ($0 \leq k < n$) of $\sigma_n$ such that $\btext{x} \in \operatorname{Int}(\sigma_k)$, for $\sigma_k$ must be the face of $\sigma_n$ spanned by the those $\btext{a}_i$, for which $t_{i}$ is positive.
\end{enumerate}
\end{remark}

\begin{table}[h!]
    \centering
    \renewcommand{\arraystretch}{1.5}
    \begin{tabular}{|p{2cm}|p{3cm}|c|p{3cm}|}
        \hline
        \textbf{Type} & \textbf{Definition} & \textbf{Mathematical Representation} & \textbf{Properties} \\
        \hline
        \textbf{Simplex} $ \sigma_n $ & Convex hull of $ n+1 $ geometrically independent points. & 
        $ \sigma_n = \left\{ x \in \mathbb{R}^N \mid x = \sum_{i=0}^{n} t_i \mathbf{x}_i, \ t_i \geq 0, \ \sum t_i = 1 \right\} $ & Convex, compact, contains interior and boundary. \\
        \hline
        \textbf{Open Simplex} $ \sigma_n^\circ $ & Interior of $ \sigma_n $, excluding boundary points. & 
        $ \sigma_n^\circ = \left\{ x \in \sigma_n \mid t_i > 0, \quad \forall i \right\} $ & Convex, non-compact, contains only points with strictly positive barycentric coordinates. \\
        \hline
        \textbf{Closed Simplex} $ \bar{\sigma}_n $ & Closure of $ \sigma_n $, containing all faces. & 
        $ \bar{\sigma}_n = \sigma_n $ (since $ \sigma_n $ is already closed in $ \mathbb{R}^N $) & Convex, compact, includes boundary and interior. \\
        \hline
        \textbf{Interior of $ \sigma_n $} & Set of points strictly inside $ \sigma_n $, excluding its boundary. & 
        $ \text{Int}(\sigma_n) = \left\{ x \in \sigma_n \mid t_i > 0 \text{ for all } i \right\} $ & Open in the subspace topology. \\
        \hline
        \textbf{Closure of $ \sigma_n $} & Smallest closed set containing $ \sigma_n $. & 
        $ \text{Cl}(\sigma_n) = \sigma_n $ (since $ \sigma_n $ is closed in $ \mathbb{R}^N $) & Convex, compact, contains all boundary points. \\
        \hline
        \textbf{Boundary of $ \sigma_n $} & Set of points on lower-dimensional faces of $ \sigma_n $. & 
        $\operatorname{Bd}(\sigma_n) = \left\{ x \in \sigma_n \mid t_i = 0 \text{ for at least one } i \right\} $ & Compact, convex, union of all $ (n-1) $-faces. \\
        \hline
    \end{tabular}
    \caption{Subsets of Simplices in $\R^N$}
    \label{tab:Subsets of Simplices RN}
\end{table}

%\subsection{Examples of Low Dimension Simplices}\label{ss:Examples of Low Dimension Simplices}

To consolidate our understanding of the foundational concepts related to simplices, we now present some examples of simplices of dimension in finite Euclidean space $\R^n$ in a concrete setting. These examples will help illustrate the definitions of a simplex, its dimension, barycentric coordinates, faces, interior, boundary, and closure.      

\begin{example}[0-simplex]\label{eg:0-simplex}
Define 0-simplex $\sigma_0$ in $\R^n$ and further determine vertex set, dimension, faces, interior, closure and boundary and also display their geometric representation.
\end{example}

\begin{solution*}
A 0-simplex $\sigma_0$ in $\R^n$ is spanned by $A = \{\btext{a}_0\}$ is defined by 
\begin{equation}\label{eq:0-simplex}
 \sigma_0 = \{\btext{x} \in \R \colon \btext{x} = t_{0}\btext{a}_{0}, \, \text{where} \, t_{0} = 1 \, \}   
\end{equation}
\begin{enumerate}
\item \textbf{Vertex Set of $\sigma_0$:} The vertex set of the $\sigma_0$ is $V(\sigma_0) = \{\btext{a}_0\}$. 
\item \textbf{Dimension of $\sigma_0$:} Since dimension of a simplex is defined as one less than the number of its geometrically independent vertices. So, $\dim (\sigma_0) = 1- 1 = 0$. 
\item \textbf{Geometric Picture of $\sigma_0$:} In $\mathbb{R}^n$, a 0-simplex is just a single point 
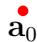
\begin{figure}[h!]
\centering
\begin{tikzpicture}
 \fill [red] (0,0) circle (2pt);
    \node[below] at (0,0) {\textcolor{black}{$\btext{a}_0$}};
\end{tikzpicture}
\caption{0-Simplex $\sigma_0$}\label{fig:0-simplex}
\end{figure}
\item \textbf{Face of $\sigma_0$:} A face of a simplex is any simplex formed by a non-empty subset of its vertex set. Since the only non-empty subset of $A = \{\btext{a}_0\}$ is itself,
the only face of $\sigma_0$ is $\sigma_0$ itself. There are no proper faces of a 0-simplex.
\item \textbf{Interior of $\sigma_0$:} The interior of a simplex consists of all points with strictly positive barycentric coordinates. In this case, the only point is $\btext{a}_0$, and its barycentric coordinate is 1. Since $1 > 0$, $\btext{a}_0$ lies in the interior. Thus, $\text{Int}(\sigma_0) = \{\btext{a}_0\}$. 
\item \textbf{Closure  of $\sigma_0$:} Since $A = \{\sigma_0\}$ is a singleton and already closed as a set in $\mathbb{R}^n$, thus $\text{Cl}(\sigma_0) = \sigma_0$.
\item \textbf{Boundary of $\sigma_0$:} The boundary of an $n$-simplex is the union of all its $(n-1)$-faces. A 0-simplex has no $(0-1) = -1$-faces, so its boundary is empty. Thus $\operatorname{Bd}(\sigma_0) = \emptyset$.
\end{enumerate}
\end{solution*}

\begin{example}[1-simplex]\label{eg:1-simplex}
Define 1-simplex $\sigma_1$ in $\R^n$ and further determine vertex set, dimension, faces, interior, closure and boundary and also display their geometric representation.
\end{example}

\begin{solution*}
A $1$-simplex $\sigma_1$ spanned by $A = \{\btext{a}_0, \btext{a}_1\}$ is 
\begin{equation}\label{eq:1-simplex1}
 \sigma_1 = \{\btext{x} \in \R^n \colon \btext{x} = t_{0}\btext{a}_{0} + t_{1}\btext{a}_{1}, \, \text{where} \, t_{0} + t_1 = 1\, \text{and} \, t_0, t_1 \geq 0 \}   
\end{equation}
Or we can say 
\begin{equation}\label{eq:1-simplex2}
 \sigma_1 = \{\btext{x} \in \R^n \colon \btext{x} = t\btext{a}_{0} + (1 - t)\btext{a}_{1}\, \forall t \in [0,1] \}   
\end{equation}
\begin{enumerate}
\item \textbf{Vertex Set of $\sigma_1$:} The vertex set of the $\sigma_1$ is $V(\sigma_1) = \{\btext{a}_0, \btext{a}_1\}$. 
\item \textbf{Dimension of $\sigma_1$:} Since dimension of a simplex is defined as one less than the number of its geometrically independent vertices. So, $\dim (\sigma_1) = 2 - 1 = 1$. 
\item \textbf{Geometric Picture of $\sigma_1$:} In $\mathbb{R}^n$, a 1-simplex is just a line segment joining the vertices $\btext{a}_0$ and $\btext{a}_1$. 
\begin{figure}[h!]
\centering
\begin{tikzpicture}
        % Define points
        \coordinate (A0) at (0,0);
        \coordinate (A1) at (4,0);
        % Draw the line (1-simplex)
        \draw[thick, blue] (A0) -- (A1);
        % Draw nodes
        \filldraw[red] (A0) circle (2pt) node[below] {\textcolor{black}{$\btext{a}_0$}};
        \filldraw[red] (A1) circle (2pt) node[below] {\textcolor{black}{$\btext{a}_1$}};
    \end{tikzpicture}
\caption{1-Simplex $\sigma_1$}\label{fig:1-simplex}
\end{figure}
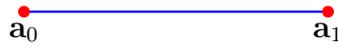
\item \textbf{Faces of $\sigma_1$:} There are two types of faces of $\sigma_1$ are possible according to their dimension 0 and 1. 
\begin{enumerate}
\item \textbf{Faces of $\sigma_1$ of 0-dimension:} The faces of $\sigma_1$ of 0-dimension consists of only vertices of $\sigma_1$ and that is $\{\sigma_{0}^{1}= \{\btext{a}_0\}, \sigma_{0}^{2} = \{\btext{a}_1\}\}$.
\item \textbf{Faces of $\sigma_1$ of 1-dimension:} The face of $\sigma_1$ of 1-dimension consists only a line segment of joining the vertices $\btext{a}_0$ and $\btext{a}_1$ of $\sigma_1$ and that is $\sigma_1 = \{\btext{a}_0, \btext{a}_1\}$.
\end{enumerate}
So, finally collection of all faces of 0 and 1 dimension of a $\sigma_1$ is 
\[\{\{\btext{a}_0\}, \{\btext{a}_1\}, \{\btext{a}_0, \btext{a}_1\}\}\]
\begin{figure}[h!]
\centering
\begin{tikzpicture}[scale=1]
        % Define points
        \coordinate (A0) at (0,0);
        \coordinate (A1) at (10,0);

        % Draw the 1-simplex (the whole edge)
        \draw[thick, blue] (A0) -- (A1) node[midway, above] {\textcolor{black}{$\sigma_1 = \{\btext{a}_0, \btext{a}_1 \}$}};

        % Draw 0-faces (vertices)
        \filldraw[red] (A0) circle (3pt) node[below] {\textcolor{black}{$\sigma_{0}^{1}=\{\mathbf{a}_0\}$}};

        \filldraw[red] (A1) circle (3pt) node[below] {\textcolor{black}{$\sigma_{0}^{2}=\{\mathbf{a}_1\}$}};

        % Labels for faces
       % \node[red] at (-0.4,0) {\textbf{0-face}};
       % \node[red] at (10,0) {\textbf{0-face}};
       % \node[blue] at (2,0.5) {\textbf{1-face}};
    \end{tikzpicture}
\caption{Collection of all faces of 1-Simplex $\sigma_1$}\label{fig:All faces of 1-simplex}
\end{figure}
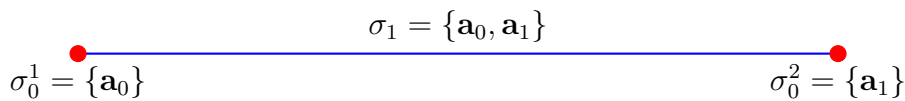
\item \textbf{Interior of $\sigma_1$:} Interior points lie strictly inside the segment joining the vertices $\btext{a}_0$ and $\btext{a}_1$, excluding the endpoints (vertices). 
\begin{equation}\label{eq:1-simplex3}
 \operatorname{Int}(\sigma_1) = \{\btext{x} \in \sigma_1 \colon \btext{x} = t\btext{a}_{0} + (1 - t)\btext{a}_{1}, \, \forall t \in (0,1) \}   
\end{equation}
\begin{figure}[h!]
\centering
    \begin{tikzpicture}[scale=1]
        % Define points
        \coordinate (A0) at (0,0);
        \coordinate (A1) at (10,0);
       % \coordinate (M) at (5,0); % Midpoint for label

        % Draw the full 1-simplex
        \draw[thick, blue] (A0) -- (A1);

        % Draw 0-faces (vertices)
        \filldraw[red] (A0) circle (2.5pt) node[below] {\textcolor{black}{$\btext{a}_0$}};
       \filldraw[red] (A1) circle (2.5pt) node[below] {\textcolor{black}{$\btext{a}_1$}};

        % Draw interior points
        \draw[blue, fill=white] (1,0) circle (2.5pt);
        \draw[blue, fill=white] (2.5,0) circle (2.5pt);
        \draw[blue, fill=white] (4.0,0) circle (2.5pt);
        \draw[blue, fill=white] (5.5,0) circle (2.5pt);
        \draw[blue, fill=white] (7.0,0) circle (2.5pt);
        \draw[blue, fill=white] (8.5,0) circle (2.5pt);
        \node[below] at (5,0) {\textcolor{black}{ \textbf{Interior points}}};
    \end{tikzpicture}
\caption{Interior of $\sigma_1$}\label{fig:Interior of 1-simplex}
\end{figure}
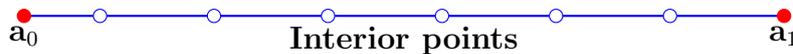
\item \textbf{Closure  of $\sigma_1$:} Since the 1-simplex $\sigma_1$ is already a closed set in $\R^n$, so its closure is the set itself, thus $\text{Cl}(\sigma_1) = \sigma_1$.
\item \textbf{Boundary of $\sigma_1$:} The boundary of an 1-simplex is the union of all its $(1 - 1) = 0$-faces i.e. 
\begin{equation}\label{eq:1-simplex3}
\operatorname{Bd} (\sigma_1) = \sigma_{0}^{1} \cup \sigma_{0}^{2} = \{\btext{a}_{0}\} \cup \{\btext{a}_{0}\} \end{equation}  
\begin{figure}[h!]
\centering
\begin{tikzpicture}[scale=1]
        % Define points
        \coordinate (A0) at (0,0);
        \coordinate (A1) at (10,0);

        % Draw the 1-simplex (the whole edge)
        \draw[thick, blue] (A0) -- (A1);

        % Draw 0-faces (vertices)
        \filldraw[red] (A0) circle (3pt) node[below] {\textcolor{black}{$\sigma_{0}^{1}=\{\mathbf{a}_0\}$ = \textbf{Boundary point}}};

        \filldraw[red] (A1) circle (3pt) node[below] {\textcolor{black}{$\sigma_{0}^{2}=\{\mathbf{a}_1\}$ = \textbf{Boundary point}}};

        % Labels for faces
       % \node[red] at (-0.4,0) {\textbf{0-face}};
       % \node[red] at (10,0) {\textbf{0-face}};
       % \node[blue] at (2,0.5) {\textbf{1-face}};
    \end{tikzpicture}
\caption{Boundaries of 1-Simplex $\sigma_1$}\label{fig:Boundaries of 1-simplex}
\end{figure}
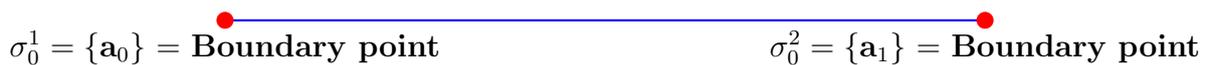
\end{enumerate}
\end{solution*}

\begin{example}[2-simplex]\label{eg:2-simplex}
Define 2-simplex $\sigma_2$ in $\R^n$ and further determine vertex set, dimension, faces, interior, closure and boundary and also display their geometric representation. 
\end{example}

\begin{solution*}
A $2$-simplex $\sigma_2$ spanned by $A = \{\btext{a}_0, \btext{a}_1, \btext{a}_2\}$ is 
\begin{equation}\label{eq:2-simplex1}
 \sigma_2 = \{\btext{x} \in \R^n \colon \btext{x} = t_{0}\btext{a}_{0} + t_{1}\btext{a}_{1} + t_{2}\btext{a}_{2}, \, \text{where} \, t_{0} + t_1 + t_2= 1\, \text{and} \, t_0, t_1, t_2 \geq 0 \}   
\end{equation}
Or we can say 
\begin{equation}\label{eq:2-simplex2}
 \sigma_2 = \Big \{\btext{x} \in \R^n \colon \btext{x} = t_{0}\btext{a}_{0} + (1 - t_{0}) \left[ \dfrac{t_{1}}{\lambda}\btext{a}_{1}+ \dfrac{t_{2}}{\lambda}\btext{a}_{2}\right],\, \text{where}\,  \lambda = 1 - t_{0} \Big \} 
\end{equation}
where the term
\[(1 - t_{0}) \left[ \dfrac{t_{1}}{\lambda}\btext{a}_{1}+ \dfrac{t_{2}}{\lambda}\btext{a}_{2}\right]\]
represent a point $\btext{p} \in \sigma_2$ of the line segment joining the point $\btext{a}_1$ and $\btext{a}_2$ and $\dfrac{t_{1}+ t_{2}}{\lambda} = 1$ and $\dfrac{t_{1}}{\lambda}, \dfrac{t_{2}}{\lambda} \geq 0 $. 
\begin{enumerate}
\item \textbf{Vertex Set of $\sigma_2$:} The vertex set of the $\sigma_2$ is $V(\sigma_2) = \{\btext{a}_0, \btext{a}_1, \btext{a}_2\}$. 
\item \textbf{Dimension of $\sigma_2$:} Since dimension of a simplex is defined as one less than the number of its geometrically independent vertices. So, $\dim (\sigma_2) = 3 - 1 = 2$. 
\item \textbf{Geometric Picture of $\sigma_2$:} In $\mathbb{R}^n$, a 2-simplex is the filled triangle made by the points $\btext{a}_0, \btext{a}_1, \btext{a}_2$.
\begin{figure}[h!]
  \centering
  \begin{tikzpicture}[scale=1]
        % Define points
        \coordinate (A0) at (0,0);
        \coordinate (A1) at (4,0);
        \coordinate (A2) at (2,3);

        \filldraw[red] (A0) circle (2pt) node[below] {\textcolor{black}{$\mathbf{a}_0$}};
        \filldraw[red] (A1) circle (2pt) node[below] {\textcolor{black}{$\mathbf{a}_1$}};
        \filldraw[red] (A2) circle (2pt) node[above] {\textcolor{black}{$\mathbf{a}_2$}};
        % Draw the filled triangle
        \fill[blue!20] (A0) -- (A1) -- (A2) -- cycle;
        
        % Draw the edges of the simplex
        \draw[thick, black] (A0) -- (A1) -- (A2) -- cycle;
        
        % Label the vertices
        %\node[left] at (A0) {$\mathbf{a}_0$};
        %\node[right] at (A1) {$\mathbf{a}_1$};
        \node[above] at (A2) {$\mathbf{a}_2$};
    \end{tikzpicture}
  \caption{2-Simplex $\sigma_2$}\label{fig:2-simplex}
\end{figure}
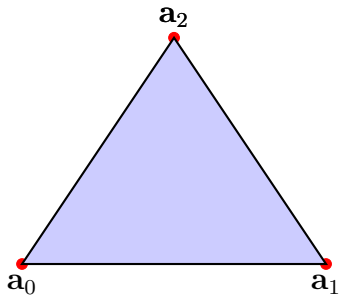
\item \textbf{Faces of $\sigma_2$:} There are three types of faces of $\sigma_2$ are possible according to their dimensions of 0, 1 and 2. 
\begin{enumerate}
\item \textbf{Faces of $\sigma_2$ of 0-dimension:} The faces of $\sigma_2$ of 0-dimension consists of only vertices of $\sigma_2$ and that is $\{\sigma_{0}^{1}= \{\btext{a}_0\}, \sigma_{0}^{2} = \{\btext{a}_1\}, \sigma_{0}^{3} = \{\btext{a}_2\}\}$.
\item \textbf{Faces of $\sigma_2$ of 1-dimension:} The face of $\sigma_2$ of 1-dimension consists only a line segment of joining the vertices $\btext{a}_0$ and $\btext{a}_1$; $\btext{a}_1$ and $\btext{a}_2$; $\btext{a}_2$ and $\btext{a}_1$  that is $\{\sigma_{1}^{1} = \{\btext{a}_{0}, \btext{a}_{1} \},  \sigma_{1}^{2} = \{\btext{a}_{1}, \btext{a}_{2}\}, \sigma_{1}^{3} = \{\btext{a}_{2}, \btext{a}_{0} \}\}$.
\item \textbf{Faces of $\sigma_2$ of 2-dimension:} The face of $\sigma_2$ of 2-dimension is filled triangle $\sigma_{2}^{1} = \{\btext{a}_{0}, \btext{a}_{1}, \btext{a}_2 \}$.
\end{enumerate}
So, finally collection of all faces of 0, 1 and 2 dimension of a $\sigma_2$ is
\[\{\{\btext{a}_0\}, \{\btext{a}_1\},\{\btext{a}_2\}, \{\btext{a}_0, \btext{a}_{1}\}, \{\btext{a}_1, \btext{a}_{2}\}, \{\btext{a}_2, \btext{a}_{1}\}, \{\btext{a}_0, \btext{a}_1, \btext{a}_2\}\}\]
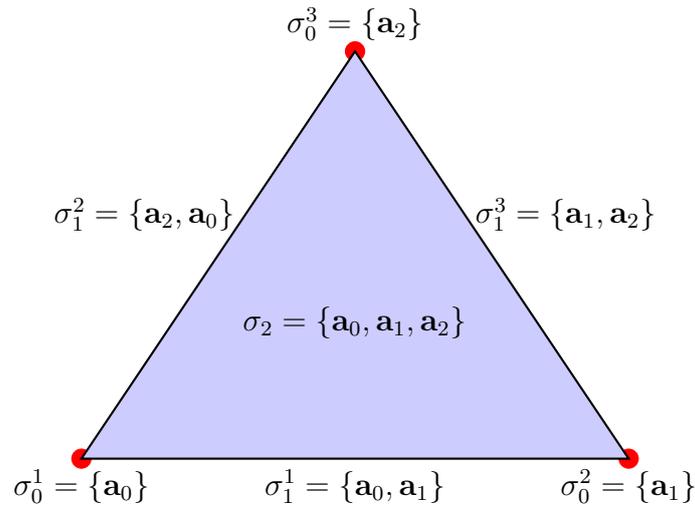
\begin{figure}[h!]
\centering
\begin{tikzpicture}[scale=1.8]
        % Define points
        \coordinate (A0) at (0,0);
        \coordinate (A1) at (4,0);
        \coordinate (A2) at (2,3);

        \filldraw[red] (A0) circle (2pt) node[below] {\textcolor{black}{$\sigma_{0}^{1} = \{\mathbf{a}_0$\}}};
        \filldraw[red] (A1) circle (2pt) node[below] {\textcolor{black}{$\sigma_{0}^{2} = \{\mathbf{a}_1$\}}};
        \filldraw[red] (A2) circle (2pt) node[above] {\textcolor{black}{$\sigma_{0}^{3} = \{\mathbf{a}_2$\}}};
        % Draw the filled triangle
        \fill[blue!20] (A0) -- (A1) -- (A2) -- cycle;
        % Draw the edges of the simplex
        \draw[thick, black] (A0) -- (A1) -- (A2) -- cycle;
        %\node[below] at (0,0); 
        % Label the vertices
        \node[below] at (2,0) {$\sigma_{1}^{1} = \{\btext{a}_0, \btext{a}_1\}$};
        \node[left] at (1.2,1.8) {$\sigma_{1}^{2} = \{\btext{a}_2, \btext{a}_0\}$};
         \node[right] at (2.8,1.8) {$\sigma_{1}^{3} = \{\btext{a}_1, \btext{a}_2\}$};
         \node[right] at (1.1,1) {$\sigma_{2} = \{\btext{a}_0, \btext{a}_1, \btext{a}_2 \}$};
    \end{tikzpicture}
\caption{Collection of all faces of 2-Simplex $\sigma_2$}\label{fig:Faces of 2-simplex}
\end{figure}
\item \textbf{Interior of $\sigma_2$:} The interior of the $\sigma_2$ consists of points strictly inside the triangle, excluding edges and vertices. This happens when all barycentric coordinates of any $\btext{x} \in \sigma_2$ are strictly positive. Thus 
\begin{equation}\label{eq:2-simplex3}
 \operatorname{Int}(\sigma_2) = \{\btext{x} \in \sigma_2 \colon \btext{x} = t_0\btext{a}_{0} + t_1\btext{a}_{1} + t_2\btext{a}_{2}, \, 0 < t_0, t_1, t_2 < 1, t_0 +t_1 + t_2 = 1 \}   
\end{equation}
\begin{figure}[h!]
  \centering
  \begin{tikzpicture}[scale=1.4]
        % Define points
        \coordinate (A0) at (0,0);
        \coordinate (A1) at (4,0);
        \coordinate (A2) at (2,3);

        \filldraw[red] (A0) circle (2pt) node[below] {\textcolor{black}{$\mathbf{a}_0$}};
        \filldraw[red] (A1) circle (2pt) node[below] {\textcolor{black}{$\mathbf{a}_1$}};
        \filldraw[red] (A2) circle (2pt) node[above] {\textcolor{black}{$\mathbf{a}_2$}};
        % Draw the filled triangle
        \fill[blue!20] (A0) -- (A1) -- (A2) -- cycle;
        
        % Draw the edges of the simplex
        \draw[thick, black] (A0) -- (A1) -- (A2) -- cycle;
        \filldraw[black] (2,1) circle (2pt);
        \filldraw[black] (1,1.2) circle (2pt);
        \filldraw[black] (2.3,1.6) circle (2pt);
        \node[above] at (A2) {$\mathbf{a}_2$};
        \node[below] at (2,0.6) {\textcolor{black}{ \textbf{Interior points}}};
    \end{tikzpicture}
  \caption{Interior points of 2-Simplex $\sigma_2$}\label{fig:Interior points of 2-simplex}
\end{figure}
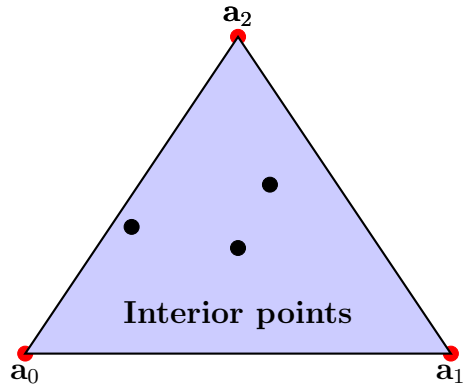
\item \textbf{Closure  of $\sigma_2$:} Since the 2-simplex $\sigma_2$ is already a closed set in $\R^n$, so its closure is the set itself, thus $\text{Cl}(\sigma_2) = \sigma_2$.
\item \textbf{Boundary of $\sigma_2$:} The boundary of $\sigma_2$ is union of its all faces of $1$-dimension
\begin{equation}\label{eq:2-simplex4}
\operatorname{Bd}(\sigma_2) = \sigma_{1}^{1} \cup  \sigma_{1}^{2} \cup \sigma_{1}^{3} = \{\btext{a}_0, \btext{a}_1\} \cup \{\btext{a}_1, \btext{a}_2\} \cup\{\btext{a}_2, \btext{a}_3\} 
\end{equation}
\begin{figure}[h!]
\centering
\begin{tikzpicture}[scale=1.8]
        % Define points
        \coordinate (A0) at (0,0);
        \coordinate (A1) at (4,0);
        \coordinate (A2) at (2,3);

        \filldraw[red] (A0) circle (2pt) node[below] {\textcolor{black}{$\sigma_{0}^{1} = \{\mathbf{a}_0$\}}};
        \filldraw[red] (A1) circle (2pt) node[below] {\textcolor{black}{$\sigma_{0}^{2} = \{\mathbf{a}_1$\}}};
        \filldraw[red] (A2) circle (2pt) node[above] {\textcolor{black}{$\sigma_{0}^{3} = \{\mathbf{a}_2$\}}};
        % Draw the filled triangle
        %\fill[blue!20] (A0) -- (A1) -- (A2) -- cycle;
        % Draw the edges of the simplex
        \draw[thick, black] (A0) -- (A1) -- (A2) -- cycle;
        %\node[below] at (0,0); 
        % Label the vertices
        \node[below] at (2,0) {$\sigma_{1}^{1} = \{\btext{a}_0, \btext{a}_1\}$};
        \node[left] at (1.2,1.8) {$\sigma_{1}^{2} = \{\btext{a}_2, \btext{a}_0\}$};
         \node[right] at (2.8,1.8) {$\sigma_{1}^{3} = \{\btext{a}_1, \btext{a}_2\}$};
         %\node[right] at (1.1,1) {$\sigma_{2} = \{\btext{a}_0, \btext{a}_1, \btext{a}_2 \}$};
    \end{tikzpicture}
\caption{Boundaries of 2-Simplex $\sigma_2$}\label{fig:Boundaries of 2-Simplex}
\end{figure}
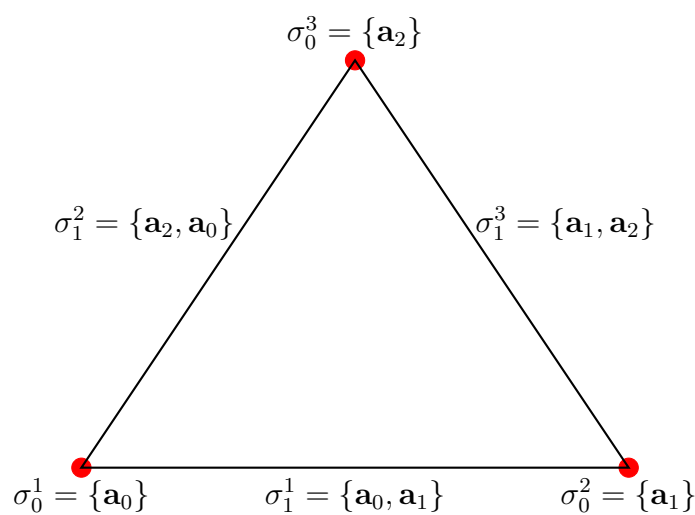
\end{enumerate}
\end{solution*}

\begin{example}[3-simplex]\label{eg:3-simplex}
Define 3-simplex $\sigma_3$ in $\R^n$ and further determine vertex set, dimension, faces, interior, closure and boundary and also display their geometric representation. 
\end{example}

\begin{solution*}
A 3-simplex $\sigma_3$ spanned by  $A = \{\btext{a}_0, \btext{a}_1, \btext{a}_2, \btext{a}_3\}$ is 
\begin{equation}\label{eq:3-simplex1}
 \sigma_3 = \Big \{\btext{x} \in \R^n \colon \btext{x} = \sum_{i = 0}^{3}t_{i}\btext{a}_{i}, \, \text{where} \, \sum_{i = 0}^{3}t_{i} = 1\, \text{and} \, t_i \geq 0, \, i = 0, 1, 2, 3  \Big \} 
\end{equation}
\begin{enumerate}
\item \textbf{Vertex Set of $\sigma_3$:} The vertex set of the $\sigma_3$ is $V(\sigma_3) = \{\btext{a}_0, \btext{a}_1, \btext{a}_2, \btext{a}_3\}$. 
\item \textbf{Dimension of $\sigma_3$:} Since dimension of a simplex is defined as one less than the number of its geometrically independent vertices. So, $\dim (\sigma_3) = 4 - 1 = 3$.
\item \textbf{Geometric Picture of $\sigma_3$:} In $\mathbb{R}^n$, a 3-simplex is the tetrahedron made by four filled triangles.
\begin{figure}[h!]
    \centering
\begin{tikzpicture}[scale=5]
    % Define vertices
    \coordinate (A0) at (0,0,0);
    \coordinate (A1) at (1,0,0);
    \coordinate (A2) at (0,0,1.5);
    \coordinate (A3) at (0,1,0); % Height for 3D effect

    % Draw edges (1-dimensional faces)
    \draw[ultra thick, blue] (A0) -- (A1);
    \draw[ultra thick, blue] (A1) -- (A2);
    \draw[ultra thick, blue] (A2) -- (A0);
    \draw[ultra thick, blue] (A0) -- (A3);
    \draw[ultra thick, blue] (A1) -- (A3);
    \draw[ultra thick, blue] (A2) -- (A3);

    % Fill faces (2-dimensional faces)
    \fill[blue!20,opacity=0.3] (A0) -- (A1) -- (A2) -- cycle;
    \fill[green!20,opacity=0.3] (A0) -- (A1) -- (A3) -- cycle;
    \fill[red!20,opacity=0.3] (A1) -- (A2) -- (A3) -- cycle;
    \fill[yellow!20,opacity=0.3] (A2) -- (A0) -- (A3) -- cycle;
%\filldraw[red] (A1) circle (2pt) node[below] {\textcolor{black}{$\btext{a}_1$}};
    % Draw vertices (0-simplex faces)
    \filldraw[red] (A0) circle (0.7pt) node[below]{\textcolor{black} {$\btext{a}_0$}};
    \filldraw[red] (A1) circle (0.7pt) node[right] {\textcolor{black}{$\mathbf{a}_1$}};
    \filldraw[red] (A2) circle (0.7pt) node[below]{\textcolor{black} {$\mathbf{a}_2$}};
    \filldraw[red] (A3) circle (0.7pt) node[above]{\textcolor{black} {$\mathbf{a}_3$}};
\end{tikzpicture}
\caption{3-simplex $\sigma_3$}  \label{fig:3-simplex}
\end{figure}
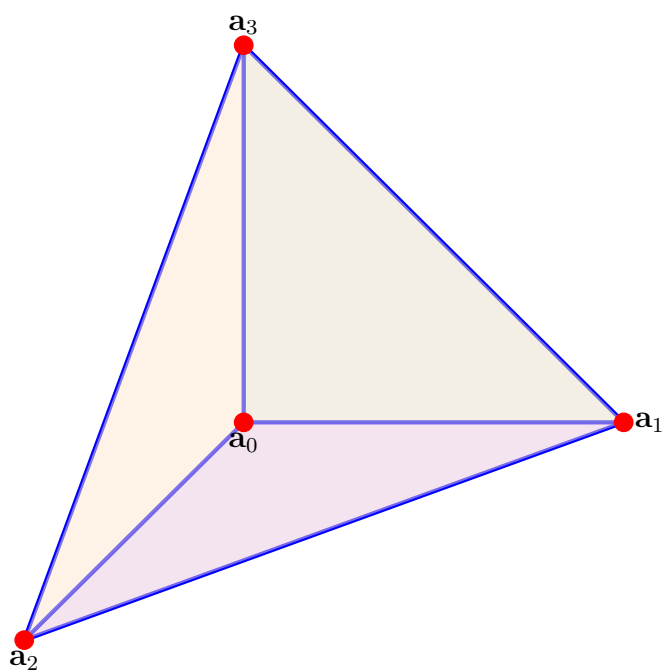
\item \textbf{Faces of $\sigma_3$:} There are four types of faces of $\sigma_3$ are possible according to their dimensions of 0, 1, 2 and 3.
\begin{enumerate}
\item \textbf{Faces of $\sigma_3$ of 0-dimension:} The faces of $\sigma_3$ of 0-dimension consists of only vertices of $\sigma_3$ and that is $\{\sigma_{0}^{1}= \{\btext{a}_0\}, \sigma_{0}^{2} = \{\btext{a}_1\}, \sigma_{0}^{3} = \{\btext{a}_2\}, \sigma_{0}^{4} = \{\btext{a}_3\}\}$.
\item \textbf{Faces of $\sigma_3$ of 1-dimension:} The face of $\sigma_3$ of 1-dimension consists only all line segments of joining the any two vertices that is 
\begin{multline*}
\{\sigma_{1}^{1} = \{\btext{a}_{0}, \btext{a}_{1} \},  \sigma_{1}^{2} = \{\btext{a}_{0}, \btext{a}_{2}\}, \\ \sigma_{1}^{3} = \{\btext{a}_{0}, \btext{a}_{3} \},  \sigma_{1}^{4} = \{\btext{a}_{1}, \btext{a}_{2} \}, \sigma_{1}^{5} = \{\btext{a}_{1}, \btext{a}_{3} \}, \sigma_{1}^{6} = \{\btext{a}_{3}, \btext{a}_{2} \}\}    
\end{multline*}
\item \textbf{Faces of $\sigma_2$ of 2-dimension:} The face of $\sigma_2$ of 2-dimension is filled triangle 
\[
\sigma_{2}^{1} = \{\btext{a}_{0}, \btext{a}_{1}, \btext{a}_2 \}, \sigma_{2}^{2} = \{\btext{a}_{0}, \btext{a}_{1}, \btext{a}_3 \}, \sigma_{2}^{3} = \{\btext{a}_{0}, \btext{a}_{2}, \btext{a}_3 \}, \sigma_{2}^{4} = \{\btext{a}_{1}, \btext{a}_{1}, \btext{a}_3 \}\]
\item \textbf{Faces of $\sigma_3$ of 3-dimension:} The face of $\sigma_3$ of 3-dimension is tetrahedron $\sigma_{3} = \{\btext{a}_{0}, \btext{a}_{1}, \btext{a}_2,  \btext{a}_{3}, \btext{a}_{4}\}$.
\end{enumerate}
So, finally collection of all faces of 0, 1, 2 and 3 dimension of $\sigma_3$ is
\begin{multline*}
\{\{\btext{a}_0\}, \{\btext{a}_1\}, \{\btext{a}_2\}, \{\btext{a}_3\}, \{\btext{a}_{0}, \btext{a}_{1} \}, \{\btext{a}_{0}, \btext{a}_{2} \}, \\ \{\btext{a}_{0}, \btext{a}_{3} \}, \{\btext{a}_{1}, \btext{a}_{2}\}, \{\btext{a}_{1}, \btext{a}_{3}\}, \{\btext{a}_{3}, \btext{a}_{2}\},  \{\btext{a}_{0}, \btext{a}_{1}, \btext{a}_2 \},  \\ \{\btext{a}_{0}, \btext{a}_{1}, \btext{a}_3 \}, \{\btext{a}_{0}, \btext{a}_{2}, \btext{a}_3 \}, \{\btext{a}_{1}, \btext{a}_{1}, \btext{a}_3 \} \}
\end{multline*}
\item \textbf{Interior of $\sigma_3$:} The interior of $\sigma_3$, consists of all points $\btext{x}$ in $\sigma_3$ for which all the barycentric coordinates $t_i > 0$
\begin{equation}\label{eq:3-simplex3}
 \operatorname{Int}(\sigma_3) = \left \{\btext{x} \in \sigma_3 \colon \btext{x} = \sum_{i = 0}^{3}t_{i}\btext{a}_{i}, \, 0 < t_{i}< 1, \sum_{i = 0}t_{i} = 1, \forall i \right \}   
\end{equation}
\item \textbf{Closure  of $\sigma_3$:} Since the 3-simplex $\sigma_3$ is already a closed set in $\R^n$, so its closure is the set itself, thus $\text{Cl}(\sigma_3) = \sigma_3$.
\item \textbf{Boundary of $\sigma_3$:} The boundary of $\sigma_3$ is union of its all faces of $2$-dimension
\begin{equation}\label{eq:2-simplex4}
\operatorname{Bd}(\sigma_3) = \sigma_{2}^{1} \cup  \sigma_{2}^{2} \cup \sigma_{2}^{3} \cup \sigma_{2}^{4} = \{\btext{a}_{0}, \btext{a}_{1}, \btext{a}_2 \}\cup   \{\btext{a}_{0}, \btext{a}_{1}, \btext{a}_3 \}\cup \{\btext{a}_{0}, \btext{a}_{2}, \btext{a}_3 \}\cup \{\btext{a}_{1}, \btext{a}_{1}, \btext{a}_3 \} 
\end{equation}
\end{enumerate}
\end{solution*}

%\subsection{Examples of High Dimension Simplices}\label{ss:Examples of High Dimension Simplices}

%%%%%%%%%%%%%%%%%%%%%%%%%%%%%%%%%%%%%%%%%%%%%%%%%%%%%%%%%%%%%%%%%%%%%%%%

Now we are going to discuss the important properties of simplifications which are very useful for the study of simplification complexes.

For convenience, let us recall some important notations related to simplices. Consider a geometrically independent set of points 
\( A = \{\btext{a}_0, \btext{a}_1, \btext{a}_2, \ldots, \btext{a}_n\} \subset \mathbb{R}^N \), which determines an \( n \)-plane \( P_n \). The simplex spanned by these points is denoted as \(\sigma_n\). For an arbitrary point \(\btext{x} \in \sigma_n\), there exist unique barycentric coordinates \(\{t_i(\btext{x})\}\) satisfying
\[
\btext{x} = \sum_{i=0}^n t_i \btext{a}_i, \quad \text{where} \quad \sum_{i=0}^n t_i = 1 \quad \text{and} \quad t_i \geq 0 \quad \text{for all } i.
\]

The barycentric coordinates not only describe the position of $\btext{x}$ within the simplex, but also vary continuously as $\btext{x}$ moves within the simplex. This continuity plays a fundamental role in topology and geometry, ensuring that simplicial maps and constructions behave well under limits and preserve structure.
\begin{theorem}[Continuity of Barycentric Coordinates]\label{t:Continuity of Barycentric Coordinates}
Let $n$-simplex $\sigma_n$ is spanned by the set $A = \{\btext{a}_0, \btext{a}_1, \ldots, \btext{a}_n\}$ in $\R^N$. The barycentric coordinate $t_{i}(\btext{x})$ of $\btext{x} \in \sigma_n$ w.r.t. set $A$ are continuous functions of $\btext{x}$.
\end{theorem}

\begin{proof}
By definition of $n$-simplex, any point $\btext{x} \in \sigma_n$ can be expressed uniquely as a convex combination of the points in $A = \{\btext{a}_0, \btext{a}_1, \ldots, \btext{a}_n\}$, so
\begin{equation}\label{eq:Continuity of barycentric coordinates1}
 \btext{x} = \sum_{i=0}^n t_i(\btext{x}) \btext{a}_i, \quad \text{where} t_i(\btext{x}) \geq 0, \quad \sum_{i=0}^n t_i(\btext{x}) = 1 
\end{equation}    
where scalars $t_i(\btext{x})$ are called the barycentric coordinates of $\btext{x}$ with respect to the set $A$. The uniqueness of the barycentric coordinates follows from the geometric independence of $A$, which ensures that the affine hull of $A$ is $n$-dimensional.

Let us define the matrix $M$ whose columns are $\btext{a}_1 - \btext{a}_0, \btext{a}_2 - \btext{a}_0, \ldots, \btext{a}_n - \btext{a}_0$
\[  M = \begin{bmatrix}
    | & | & \cdots & | \\
    \btext{a}_1 - \btext{a}_0 & \btext{a}_2 - \btext{a}_0 & \cdots & \btext{a}_n - \btext{a}_0 \\
    | & | & \cdots & |
    \end{bmatrix}
    \]
    
The matrix $M$ has size $ N \times n$, and its rank is $n$ due to the geometric independence of $A$. For $\btext{x} \in \sigma_n$, rewrite $\btext{x}$ in terms of $t_0, t_1, \ldots, t_n$
\[ \btext{\btext{x}} - \btext{a}_0 = \sum_{i=1}^n t_i(\btext{x})(\btext{a}_i - \btext{a}_0)  \]

In matrix form
\[\btext{x} - \btext{a}_0 = M \cdot \btext{t},\]
where $\btext{t} = [t_1(x), t_2(x), \ldots, t_n(x)]^\top$. 

Now solve for $\btext{t}$, we get 
\[\btext{t} = M^\dagger (\btext{x} - \btext{a}_0), \]
where $M^\dagger$ is the pseudoinverse of $M$  (or the inverse if $n = N$). The coordinate $t_0(\btext{x})$ is determined using the condition
\[t_0(\btext{x}) = 1 - \sum_{i=1}^n t_i(\btext{x})\]

Since $M$ is constant and non-singular, the entries of $M^\dagger$ depend continuously on $M$ (which is fixed) and linearly on $\btext{x}$. Thus, the mapping:
\[\btext{x} \mapsto M^\dagger (\btext{x} - \btext{a}_0)    \]
is a continuous function of $\btext{x}$. The barycentric coordinates $ (t_1(\btext{x}), t_2(\btext{x}), \ldots, t_n(\btext{x})) $ are linear transformations of $\btext{x}$ and hence are continuous functions. Finally, $t_0(\btext{x}) = 1 - \sum_{i=1}^n t_i(\btext{x})$ is also a continuous function since it is a linear combination of continuous functions.
\end{proof}
\begin{workingrule}[Determining the Barycentric Coordinates of a Point in a Simplex]
Let $\sigma_n$ be an $n$-simplex in $\mathbb{R}^m$ with vertex set $V(\sigma_n) = \{\btext{a}_0, \btext{a}_1, \ldots, \btext{a}_n\}$. To determine the barycentric coordinates $(t_0, t_1, \ldots, t_n)$ of a point $\btext{x} \in \sigma_n$, follow these steps:

\begin{enumerate}
\item \textbf{Formulate position vector:} Suppose the vertices of simplex $\sigma_n$ are position vectors in $\R^m$. 
\item \textbf{Define relative vectors:} Form the relative vectors w.r.t. $\btext{a}_0$ 
\[    \btext{v}_i = \btext{a}_i - \btext{a}_0 \quad \text{for } i = 1, 2, \ldots, n    \]
Also define $\btext{v} = \btext{x} - \btext{a}_0$.
\item \textbf{Construct the matrix equation:}     Express vector $\btext{v}$ as a linear combination of the $\btext{v}_i$, then write
\[\btext{v} = \sum_{i=1}^n t_i \btext{v}_i\]
This leads to the matrix equation
\[\btext{v} = M \cdot \btext{t'},\]
where $M$ is the $m \times n$ matrix with columns $\btext{v}_1, \ldots, \btext{v}_n$ and 
$\btext{t'} = \begin{bmatrix} t_1 \\ \vdots \\ t_n \end{bmatrix}$.
\item \textbf{Solve the linear system:} Solve the matrix equation
\[ \btext{v} = M \cdot \btext{t'} \]
using Gaussian elimination or another method, we will obtain $t_1, \ldots, t_n$. 
\item \textbf{Determine $t_0$:} Compute
\[t_0 = 1 - \sum_{i=1}^n t_i\]
    
\item \textbf{Verify:} Confirm these condition for all $t_i$
\[ \sum_{i=0}^n t_i = 1 \quad \text{and} \quad t_i \geq 0 \text{ for all } i\]
If above conditions are true, then $(t_0, t_1, \ldots, t_n)$ are the barycentric coordinates of $\btext{x}$ with respect to $\sigma_n$.
\end{enumerate}
\end{workingrule}

\begin{example}
Let $\sigma_2$ be a $2$-simplex in $\mathbb{R}^2$ with vertex set
\[
V(\sigma_2) = \{\btext{a}_0 = (0,0), \, \btext{a}_1 = (2,0), \, \btext{a}_2 = (0,2)\}
\]
Determine the barycentric coordinates of the point $\btext{x} = (1,1)$ with respect to $\sigma_2$.
\end{example}

\begin{solution*}
We follow the working rule
\begin{enumerate}
\item \textbf{Formulate position vectors:} Suppose the vertices of simplex $\sigma_2$ are position
vectors in $\R^2$.
    \item \textbf{Define relative vectors:} The relative vectors w.r.t. $\btext{a}_0$ as follows 
    \[
    \btext{v}_1 = \btext{a}_1 - \btext{a}_0 = (2,0), \quad
    \btext{v}_2 = \btext{a}_2 - \btext{a}_0 = (0,2),
    \]
    \[
    \btext{v} = \btext{x} - \btext{a}_0 = (1,1).
    \]

    \item \textbf{Construct the matrix equation:} The expression of vector $\btext{v}$ as a linear combination of $\btext{v}_{1}$ and $\btext{v}_{2}$ is 
    \[    \btext{v} = t_1 \btext{v}_1 + t_2 \btext{v}_2 = t_1(2,0) + t_2(0,2) = (2t_1, 2t_2)
    \]
This leads the matrix equation 
\begin{align*}
\btext{v} & = M \cdot \btext{t'} \\
\begin{bmatrix} 1 \\  1\end{bmatrix} & = 
\begin{bmatrix} 2 &  0 \\  0 & 2 \end{bmatrix} \begin{bmatrix} t_1 \\  t_2\end{bmatrix}
\end{align*}
    \item \textbf{Solve the linear system:}
    On solving of abolve matrix equation of linear system, we get 
    \[
    (2t_1, 2t_2) = (1,1) \Rightarrow
    \begin{cases}
    2t_1 = 1 \Rightarrow t_1 = \frac{1}{2} \\
    2t_2 = 1 \Rightarrow t_2 = \frac{1}{2}
    \end{cases}
    \]
    \item \textbf{Determine $t_0$:}
    \[
    t_0 = 1 - t_1 - t_2 = 1 - \frac{1}{2} - \frac{1}{2} = 0.
    \]
    \end{enumerate}
Now we can see that all $t_{i} \geq 0$ and $\sum_{i = 0}^{2}t_{i} = 1$, therefore, the barycentric coordinates of $\btext{x} = (1,1)$ with respect to $\sigma_2$ are
\[    (t_0, t_1, t_2) = \left(0, \frac{1}{2}, \frac{1}{2}\right)\]
\end{solution*}

\begin{example}
Let $\sigma_3$ be a 3-simplex in $\mathbb{R}^4$ with vertex set
\[
V(\sigma_3) = \{\btext{a}_0 = (0, 0, 0, 0), \, \btext{a}_1 = (1, 0, 0, 0), \, \btext{a}_2 = (0, 1, 0, 0), \, \btext{a}_3 = (0, 0, 1, 1)\}.
\]
Determine the barycentric coordinates of the point $\btext{x} = \left(\frac{1}{4}, \frac{1}{4}, \frac{1}{4}, \frac{1}{4}\right)$.
\end{example}

\begin{solution*}
We follow the working rule step-by-step:
\begin{enumerate}
\item Suppose the vertices of simplex $\sigma_3$ are position vectors in $\R^4$. 
    \item The relative vectors w.r.t. $\btext{a}_0$ as follows
    \begin{gather*}
    \btext{v}_1 = \btext{a}_1 - \btext{a}_0 = (1, 0, 0, 0), \,
    \btext{v}_2 = \btext{a}_2 - \btext{a}_0 = (0, 1, 0, 0), \,
    \btext{v}_3 = \btext{a}_3 - \btext{a}_0 = (0, 0, 1, 1), \\
    \btext{v} = \btext{x} - \btext{a}_0 = \left(\frac{1}{4}, \frac{1}{4}, \frac{1}{4}, \frac{1}{4}\right)     
    \end{gather*}
    
    \item The expression of vector $\btext{v}$ as linear combination of $\btext{v}_1, \btext{v}_2$ and $\btext{v}_3$ is 
    \[    \btext{v} = t_1 \btext{v}_1 + t_2 \btext{v}_2 + t_3 \btext{v}_3    \]
    This leads the matrix equation
    \begin{align*}
\btext{v} & = M \cdot \btext{t'} \\
\begin{bmatrix} \tfrac{1}{4} \\  \tfrac{1}{4} \\ \tfrac{1}{4} \\ \tfrac{1}{4} \end{bmatrix} & = 
\begin{bmatrix} 1 &  0 & 0 \\  0 & 1 & 0 \\ 0 & 0 & 1 \\ 0 & 0 & 1 \end{bmatrix} \begin{bmatrix} t_1 \\  t_2 \\ t_3 \end{bmatrix}
\end{align*}
    \item On solving of above matrix equation of linear system we, 
    \[t_1 = \frac{1}{4}, t_2 = \frac{1}{4}, t_3 = \frac{1}{4}\]
    and 
    \[    t_0 = 1 - t_1 - t_2 - t_3 = 1 - \frac{1}{4} - \frac{1}{4}  - \frac{1}{4} = \frac{1}{4}   \]
\end{enumerate}
Now we can see that all $t_{i} \geq 0$ and $\sum_{i = 0}^{3}t_{i} = 1$, therefore, the barycentric coordinates of  $\btext{x} = \left(\frac{1}{4}, \frac{1}{4}, \frac{1}{4}, \frac{1}{4}\right)$ with respect to $\sigma_3$ are
\[\left(t_0, t_1, t_2, t_3\right) = \left( \frac{1}{4}, \frac{1}{4}, \frac{1}{4}, \frac{1}{4} \right)\]
\end{solution*}

The structure of an $n$-simplex can be better understood by studying how it is built from lower-dimensional simplices. One particularly intuitive and foundational idea is that an $n$-simplex can be constructed as a cone over an $(n-1)$-simplex. The following theorem captures this recursive geometric nature of simplices. It allows us to decompose higher-dimensional simplices into simpler parts and visualize them through familiar structures like line segments and lower-dimensional simplices.

\begin{theorem}[Conical Construction of an $n$-Simplex from an $(n-1)$-Simplex]\label{t:Conical Construction of an n Simplex}
The $n$-simplex $\sigma_n$ is equal to union of all line segments joining $\btext{a}_0$ to the points of the simplex $\sigma_{n-1}$ spanned by $\btext{a}_1, \btext{a}_2, \ldots, \btext{a}_n$. Two such line segments intersect only in the point $\btext{a}_0$.
\end{theorem}

\begin{proof}
The $n$-simplex $\sigma_n$ is the convex hull of $(n+1)$ geometrically independent points $A = \{\btext{a}_0, \btext{a}_1, \ldots, \btext{a}_n\}$
\begin{equation}\label{eq:Simplex is unions of other simplex1}
\sigma_n = \text{conv}(A) = \left\{\btext{x} \in \R^N \colon \btext{x} = \sum_{i=0}^n t_{i} \btext{a}_i, \; t_i \geq 0, \; \sum_{i=0}^n t_i = 1 \right\}
\end{equation}

The subset $\sigma_{n-1} \subset \sigma_n$ is the simplex spanned by the $n$ points $\{\btext{a}_1, \btext{a}_2, \ldots, \btext{a}_n\}$
\begin{equation}\label{eq:Simplex is unions of other simplex2}
 \sigma_{n-1} = \text{conv}(\{\btext{a}_1, \btext{a}_2, \ldots, \btext{a}_n\}) = \left\{\btext{y}\in \R^N \colon \btext{y} = \sum_{i=1}^n t_{i}\btext{a}_i, \; t_i \geq 0, \; \sum_{i=1}^n t_i = 1 \right\}   
\end{equation}

Any point $\btext{x} \in \sigma_n$ can be expressed as
\[\btext{x} = t_0 \btext{a}_0 + (1 - t_0)\btext{y}, \, \text{where}\, \btext{y} \in \sigma_{n-1}, \; t_0 \in [0, 1].   \]

Rearranging, we have
\[\btext{x} = (1 - t_0)\btext{y} + t_0 \btext{a}_0,\]
which is a convex combination of $\btext{y} \in \sigma_{n-1}$ and $\btext{a}_0$. Thus, $\btext{x}$ lies on the line segment joining $\btext{a}_0$ to some point $\btext{y} \in \sigma_{n-1}$.

Taking the union of all such line segments for $\btext{y} \in \sigma_{n-1}$, we obtain
\[ \sigma_n = \bigcup_{\btext{y} \in \sigma_{n-1}} \text{conv}(\{\btext{a}_0, \btext{y}\})\]

Consider two distinct points $\btext{y}_1, \btext{y}_2 \in \sigma_{n-1}$ such that $\btext{y}_1 \neq \btext{y}_2$.  The line segment joining $\btext{a}_0$ to $\btext{y}_1$ is:
\[L_1 = \{\btext{x} \in \R^N \colon \btext{x} = t \btext{a}_0 + (1 - t)\btext{y}_1, \; 0 \leq t \leq 1\}\]

Similarly, the line segment joining $\btext{a}_0$ to $\btext{y}_2$ is:
\[L_2 = \{\btext{x} \in \R^N \colon \btext{x} = t \btext{a}_0 + (1 - t)\btext{y}_2, \; 0 \leq t \leq 1\}\]

If $L_1 \cap L_2 \neq \emptyset$, then there exists some $\btext{x} \in L_1 \cap L_2$ such that
\[\btext{x} = t_1 \btext{a}_0 + (1 - t_1)\btext{y}_1 = t_2 \btext{a}_0 + (1 - t_2)\btext{y}_2,   \]
where $0 \leq t_1, t_2 \leq 1$. Rearranging, we find
\[(t_1 - t_2)\btext{a}_0 = (1 - t_2)\btext{y}_2 - (1 - t_1)\btext{y}_1\]

Since $\btext{a}_0, \btext{y}_1,$ and $\btext{y}_2$ are geometrically independent, this equation holds if and only if
\[t_1 = t_2 \quad \text{and} \quad \btext{y}_1 = \btext{y}_2\]

Hence, the only point of intersection between $L_1$ and $L_2$ is $\btext{a}_0$.

The $n$-simplex $\sigma_n$ is the union of all line segments joining $\btext{a}_0$ to the points of $\sigma_{n-1}$
\[\sigma_n = \bigcup_{\btext{y} \in \sigma_{n-1}} \text{conv}(\{\btext{a}_0, \btext{y}\}). \]

Two such line segments intersect only at $\btext{a}_0$. 
\end{proof}
\begin{example}
Let the 2-simplex $\sigma_2$ in $\mathbb{R}^3$ with vertex set is $V(\sigma_2) = \{\btext{a}_0 = (0, 0, 0), \btext{a}_1= (1, 0, 0), \btext{a}_2 = (0, 1, 0)\}$. Show that $\sigma_2$ can be viewed as the union of all line segments connecting $\btext{a}_0$ to every point on the 1-simplex $\sigma_1$ spanned by vertex set $V(\sigma_1) = \{\btext{a}_1, \btext{a}_2\}$.
\end{example}

\begin{solution*}
\begin{figure}[h]
\centering
\begin{tikzpicture}[scale=5]

% Define the vertices
\coordinate (A0) at (0, 0);
\coordinate (A1) at (1.2, 0.2);
\coordinate (A2) at (0.6, 1.1);

% Fill the triangle
\fill[blue!10] (A0) -- (A1) -- (A2) -- cycle;

% Draw edges
\draw[thick] (A0) -- (A1);
\draw[thick] (A0) -- (A2);
\draw[thick] (A1) -- (A2);

% Draw a few sample line segments from a0 to points on edge [a1,a2]
\foreach \t in {0.2, 0.4, 0.6, 0.8} {
    \coordinate (P) at ($ (A1)!\t!(A2) $);
    \draw[dashed, gray] (A0) -- (P);
}

% Draw vertices
\filldraw[red] (A0) circle (0.7pt) node[below left] {$\bm{a}_0$};
\filldraw[black] (A1) circle (0.7pt) node[below right] {$\bm{a}_1$};
\filldraw[black] (A2) circle (0.7pt) node[above] {$\bm{a}_2$};
% Label
\node at (0.5, 0.3) {$\sigma_2$};

\end{tikzpicture}
\caption{Constructing the 2-simplex $\sigma_2$ from $\bm{a}_0$ and the edge $[\bm{a}_1, \bm{a}_2]$}
\label{fig:simplex2}
\end{figure}
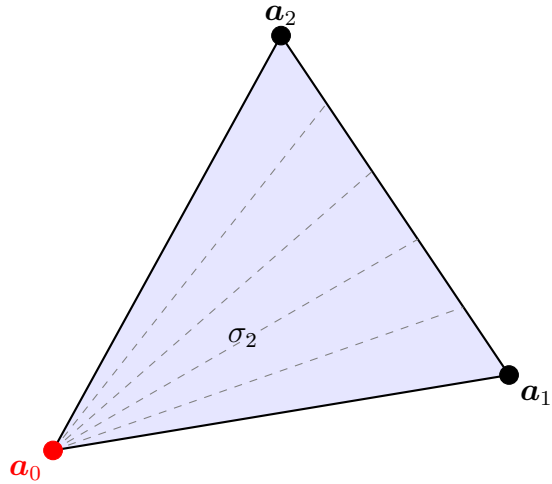
The 1-simplex is the set of all convex combinations defined as 
\begin{align*}
\sigma_1 & = \{\btext{y} \in \R^3 \colon  \btext{y} = t\btext{a}_1 + (1 - t)\btext{a}_2, \, t \in [0, 1]\} \\   
& = \{\btext{y} \in \R^3 \colon  \btext{y} = t(1,0, 0)+ (1 - t)(0, 1, 0), \, t \in [0, 1]\} \\
& = \{\btext{y} \in \R^3 \colon  \btext{y} = (t, 1-t, 0), \, t \in [0, 1]\}
\end{align*}
A line segment from $\btext{a}_0$ to a point $\btext{y}_{0} \in \sigma_1$ can write as
\begin{align*}
L_{\btext{a}_0 \to \btext{y}_0} & = \{\btext{x} \in \R^3 \colon \btext{x} = s\btext{a}_{0} + (1 - s)\btext{y}_0, \, s \in [0, 1]\}\\
& = \{\btext{x} \in \R^3 \colon \btext{x} = s(1, 0, 0) + (1 - s)(t, 1-t, 0), \, s, t \in [0, 1]\} \\
& = \{\btext{x} \in \R^3 \colon \btext{x} = (s+t-ts, 1- s - t + st, 0), \, s, t \in [0, 1]\}
\end{align*}
Therefore, collection of all line segments from $\btext{a}_0$ to each point $\btext{y} \in \sigma_1$ actual gives the 2-simplex $\sigma_2$. So,
\[\sigma_2 = \bigcup_{t \in [0, 1]} \{s\btext{a}_0 + (1 - s)\btext{y}\colon \btext{y}= (t, 1- t, 0), s \in [0, 1]\}\]
\end{solution*}

The concept of convexity plays a foundational role in the geometry and topology of simplices. A set in $\R^n$ is said to be convex if, for any two points in the set, the entire line segment joining them also lies within the set. This property ensures that the shape has no ``dents" or ``holes" and is geometrically well-behaved. The $n$-simplex, defined as the set of all convex combinations of its $n+1$ geometrically independent vertices, naturally inherits this convex structure. Understanding the convexity of simplices not only confirms their geometric coherence but also establishes a crucial bridge to more advanced constructions in simplicial complexes, convex polytopes, and topological spaces.

\begin{definition}[Convex set in $\R^N$]\label{d:Convex set in R^N}
A subset $A$ of $\R^N$ is called convex if for each pair $\btext{x}, \btext{y}$ of points $A$, the line segments joining them lies in $A$. Symbolically, $A$ is convex if 
\begin{equation}\label{eq:Condition for convexity of set in R^N}
(1 - t)\btext{x} + t\btext{y} \in A, \forall \, \btext{x}, \btext{y} \in A \, \text{and} \, \forall \, t \in [0, 1]
\end{equation}
where the expression $(1 - t)\btext{x} + t\btext{y}$ represents a point on the line segment between points $\btext{x}$ and $\btext{y}$ of $A$. The parameter $t$ determines how far along the line segment the point lies. If $t = 0$, the point is at $\btext{x}$, if $t = 1$, the point is at $\btext{y}$ and if $t \in (0, 1)$, the point lies somewhere between $\btext{x}$ and $\btext{y}$.
\end{definition}

\begin{example}[Examples of Convex Sets in $\R^N$]\label{eg:Examples of Convex Sets in RN}
The set $\R^N$, line segments, triangle  (including its interior), a circle or sphere (including the interior) and  convex polygon or polyhedron are convex sets in $\R^N$. 
\end{example}

\begin{example}[Non-Examples of Convex Sets in $\R^N$]\label{eg:Non-Examples of Convex Sets in RN}
A crescent or ``moon-shaped" region, a set with a ``hole" or gap (e.g., a circle with its interior removed) and two disjoint sets.
\end{example}

\begin{theorem}[Convexity of $n$-simplex]\label{t:Convexity of n-simplex}
The $n$-simplex $\sigma_n$ spanned by the points of the $A = \{\btext{a}_0, \btext{a}_1, \ldots, \btext{a}_n\}$ is convex in $\R^N$. 
\end{theorem}

\begin{proof}
Let $A = \{\btext{a}_0, \btext{a}_1, \ldots, \btext{a}_n\}$ is geometrically independent subset in $\R^N$ and $\sigma_n$ is $n$-simplex spanned by the points $\btext{a}_0, \btext{a}_1, \ldots, \btext{a}_n$ of $A$, then by Definition \eqref{d:n-simplex}
\begin{equation}\label{eq:Convexity of n-simplex1}
 \sigma_n = \big \{\btext{x} \in \R^N \colon \btext{x} = \sum_{i = 0}^{n}t_{i}\btext{a}_{i}, \, \text{where} \, \sum_{i = 0}^{n}t_{i} = 1 \, \text{and} \, t_{i} \geq 0, \, \forall i \big\} 
\end{equation}

Let $\btext{x}, \btext{y} \in \sigma_n$, then by \eqref{eq:Convexity of n-simplex1} 
\begin{equation}\label{eq:Convexity of n-simplex2}
\btext{x} = \sum_{i = 0}^{n}\alpha_{i}\btext{a}_{i}, \, \text{where} \, \sum_{i = 0}^{n}\alpha_{i} = 1 \, \text{and} \, \alpha_{i} \geq 0, \, \forall i    
\end{equation}

\begin{equation}\label{eq:Convexity of n-simplex3}
\btext{y} = \sum_{i = 0}^{n}\beta_{i}\btext{a}_{i}, \, \text{where} \, \sum_{i = 0}^{n}\beta_{i} = 1 \, \text{and} \, \beta_{i} \geq 0, \, \forall i \\
\end{equation}

Let us consider a point $\btext{z}$ on the line segment between $\btext{x}$ and $\btext{y}$, given by 
\begin{equation}\label{eq:Convexity of n-simplex4}
\btext{z} = (1- t)\btext{x} + t\btext{y}, \, t \in [0, 1]
\end{equation}

Substituting $\btext{x}$ and $\btext{y}$ into this equation \eqref{eq:Convexity of n-simplex4}, we get
\begin{align}\label{eq:Convexity of n-simplex5}
  \btext{z} & = (1 - t)\sum_{i = 0}^{n}\alpha_{i}\btext{a}_{i} + t\sum_{i = 0}^{n}\beta_{i}\btext{a}_{i} \notag \\  
  \btext{z} & = \sum_{i = 0}^{n}[(1 - t)\alpha_{i} + t\beta_{i}]\btext{a}_{i} 
\end{align}

Let $\gamma_{i} = (1 - t)\alpha_{i} + t\beta_{i}$ for all $i$, then \eqref{eq:Convexity of n-simplex5}, gives 
\begin{equation}\label{eq:Convexity of n-simplex6}
\btext{z} = \sum_{i = 0}^{n}\gamma_{i}\btext{a}_{i}
\end{equation}

By by \eqref{eq:Convexity of n-simplex2} and \eqref{eq:Convexity of n-simplex3} $\alpha_{i}, \beta_{i} \geq 0$ for all $i$ and $t \in [0, 1]$, therefore $\gamma_{i} \geq 0$ for all $i$.

\begin{align}\label{eq:Convexity of n-simplex7}
    \sum_{i = 0}^{n}\gamma_{i} & = \sum_{i = 0}^{n}[(1 - t)\alpha_{i} + t\beta_{i}] \notag \\
    & = (1 - t) \sum_{i = 0}^{n}\alpha_{i} + t\sum_{i = 0}^{n}\beta_{i} \notag \\
    & = (1 - t) \cdot 1 + t \cdot 1 \quad (\text{by} \, \eqref{eq:Convexity of n-simplex2} \, \text{and}\, \eqref{eq:Convexity of n-simplex3}) \notag \\
    & = 1- t + t \notag \\
 \sum_{i = 0}^{n}\gamma_{i}   & = 1
\end{align}

So, we have a point $\btext{z} = \sum_{i = 0}^{n}\gamma_{i}\btext{a}_{i}$ such that $\sum_{i = 0}^{n}\gamma_{i}  = 1$ and $\gamma \geq 0$ for all $i$, therefore by Definition \eqref{eq:Convexity of n-simplex1}, $\btext{z} \in \sigma_n$. Now, if $\btext{x}, \btext{y} \in \sigma_{n}$, then there exist a arbitrary point $\btext{z} \in \sigma_n$ such that $\btext{z} = (1 - t)\btext{x} + t\btext{y}$, where $t \in [0, 1]$. Therefore by the Definition of convexity of set \eqref{d:Convex set in R^N}, the $n$-simplex $\sigma_n$ is convex in $\R^N$. 
\end{proof}

\begin{theorem}[Compactness of $n$-simplex]\label{t:Compactness of n-simplex}
The $n$-simplex $\sigma_n$ spanned by the points of the $A = \{\btext{a}_0, \btext{a}_1, \ldots, \btext{a}_n\}$ is compact in $\R^N$.
\end{theorem}

\begin{proof}\hfill
\begin{enumerate}
\item \textbf{Boundedness of $\sigma_n$:} Since $\btext{x} \in \sigma_n$ is a convex combination of the points $\btext{a}_0, \btext{a}_1, \ldots, \btext{a}_n$, we can write
\[
\btext{x} = \sum_{i=0}^n t_{i}\btext{a}_i, \,\text{where} \, \sum_{i=0}^n t_i = 1, \; t_i \geq 0.
\]
Let $M = \max\{\|\btext{a}_0\|, \|\btext{a}_1\|, \ldots, \|\btext{a}_n\|\}$, where $\|\cdot\|$ denotes the Euclidean norm in $\R^N$. Then
\[\|\btext{x}\| \leq \sum_{i=0}^n t_i \|\btext{a}_i\| \leq \sum_{i=0}^n t_i M = M \sum_{i=0}^n t_i = M\]
Thus, $\sigma_n$ is bounded because all points $\btext{x} \in \sigma_n$ lie within a ball of radius $M$ centered at the origin.
\item \textbf{Closeness of $\sigma_n$:} Let $\{\btext{x}_k\}_{k=1}^\infty$ be a sequence of points in $\sigma_n$ that converges to some $\btext{x} \in \R^N$. For each $\btext{x}_k$, there exist barycentric coordinates $\{t_i^{(k)}\}_{i=0}^n$ such that
\[\btext{x}_k = \sum_{i=0}^n t_i^{(k)} \btext{a}_i, \, t_i^{(k)} \geq 0, \quad \sum_{i=0}^n t_i^{(k)} = 1\]
Since $t_i^{(k)} \in [0, 1]$, the sequence $\{t_i^{(k)}\}_{k=1}^\infty$ is bounded. By the Bolzano-Weierstrass Theorem, there exists a convergent subsequence $\{t_i^{(k_j)}\}_{j=1}^\infty$ with limit $t_i$. The limit $\{t_i\}_{i=0}^n$ satisfies
\[t_i \geq 0, \quad \sum_{i=0}^n t_i = 1\]
Thus, the limit point $\btext{x}$ can be written as
\[\btext{x} = \sum_{i=0}^n t_i \btext{a}_i\]
Since $\btext{x} \in \sigma_n$, the set $\sigma_n$ is closed.
\end{enumerate}
The $n$-simplex $\sigma_n$ is a subset of $\R^N$ that is both bounded and closed. By the Heine-Borel Theorem, any closed and bounded subset of $\R^N$ is compact. Therefore the $n$-simplex $\sigma_n$ is compact in $\R^N$. 
\end{proof}

\begin{theorem}\label{t:Simplex is equal to intersection of all convex in Rn}
Let $\sigma_n$ is $n$-simplex spanned by the geometrically independent set $A = \{\btext{a}_0, \btext{a}_1, \ldots, \btext{a}_n\} \subset \R^N$. If $\sigma_n$ is convex and compact in $\R^N$, then $\sigma_n$ is equal to the intersection of all convex sets in $\R^N$ containing $\btext{a}_0, \btext{a}_1, \ldots, \btext{a}_n$.
\end{theorem}

\begin{proof}
As given $\sigma_n$ is $n$-simplex spanned by the $A = \{\btext{a}_0, \btext{a}_1, \ldots, \btext{a}_n\} \subset \R^N$. Let $\mcal{C}$ is collection of all convex set in $\R^N$ that contain the points $ \{\btext{a}_0, \btext{a}_1, \ldots, \btext{a}_n\}$. Now our aim is to show that $\sigma_n = \bigcap \mcal{C}$ i.e. $\sigma_n \subseteq \bigcap \mcal{C}$ and $\bigcap \mcal{C} \subseteq \sigma_n$.

By definition, $\sigma_n$  is a convex set, and it contain $\btext{a}_0, \btext{a}_1, \ldots, \btext{a}_n$. Any convex set $C \in \mcal{C}$  that contain $\btext{a}_0, \btext{a}_1, \ldots, \btext{a}_n$ must also contain all convex combinations of these points, because convex sets are closed under convex combinations. Thus, $\sigma_n \in C$ for every $C \in \mcal{C}$. Since this is true for every $C \in \mcal{C}$, therefore $\sigma_n \subseteq \bigcap \mcal{C}$.

Consider any point $\btext{x} \in \bigcap \mcal{C}$. By definition, $\btext{x} \in C$ for every $C \in \mcal{C}$ that contains $\btext{a}_0, \btext{a}_1, \ldots, \btext{a}_n$. As $\sigma_n$ is smallest convex set that contain $\btext{a}_0, \btext{a}_1, \ldots, \btext{a}_n$ i.e. any other convex set $C$ that contains $\btext{a}_0, \btext{a}_1, \ldots, \btext{a}_n$ must also contain $\sigma_n$. Therefore $\btext{x} \in \sigma_n$. Thus, $\bigcap \mcal{C} \subseteq \sigma_n$.
\end{proof}

\begin{theorem}\label{t:Uniqness of GI set that spaning simplex}
Given a simplex $\sigma_n$, there is one and only one geometrically independent set of points spanning $\sigma_n$. 
\end{theorem}

\begin{proof}
An $n$-simplex $\sigma_n$ is the convex hull of $n+1$ geometrically independent points $A = \{\btext{a}_0, \btext{a}_1, \ldots, \btext{a}_n\} \subset \R^N$. A set of points $A = \{\btext{a}_0, \btext{a}_1, \ldots, \btext{a}_n\}$ is geometrically independent if the vectors $\{\btext{a}_1 - \btext{a}_0, \btext{a}_2 - \btext{a}_0, \ldots, \btext{a}_n - \btext{a}_0\}$ are linearly independent, meaning they span an $n$-dimensional affine subspace. We aim to show that the $n + 1$ points of $A$ that define $\sigma_n$ are unique up to relabeling.
\begin{enumerate}
\item \textbf{Any $n+1$ points spanning $\sigma_n$ must be geometrically independent:} Assume $\sigma_n = \text{conv}(A)$ is defined by $A = \{\btext{a}_0, \btext{a}_1, \ldots, \btext{a}_n\}$. If $A$ were not geometrically independent, at least one point in $A$, say $a_k$, could be written as a linear combination of the other points
\[ \btext{a}_k = \sum_{\substack{i=0 \\ i \neq k}}^n t_i \btext{a}_i, \, \text{where} \, \sum_{\substack{i=0 \\ i \neq k}}^n t_i = 1, \, t_i \geq 0\]
In this case, $\btext{a}_k$ would lie within the convex hull of the remaining points. Hence, $A$ would not span the entire simplex $\sigma_n$, contradicting the definition of an $n$-simplex. Therefore, $A$ must consist of $n+1$ geometrically independent points.

\item \textbf{Any two geometrically independent sets of points spanning $\sigma_n$ must be identical:}  Suppose $B = \{\btext{b}_0, \btext{b}_1, \ldots, \btext{b}_n\}$ is another set of $n+1$ points that spans $\sigma_n$. Since both $A$ and $B$ define the same simplex $\sigma_n$, every point $x \in \sigma_n$ can be uniquely expressed as a convex combination of the points in $A$
\[\btext{x} = \sum_{i=0}^n t_i \btext{b}_i, \, t_i \geq 0, \, \sum_{i=0}^n t_i = 1\]
Similarly, $\btext{x}$ can be uniquely expressed as a convex combination of the points in $B$
\[\btext{x} = \sum_{i=0}^n s_i \btext{b}_i, \, s_i \geq 0, \, \sum_{i=0}^n s_i = 1\]
Since the representation of $\btext{x}$ is unique, the sets $A$ and $B$ must define the same affine subspace of $\R^N$. Because both sets contain $n+1$ points, and the affine hull of $n+1$ geometrically independent points is unique, it follows that $A$ and $B$ must represent the same set of points up to relabeling.
\end{enumerate}
The $n$-simplex $\sigma_n$ is spanned by exactly one geometrically independent set of $n+1$ points, up to relabeling.
\end{proof}

\begin{theorem}[Properties of Interior of Simplex]\label{t:Properties of Interior of Simplex}
Let $n$-simplex $\sigma_n$ is spanned by the geometrically independent set $A = \{\btext{a}_0, \btext{a}_1, \ldots, \btext{a}_n\}$ in $\R^N$. The set $\operatorname{Int}(\sigma_n)$ is convex and open in the $n$-plane $P_n$. The closure of $\operatorname{Int}(\btext{a}_n)$ is denoted by $\mathrm{Cl}(\operatorname{Int}(\btext{a}_n))$ is $\sigma_n$. Furthermore,  $\operatorname{Int}(\btext{a}_n)$ equal to the union of all open line segment joining $\btext{a}_0$ to points of $\operatorname{Int}(\sigma_k)$, where $\sigma_k$ $(0 \leq k < n)$ is the face of $\sigma_n$ opposite $\btext{a}_0$.
\end{theorem}

\begin{proof}
Let $n$-simplex $\sigma_n$ be spanned by the geometrically independent set $A = \{\mathbf{a}_0, \mathbf{a}_1, \ldots, \mathbf{a}_n\}$ in $\mathbb{R}^N$.
\begin{enumerate}
\item \textbf{$\operatorname{Int}(\sigma_n)$ is convex:} As we know that a set is convex if for any two points in the set, the line segment joining them is also contained in the set. By Definition, the interior of the simplex consists of points that can be written as a convex combination with strictly positive barycentric coordinates
\[\btext{x} = \sum_{i=0}^{n} t_i \btext{a}_i, \, \text{where} \, t_i > 0 \, \text{and} \, \sum_{i=0}^{n} t_i = 1 \]
Take two points $\btext{x}, \btext{y} \in \mrm{Int}(\sigma_n)$, written as
\[\btext{x} = \sum_{i=0}^{n} t_i \btext{a}_i, \, y = \sum_{i=0}^{n} s_i \btext{a}_i \]
where $t_i > 0$ and $s_i > 0$ for all $i$. 

Consider the convex combination $\btext{z} = \lambda \btext{x} + (1 - \lambda)\btext{y}$ for $\lambda \in (0,1)$.
Substituting $\btext{x}$ and $\btext{y}$, we get
\[\btext{z} = \sum_{i=0}^{n} (\lambda t_i + (1 - \lambda) s_i) \btext{a}_i \]

Since $\lambda t_i + (1 - \lambda) s_i > 0$ for all $i$, it follows that $\btext{z}$ is still an interior point. Thus, $\mrm{Int}(\sigma_n)$ is convex.
\item \textbf{$\mrm{Int}(\sigma_n)$ is open in the $n$-plane $P_n$:} The simplex $\sigma_n$ is contained in the affine subspace $P_n$ of $\mbb{R}^N$ defined by the affine hull of $A$. The plane $P_n$ is a topological subspace of $\mbb{R}^N$. The interior of $\sigma_n$ is the set of points with strictly positive barycentric coordinates. Since the barycentric coordinate functions are continuous, the set where all coordinates are strictly positive is an open set in $P_n$. Thus, by Definition, $\mrm{Int}(\sigma_n)$ is open in $P_n$.
\item \textbf{ Closure of $\mrm{Int}(\sigma_n)$ is $\sigma_n$:} As we know that any boundary point of $\sigma_n$ has at least one barycentric coordinate equal to $0$. A sequence in $\mrm{Int}(\sigma_n)$ can approach any point in $\sigma_n$ by letting one of the barycentric coordinates tend to $0$. Since every boundary point can be approximated by a sequence in $\mrm{Int}(\sigma_n)$, it follows that
\[\mrm{Cl}(\mrm{Int}(\sigma_n)) = \sigma_n\]
\item \textbf{Expressing $\mrm{Int}(\sigma_n)$ in terms of faces:} Consider the face opposite to $\btext{a}_0$, denoted by $\sigma_k$ where $0 \leq k < n$. The interior of $\sigma_n$ consists of points of the form
\[\btext{x} = (1 - \lambda) \btext{a}_0 + \lambda \btext{y}, \, \text{where} \, \btext{y} \in \mrm{Int}(\sigma_k), \, 0 < \lambda < 1\]

The above expression shows that each interior point is obtained by taking an open segment from $\btext{a}_0$ to some point in the interior of $\sigma_k$. Since every face $\sigma_k$ is itself a simplex, we can apply this argument recursively until reaching $k = 0$, which consists of single points. Thus, $\mrm{Int}(\sigma_n)$ is precisely the union of all open line segments joining $\btext{a}_0$ to points in $\mrm{Int}(\sigma_k)$.
\end{enumerate}
\end{proof}

\begin{theorem}\label{t:}
Let $n$-simplex $\sigma_n$ is spanned by the geometrically independent set $A = \{\btext{a}_0, \btext{a}_1, \ldots, \btext{a}_n\}$ in $\R^N$ and $\mbb{B}^n$ is unit ball in $\R^n$. There is a homeomorphism of $\sigma_n$ with the unit ball $\mbb{B}^n$ that carries $\operatorname{Bd}(\sigma_n)$ onto the unit sphere $\mbb{S}^{n-1}$. 
\end{theorem}

\begin{proof}
Let $\sigma_n$ be the $n$-simplex in $\mathbb{R}^N$ spanned by the geometrically independent set $A = \{\btext{a}_0, \btext{a}_1, \ldots, \btext{a}_n\}$. Let $\mbb{B}^n$ be the unit ball in $\mbb{R}^n$. Then our aim is to show that there exists a homeomorphism  $\Phi \colon \sigma_n \to \mbb{B}^n$
that carries the boundary $\mrm{Bd}(\sigma_n)$ onto the unit sphere $\mbb{S}^{n-1}$ in following steps. 

\begin{enumerate}
\item \textbf{Define the Barycentric Coordinates of $\sigma_n$:} By Definition, any point $\btext{x} \in \sigma_n$ can be uniquely written as a convex combination of its vertices
\[\btext{x} = \sum_{i=0}^{n} t_i \btext{a}_i, \, \text{where} \, \sum_{i=0}^{n} t_i = 1, \, t_i \geq 0 \, \text{for all} \, i\]
Where the coefficients $\{t_0, t_1, \dots, t_n\}$ are known as the barycentric coordinates of $x$ with respect to $A$. Since $\sum_{i=0}^{n} t_i = 1$, we can eliminate $t_0$ and express points using $t_1, \dots, t_n$
\[t_0 = 1 - \sum_{i=1}^{n} t_i, \, 0 \leq t_i \leq 1, \, \sum_{i=1}^{n} t_i \leq 1\]
Thus, $\sigma_n$ can be seen as a subset of $\mbb{R}^n$
\[\left\{ (t_1, t_2, \dots, t_n) \colon 0 \leq t_i \leq 1, \, \sum_{i=1}^{n} t_i \leq 1 \right\}\]

\item \textbf{Mapping $\sigma_n$ to the Unit Ball $\mbb{B}^n$:} Let us define $\Phi \colon \sigma_n \to \mbb{B}^n$ by
\[\Phi(\btext{x}) = (s_1, s_2, \dots, s_n), \,\text{where} \, s_i = 2t_i - 1, \, \text{for} \, i = 1, 2, \dots, n\]
Let us see that map $\Phi$ is well defined. The function $\Phi$ takes a point $\btext{x} \in \sigma_n$ and maps it to a point in $\mbb{B}^n$ using the barycentric coordinates $t_i$ of $\btext{x}$ with respect to the vertices $\{\btext{a}_0, \btext{a}_1, \dots, \btext{a}_n\}$. The transformation $s_i = 2t_i - 1$ shifts and rescales each $t_i$, mapping the standard barycentric coordinate system of $\sigma_n$ to the Euclidean coordinates of the unit ball. A function is well-defined if every input $\btext{x} \in \sigma_n$ produces a unique output $\Phi(\btext{x}) \in \mbb{B}^n$. 

Since $t_i$ are barycentric coordinates, they satisfy $0 \leq t_i \leq 1$ and $\sum_{i=0}^{n} t_i = 1$ for all $i = 1, 2 \ldots, n$. Since $0 \leq t_i \leq 1$ ensure that  $-1 \leq s_i = 2t_i - 1 \leq 1$. Therefore, the transformation $s_i = 2t_i - 1$ for all $i = 1, 2 \ldots, n$ ensures that each $s_i$ lies in the interval $[-1,1]$. 
Since $s_i = 2t_i - 1$, therefore 
\begin{align*}
\sum_{i=0}^{n} s_i^2 & = \sum_{i=0}^{n} (2t_i - 1) ^2 \\ 
& \leq 4 \sum_{i=1}^{n}t_{i}^2 + 1 - 4\sum_{i=0}^{n}t_{i} \\ 
& \leq 4 \cdot 1 + 1 - 4\ \cdot 1  \\ 
& \leq 1
\end{align*}
Thus, $s_{i} \in [-1, 1]$ and $\sum_{i=0}^{n} s_i^2 \leq 1 $ for all $i = 1, \ldots, n$, therefore by Definition of unit ball $\mbb{B}^n$, $(s_{1}, s_{2}, \ldots, s_{n}) \in \mbb{B}^n$ and finally by Definition of $\Phi$, $\Phi(\btext{x}) \in \mbb{B}^n$. 

\item \textbf{$\Phi$ is a Homeomorphism:} Since $\Phi$ is linear by Definition, so it is continuous. If $\Phi(\btext{x}) = \Phi(\btext{y})$, then $s_i(\btext{x}) = s_i(\btext{y})$ for all $i$, implying $t_i(\btext{x}) = t_i(\btext{y})$. Hence, $\btext{x} = \btext{y}$. Thus, $\Phi$ is injective.  Any $(s_1, \dots, s_n) \in \mbb{B}^n$ maps back to $(t_1, \dots, t_n)$ i.e. $\Phi^{-1} \colon \mbb{B}^n \to \sigma_n$ defined by
\[ t_i = \frac{s_i + 1}{2}, \, \sum_{i = 1}^n t_i \leq 1\]
Therefore, $\Phi$ is surjective. The inverse function $\Phi^{-1}$ is also linear, so it is continuous. Thus, $\Phi$ is a continuous bijection between compact spaces, it is a homeomorphism.
\item \textbf{ Mapping the Boundary $\operatorname{Bd}(\sigma_n)$ to the Unit Sphere $\mathbb{S}^{n-1}$:}
The boundary consists of points where at least one $t_i = 0$. This implies some $s_i = -1$ or $1$, placing the image on the unit sphere
\[\sum_{i=1}^{n} s_i^2 = 1\]
Thus, $\Phi(\operatorname{Bd}(\sigma_n)) = \mbb{S}^{n-1}$.
\end{enumerate}
\end{proof}

\begin{definition}[Ray from a Point in $\R^n$]\label{d:Ray from a Point in Rn}
A ray $\mfr{R}_{\btext{w} \to \btext{p}}$ emanating from a point $\btext{w} \in \R^n$ is the set of all points of the form $\btext{w} + t\btext{p}$, where $\btext{p}$ (direction vector) is fixed point of $\R^n - \{\btext{0}\}$ and $t$ ranges over the non-negative reals that is a scalar parameter controls the distance from $\btext{w}$ along the direction $\btext{p}$. Symbolically, 
\begin{equation}\label{eq:Ray from a Point in Rn}
\mfr{R}_{\btext{w} \to \btext{p}} = \{\btext{x} \in \R^n \colon \btext{x} = \btext{w} + t\btext{p}, \btext{p} \in \R^n - \{\btext{0}\}, t \geq 0\}   
\end{equation}
\end{definition}

\begin{example}[Ray from a Point in $\R^2$]\label{d:Ray from a Point in Rn}
Find the ray $\mfr{R}_{\btext{w} \to \btext{p}}$ in $\R^2$ from a point $\btext{w} = (1, 2)$ in the direction vector $\btext{p} = (3, 1)$. Interpreted geometrically. 

By Definition of ray \eqref{eq:Ray from a Point in Rn} 
\[\btext{x} = \btext{w} + t\btext{p} = (1, 2) + t(3,1 ) = (1 + 3t, 2 + t)\]
Therefore 
\[\mfr{R}_{\btext{w} \to \btext{p}} = \{(1 + 3t, 2 + t) \colon t \geq 0\}\]

For different values of $t\geq 0$, we can compute points on the ray. For $t = 0, 1, 0.5, 2$ we get $\btext{x}_{1}= (1, 2), \btext{x}_{2}= (4, 3), \btext{x}_{3}= (2.5, 2.5), \btext{x}_{4}= (7, 4)$ respectively. 

If we plot these points $\btext{x}_1, \btext{x}_2, \btext{x}_3$, and $\btext{x}_4$, and so on, we get a straight line starting at $\btext{w} = (1,2)$ and extending infinitely in the direction of the vector $\btext{p} = (3,1)$. 
\end{example}

Now we are going to discuss about the properties of convex subsets of $\R^n$. In the next result we will find the condition under which any subset $U$ of $\R^n$ is convex in the term of a ray from an arbitrary point in $U$. Furthermore we will also discuss the homeomorphism between the closure of convex subset $\overline{U}$ of $\R^n$ and the unit ball $\mbb{B}^{n}$.

\begin{lemma}[Property of Convex Subset of $\R^n$]\label{l:Property of Convex Subset of R^n}
Let $\btext{w}$ be a  arbitrary point in bounded and open subset $U$ of $\R^n$. If $U$ is convex,  then each ray emanating from the point $\btext{w}$ intersect at single point with the boundary of the $U$. Converse is also true. 
\end{lemma}

\begin{proof}
As given that the subset $U$ of $\R^n$ is convex, bounded and open, therefore from their definitions, we can say  if $\btext{x}$ and $\btext{y} \in U$, then entire line joining these two points is contained in $U$; there exists a finite radius $\epsilon > 0$ such that $U \subset B_{\btext{0}}(\epsilon)$, where $B_{\btext{0}}(\epsilon)$ is ball centered at origin with radius $\epsilon$; and $U$ does not contain boundary i.e. $\operatorname{Bd} U$ consist of limit points that are not in $U$, respectively. 

Let us consider a ray $\mfr{R}_{\btext{w} \to \btext{p}}$ emanating from arbitrary point $\btext{w} \in U$ in the direction of point $\btext{p}$, then 
\[\mfr{R}_{\btext{w} \to \btext{p}} = \{\btext{w} + t\btext{p} \colon t \geq 0\}\]

Since $U$ is bounded, then ray $\mfr{R}_{\btext{w} \to \btext{p}}$ cannot extended indefinitely  within $U$; instead, it must extend $U$ at some finite point on its boundary. 

First we establish the existence of intersection of ray $\mfr{R}_{\btext{w} \to \btext{p}}$ with the boundary $\operatorname{Bd}(U)$. Let us define 
\[t^{\ast} = \sup \{t \geq 0 \colon \btext{w} + t \btext{p} \in U\}\]

Since $U$ is bounded so $t^*$ is finite. Since $U$ is open, we claim that $\btext{w} + t \btext{p} \in \operatorname{Bd}(U)$. There are two cases arise as follows 
\begin{enumerate}
\item If  $\btext{w} + t^\ast \btext{p} \in U$, then by hypothesis, there exists small neighborhood around the point $\btext{w} + t \btext{p}$ is still in $U$, that contradict the definition of $t^\ast$ as a supremum. 
\item If $\btext{w} + t^\ast \btext{p} \notin U$, then it must be a limit point of $U$ i.e. $\btext{w} + t^\ast \btext{p} \in \operatorname{Bd}(U)$. 
\end{enumerate}
Thus every ray $\mfr{R}_{\btext{w} \to \btext{p}}$ must intersect $\operatorname{Bd}(U)$ at $\btext{w} + t^\ast \btext{p}$. 

Now we show uniqueness of the intersection point
$\btext{w} + t^\ast \btext{p}$ in $\operatorname{Bd}(U)$. Let us assume that ray $\mfr{R}_{\btext{w} \to \btext{p}}$ intersect $\operatorname{Bd} U$ at two distinct point $\btext{x}_1 = \btext{w} + t_1 \btext{p}$ and $\btext{x}_2 = \btext{w} + t_2 \btext{p}$ in $\operatorname{Bd} U$ with $t_1 < t_2$. Therefore, convexity implies that line segment joining these two points is entirely contain in $U$. But $\btext{x}_1$ would be interior point, contradict the assumption $\btext{x}_1 \in \operatorname{Bd} U$. Hence intersection of ray with boundary is unique. 

Conversely, every ray emanating from $\btext{w} \in U$ intersect the boundary $\operatorname{Bd} U$ at unique point, then our aim is to show that $U$ is convex i.e. we show that for any arbitrary $\btext{a}, \btext{b} \in U$, then entire line segment joining these two points must contained in $U$. Therefore, by definition for $t \in (0, 1)$, there exists  $\btext{x}_t = (1 - t) \btext{a} + t \btext{b} \in U$.  

Let us assume $U$ is not convex. Then, there exists two points 
$\btext{a}, \btext{b} \in U$ such that line segment joining these two points is not entirely contained in $U$ i.e. there exists $t^\ast \in (0, 1)$ such that $\btext{x}_{t^\ast} = (1 - t^\ast)\btext{a} + t^\ast \btext{b} \notin U$. Since $U$ is open, the point $\btext{x}_{t^\ast}$, must outside the $U$ i.e. $\btext{x}_{t^\ast} \in \operatorname{Bd}(U)$. 

Consider the ray $\mfr{R}_{\btext{w} \to \btext{p}}$ given by $\btext{x}(t) = (1 - t) \btext{a} + t\btext{b}, t \geq 0$. As given every ray from any point in $U$ must intersect $\operatorname{Bd}(U)$ at a unique point, we must check wether this happen. As we already shown that $\btext{x}_{t^\ast} \in \operatorname{Bd}(U)$ i.e. ray intersect with $\operatorname{Bd}(U)$ at $\btext{x}_{t^\ast}$. However if $t^\ast$ is not the only such intersection, we get a contradiction if $U$ is not convex, some point of the ray could re-enter $U$ before reaching $\operatorname{Bd}(U)$. This is against the uniqueness condition in the hypothesis. Those our assumption is wrong, therefore $U$ is convex. 
\end{proof}

The above Lemma \eqref{l:Property of Convex Subset of R^n}  is the necessary and sufficient condition for any bounded and open subset $U$ of $\R^n$ to be a convex in term of ray. That is, helps us to determine the convexity of bounded and open subset $U$ of $\R^n$ in terms of ray from any point in $U$. 

\begin{lemma}[Homeomorphism between Closure of Convex Set with Unit Ball]\label{l:Homemorphism between Closure of Convex Set with Unit Ball}
Let $U$ be a convex, bounded, and open subset in $\R^n$. Then there is homeomorphism between the closure of $U$, $\mrm{Cl}(U)$ with $\mbb{B}^n$ that carry the boundary of $U$ onto the unit sphere $\mbb{S}^{n - 1}$.
\end{lemma}
\begin{remark}
In case of star-convexity relative to origin $\btext{0}$ of the subset $U$ of $\R^n$, a ray from $\btext{0}$ may  intersect $\operatorname{Bd}(U)$ in more than one point and $\mrm{Cl}(U)$ need not be homeomorphic to $\mbb{B}^n$. 
\end{remark}

\section{Simplicial Complex in $\R^N$}\label{s:Simplicial Complex in RN}
A simplicial complex is a fundamental concept in algebraic topology that provides a way to study topological spaces through simpler geometric objects called simplices. A simplex is a generalization of a triangle, and it can exist in any dimension. For example, a 0-simplex is a point, a 1-simplex is a line segment, a 2-simplex is a triangle, and a 3-simplex is a tetrahedron. A simplicial complex is essentially a collection of these simplices that are glued together in a specific way, where every face of a simplex is also included in the complex, and the intersection of any two simplices is itself a simplex (or possibly empty).

Simplicial complexes serve as a powerful tool for approximating and studying the properties of more complex topological spaces. They are used to define and compute various topological invariants, such as homology groups, which help classify spaces based on their shape and connectivity. Through the use of simplicial complexes, it becomes possible to break down and understand intricate topological spaces in terms of simpler, discrete elements.

\begin{definition}[Simplicial Complex in $\R^N$]\label{d:Simplicial Complex in RN}
A non-empty collection $\mcal{K}$ of simplices in $\R^N$ is called a simplicial complex if the following conditions are met:
\begin{enumerate}
\item \textbf{Closure under faces:} If $\sigma \in \mcal{K}$, then every face of $\sigma$ is also in $\mcal{K}$.
\item \textbf{Intersection of simplices:} If $\sigma, \tau \in \mcal{K}$, then  $\sigma \cap \tau$ is either empty or a face of both $\sigma$ and $\tau$.
\end{enumerate}
\end{definition}

\begin{definition}[Dimension of Simplicial Complex]\label{d:Dimension of Simplicial Complex}
The dimension of a simplicial complex $\mcal{K}$ is largest dimension of its simplices i.e. the highest number $n$ such that there exists at least one 
$n$-simplex in $\mcal{K}$. The dimension of $\mcal{K}$ is denoted by $\dim(\mcal{K})$.
\end{definition}

\begin{example}[Simplicial Complex in \(\mathbb{R}\)]\label{eg:Simplicial Complex in R}
Is the set \( A = \{\btext{a}_0, \btext{a}_1\} \) a geometrically independent set in \(\mathbb{R}\), and does the collection of simplices
\[
\mcal{K} = \{\{\btext{a}_0\}, \{\btext{a}_1\}, \{\btext{a}_0, \btext{a}_1\}\}
\]
form a simplicial complex of dimension 1?
\end{example}

\begin{solution*}
The collection $\mcal{K}$ has the following simplices: 
\begin{enumerate}
\item \textbf{0-simplexes:} $\sigma_{0}^{1} = \{\btext{a}_{0}\}$ and $\sigma_{0}^{2} = \{\btext{a}_{1}\}$.
\item \textbf{1-simplex:} $\sigma_{1}^{1} = \{\btext{a}_{0}, \btext{a}_{1}\}$.
\end{enumerate}

Let us verify the conditions for simplical complex as follows:
\begin{enumerate}
\item \textbf{Closure under faces:} There is only 1-simplex $\sigma_{1}^{1} = \{\btext{a}_{0}, \btext{a}_{1}\}$ in $\mcal{K}$ whose all faces $\sigma_{0}^{1} = \{\btext{a}_{0}\}, \sigma_{0}^{2} = \{\btext{a}_{1}\}$ are in $\mcal{K}$. 
\item \textbf{Intersection of simplices:} The intersection of any two simplexes of $\mcal{K}$ is either empty or face of each of them. For example:
\[\sigma_{0}^{1} \cap \sigma_{0}^{2} = \{\btext{a}_{0}\} \cap \{\btext{a}_{1}\} = \emptyset\]
\end{enumerate}

Therefore, $\mcal{K}$ is simplical complex. Since the largest dimension of simplex $\sigma_{1}^{1}$ is 1, therefore $\dim(\mcal{K}) = 1$. See the Figure \eqref{fig:1d-simplex}.
\end{solution*}

\begin{example}[Simplicial Complex in \(\mathbb{R}^2\)]\label{eg:Simplicial Complex in R2 1}
Is the set \( A = \{\btext{a}_0, \btext{a}_1, \btext{a}_2\} \) a geometrically independent set in \(\mathbb{R}^2\), and does the collection of simplices 
\[
\mcal{K} = \{\{\btext{a}_0\}, \{\btext{a}_1\}, \{\btext{a}_2\}, \{\btext{a}_0, \btext{a}_1\}, \{\btext{a}_1, \btext{a}_2\}, \{\btext{a}_2, \btext{a}_0\}, \{\btext{a}_0, \btext{a}_1, \btext{a}_2\}\}
\]
form a simplicial complex of dimension 2?
\end{example}

\begin{solution*}
The collection $\mcal{K}$ has the following simplices: 
\begin{enumerate}
\item \textbf{0-simplices:} $\sigma_{0}^{1} = \{\btext{a}_{0}\}, \sigma_{0}^{2} = \{\btext{a}_{1}\}$ and $\sigma_{0}^{3} = \{\btext{a}_{1}\}$.
\item \textbf{1-simplices:} $\sigma_{1}^{1} = \{\btext{a}_{0}, \btext{a}_{1}\}, \sigma_{1}^{2} = \{\btext{a}_{1}, \btext{a}_{2}\}$ and $\sigma_{1}^{3} = \{\btext{a}_{2}, \btext{a}_{0}\}$.
\item \textbf{2-simplex:} $\sigma_{2}^{1} = \{\btext{a}_{0}, \btext{a}_{1}, \btext{a}_{2}\}$
\end{enumerate}

Let us verify the conditions for simplical complex as follows:
\begin{enumerate}
\item \textbf{Closure under faces:} All faces of a 2-simplex are 1-simplices that lies in $\mcal{K}$. All faces of each 1-simplices are 0-simplices that lies in $\mcal{K}$.  
\item \textbf{Intersection of simplices:} The intersection of any two simplexes of $\mcal{K}$ is either empty or face of each of them. For example:
\begin{align*}
\sigma_{0}^{1} \cap \sigma_{0}^{2} & = \{\btext{a}_{0}\} \cap \{\btext{a}_{1}\} = \emptyset, \\
\sigma_{1}^{1} \cap \sigma_{1}^{2} & = \{\btext{a}_{0}, \btext{a}_{1}\} \cap \{ \btext{a}_{1}, \btext{a}_{2}\} =  \{\btext{a}_{1}\} = \text{face of}\, \sigma_{1}^{1}
\end{align*}
\end{enumerate}

Therefore, $\mcal{K}$ is simplical complex. Since the largest dimension of simplex $\sigma_{2}^{1}$ is 2, therefore $\dim(\mcal{K}) = 2$. This complex is filled triangle. See the Figure \eqref{fig:2d-simplex}.
\end{solution*}

\begin{example}[Simplicial Complex in $\R^3$]\label{eg:Simplicial Complex in R3}
Let $A = \{\btext{a}_0, \btext{a}_1, \btext{a}_2, \btext{a}_3\}$ is geometrical independent set in $\R^3$, then the collection of simplices 
\begin{multline*}\label{eq:Simplicial Complex in R3}
\mathcal{K} = \{
\{\mathbf{a}_0\}, \{\mathbf{a}_1\}, \{\mathbf{a}_2\}, \{\mathbf{a}_3\},  % 0-simplices
\{\mathbf{a}_0, \mathbf{a}_1\}, \{\mathbf{a}_0, \mathbf{a}_2\}, \{\mathbf{a}_0, \mathbf{a}_3\}, \{\mathbf{a}_1, \mathbf{a}_2\}, \{\mathbf{a}_1, \mathbf{a}_3\}, \{\mathbf{a}_2, \mathbf{a}_3\}, \\ % 1-simplices
\{\mathbf{a}_0, \mathbf{a}_1, \mathbf{a}_2\}, \{\mathbf{a}_0, \mathbf{a}_1, \mathbf{a}_3\}, \{\mathbf{a}_0, \mathbf{a}_2, \mathbf{a}_3\}, \{\mathbf{a}_1, \mathbf{a}_2, \mathbf{a}_3\},  % 2-simplices
\{\mathbf{a}_0, \mathbf{a}_1, \mathbf{a}_2, \mathbf{a}_3\} % 3-simplex
\}.
\end{multline*}
is simplicial complex of dimension 3. This complex is filled tetrahedron.  See the Figure \eqref{fig:3d-simplex}.
\end{example}

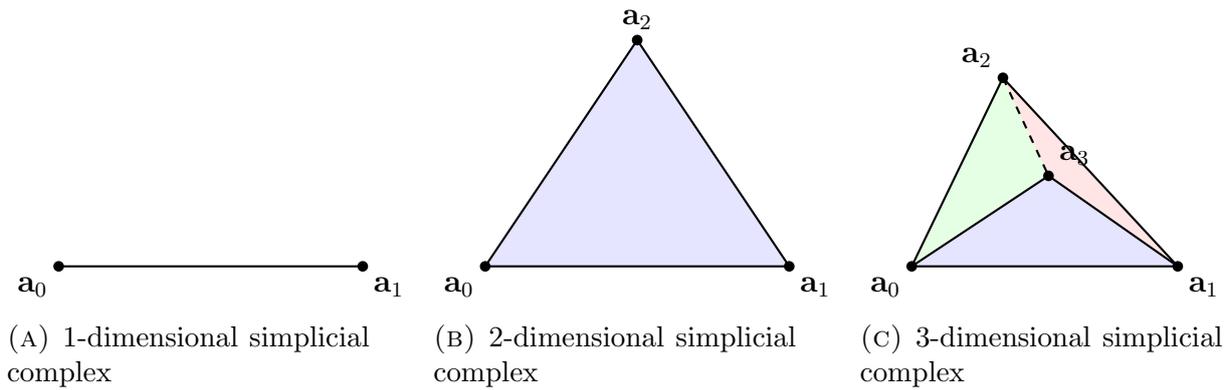
\begin{figure}[h!]
    \centering
    \begin{subfigure}{0.3\textwidth}
        \centering
        \begin{tikzpicture}
            % Define points
            \coordinate (A) at (0,0);
            \coordinate (B) at (4,0);
            % Draw filled line segment (1-simplex)
            \fill[blue!20, opacity=0.5] (A) -- (B);   
            % Draw edges (1-simplices)
            \draw[thick] (A) -- (B);
            % Draw vertices (0-simplices)
            \fill[black] (A) circle (2pt) node[below left] {$\mathbf{a}_0$};
            \fill[black] (B) circle (2pt) node[below right] {$\mathbf{a}_1$};
        \end{tikzpicture}
        \caption{1-dimensional simplicial complex}
        \label{fig:1d-simplex}
    \end{subfigure}
    \hfill
    \begin{subfigure}{0.3\textwidth}
        \centering
        \begin{tikzpicture}
            % Define points
            \coordinate (A) at (0,0);
            \coordinate (B) at (4,0);
            \coordinate (C) at (2,3);
            % Draw filled triangle (2-simplex)
            \fill[blue!20, opacity=0.5] (A) -- (B) -- (C) -- cycle;   
            % Draw edges (1-simplices)
            \draw[thick] (A) -- (B);
            \draw[thick] (B) -- (C);
            \draw[thick] (C) -- (A);
            % Draw vertices (0-simplices)
            \fill[black] (A) circle (2pt) node[below left] {$\mathbf{a}_0$};
            \fill[black] (B) circle (2pt) node[below right] {$\mathbf{a}_1$};
            \fill[black] (C) circle (2pt) node[above] {$\mathbf{a}_2$};  
        \end{tikzpicture}
        \caption{2-dimensional simplicial complex}
        \label{fig:2d-simplex}
    \end{subfigure}
    \hfill
    \begin{subfigure}{0.3\textwidth}
        \centering
        \begin{tikzpicture}
            % Vertices (same coordinate scale as 1st/2nd figures)
            \coordinate (A0) at (0,0);        % front-left
            \coordinate (A1) at (3.5,0);      % front-right  
            \coordinate (A2) at (1.2,2.5);    % back
            \coordinate (A3) at (1.8,1.2);    % top (projected)

            % Filled visible faces (2-simplices) - matching 1st/2nd style
            \fill[blue!20, opacity=0.5] (A0) -- (A1) -- (A3) -- cycle;  % front
            \fill[green!20, opacity=0.5] (A0) -- (A2) -- (A3) -- cycle; % left
            \fill[red!20, opacity=0.5] (A1) -- (A2) -- (A3) -- cycle;   % right

            % Edges (1-simplices)
            \draw[thick] (A0) -- (A1);
            \draw[thick] (A0) -- (A2);
            \draw[thick] (A1) -- (A2);
            \draw[thick] (A0) -- (A3);
            \draw[thick] (A1) -- (A3);
            \draw[thick,dashed] (A2) -- (A3);  % hidden edge

            % Vertices (0-simplices) - SAME size as 1st/2nd (2pt)
            \fill[black] (A0) circle (2pt) node[below left] {$\mathbf{a}_0$};
            \fill[black] (A1) circle (2pt) node[below right] {$\mathbf{a}_1$};
            \fill[black] (A2) circle (2pt) node[above left] {$\mathbf{a}_2$};
            \fill[black] (A3) circle (2pt) node[above right] {$\mathbf{a}_3$};
        \end{tikzpicture}
        \caption{3-dimensional simplicial complex}
        \label{fig:3d-simplex}
    \end{subfigure}
    \caption{Simplicial Complexes of Increasing Dimension}
    \label{fig:simplicial-complexes}
\end{figure}

\begin{example}\label{eg:Two filled triangles with a common edge}
Let $A = \{\btext{a}_0, \btext{a}_1, \btext{a}_2, \btext{a}_3\}$ is geometrical independent set in $\R^2$, then the collection of simplices
\begin{multline*}%\label{eq:simplicialComplexInR3}
\mcal{K} = \{\{\btext{a}_0\}, \{\btext{a}_1\}, \{\btext{a}_2\}, \{\btext{a}_3\}, 
\{\btext{a}_0, \btext{a}_1\}, \{\btext{a}_1, \btext{a}_3\}, \{\btext{a}_3, \btext{a}_2\}, \{\btext{a}_2, \btext{a}_0\}, \{\btext{a}_2, \btext{a}_1\}, \\ 
\{\btext{a}_0, \btext{a}_1, \btext{a}_2\}, \{\btext{a}_1, \btext{a}_2, \btext{a}_3\}
\}        
\end{multline*}
is simplicial complex of dimension 2. See the Figure \eqref{fig:Two filled triangles with a common edge}.
\end{example}

\begin{example}\label{eg:Two filled triangles with a common vertex}
Let $A = \{\btext{a}_0, \btext{a}_1, \btext{a}_2, \btext{a}_3\}$ is geometrical independent set in $\R^2$, then the collection of simplices

\begin{multline*}%\label{eq:Simplicial Complex in R2 2b}
\mcal{K} = \{\{\btext{a}_0\}, \{\btext{a}_1\}, \{\btext{a}_2\}, \{\btext{a}_3\}, \{\btext{a}_4\}
\{\btext{a}_0, \btext{a}_1\}, \{\btext{a}_1, \btext{a}_2\}, \{\btext{a}_2, \btext{a}_0\}, \\ \{\btext{a}_2, \btext{a}_3\}, \{\btext{a}_3, \btext{a}_4\}, \{\btext{a}_4, \btext{a}_2\}  \}  \end{multline*}
is simplicial complex of dimension 2. See the Figure \eqref{fig:Two filled triangles with a common vertex}.
\end{example}

\begin{figure}[h!]
    \centering
    \begin{subfigure}{0.48\textwidth}
        \centering
        \begin{tikzpicture}
            % Define points
            \coordinate (A0) at (0,0);
            \coordinate (A1) at (4,0);
            \coordinate (A2) at (2,3);
            \coordinate (A3) at (6,3);
            % Draw filled triangles (2-simplices)
            \fill[red!20, opacity=0.3] (A0) -- (A1) -- (A2) -- cycle;
            \fill[green!20, opacity=0.5] (A1) -- (A2) -- (A3) -- cycle;
            % Draw edges (1-simplices)
            \draw[thick] (A0) -- (A1);
            \draw[thick] (A1) -- (A2);
            \draw[thick] (A2) -- (A0);
            \draw[thick] (A1) -- (A3);
            \draw[thick] (A2) -- (A3);
            % Draw vertices (0-simplices)
            \fill[black] (A0) circle (2pt) node[below left] {$\mathbf{a}_0$};
            \fill[black] (A1) circle (2pt) node[below right] {$\mathbf{a}_1$};
            \fill[black] (A2) circle (2pt) node[above left] {$\mathbf{a}_2$};
            \fill[black] (A3) circle (2pt) node[above right] {$\mathbf{a}_3$};
        \end{tikzpicture}
        \caption{Two filled triangles with a common edge}
        \label{fig:Two filled triangles with a common edge}
    \end{subfigure}
    \hfill
    \begin{subfigure}{0.48\textwidth}
        \centering
        \begin{tikzpicture}
            % Define points
            \coordinate (A0) at (0,0);
            \coordinate (A1) at (4,0);
            \coordinate (A2) at (2,3);
            \coordinate (A3) at (7,4);
            \coordinate (A4) at (4,2);
            % Draw filled triangles (2-simplices)
            \fill[blue!20, opacity=0.5] (A0) -- (A1) -- (A2) -- cycle;
            \fill[green!20, opacity=0.5] (A2) -- (A3) -- (A4) -- cycle;
            % Draw edges (1-simplices)
            \draw[thick] (A0) -- (A1);
            \draw[thick] (A1) -- (A2);
            \draw[thick] (A2) -- (A0);
            \draw[thick] (A2) -- (A3);
            \draw[thick] (A3) -- (A4);
            \draw[thick] (A4) -- (A2);   
            % Draw vertices (0-simplices)
            \fill[black] (A0) circle (2pt) node[below left] {$\mathbf{a}_0$};
            \fill[black] (A1) circle (2pt) node[below right] {$\mathbf{a}_1$};
            \fill[black] (A2) circle (2pt) node[above left] {$\mathbf{a}_2$};
            \fill[black] (A3) circle (2pt) node[above right] {$\mathbf{a}_3$};
            \fill[black] (A4) circle (2pt) node[below right] {$\mathbf{a}_4$};
        \end{tikzpicture}
        \caption{Two filled triangles with a common vertex}
        \label{fig:Two filled triangles with a common vertex}
    \end{subfigure}
    \caption{Two examples of 3-dimensional simplicial complexes}
    \label{fig:3d-simplicial-complexes}
\end{figure}
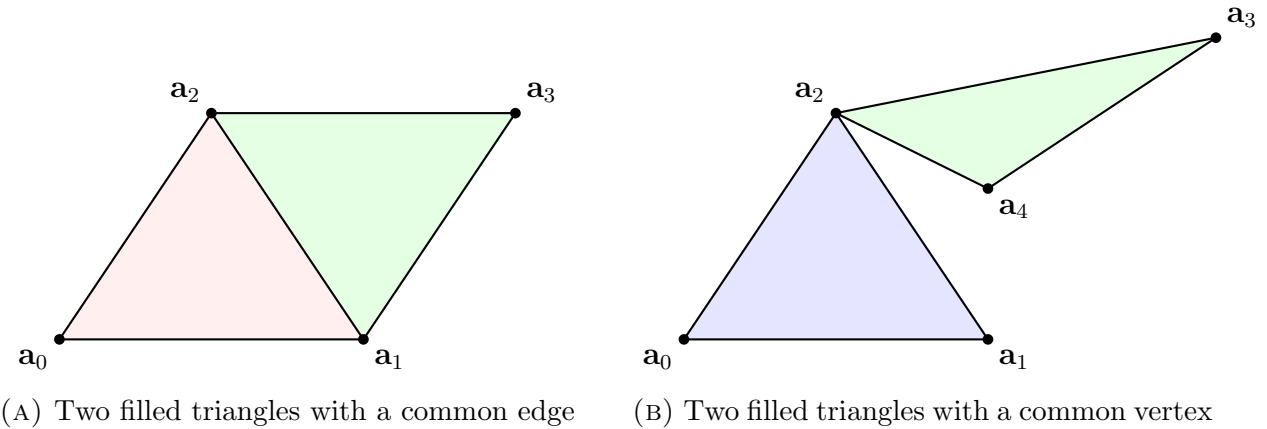

\begin{example}\label{eg:Non-Simplicial Complex in R2 2c}
Is the set \( A = \{\btext{a}_0, \btext{a}_1, \btext{a}_2, \btext{a}_3\} \), which is geometrically independent in \(\mathbb{R}^2\), such that the collection
\[
\mcal{K} = \left\{\{\btext{a}_0, \btext{a}_1, \btext{a}_2\}, \{\btext{a}_1, \btext{a}_2, \btext{a}_3\}, \{\btext{a}_0, \btext{a}_3\}\right\}
\]
forms a simplicial complex?
\end{example}

\begin{solution*}
The collection $\mcal{K}$ is not simplicial complex in $\R^2$. Since the 2-simplex as $ \{\btext{a}_0, \btext{a}_1, \btext{a}_2\} \in \mcal{K}$, but its face $\{\btext{a}_0, \btext{a}_1\} \notin \mcal{K}$. The 2-simplices $\{\btext{a}_0, \btext{a}_1, \btext{a}_2\}, \{\btext{a}_1, \btext{a}_2, \btext{a}_3\} \in \mcal{K}$, but $\{\btext{a}_0, \btext{a}_1, \btext{a}_2\} \cap \{\btext{a}_1, \btext{a}_2, \btext{a}_3\} = \{\btext{a}_1, \btext{a}_2\} \notin \mcal{K}$. Therefore, the collection $\mcal{K}$ does not satisfy the closure condition as well as the intersection condition. See the Figure \eqref{fig:Non-simplicial Complex}.

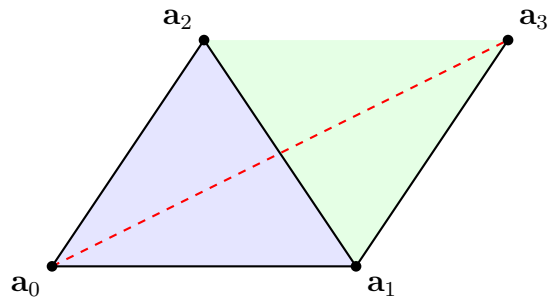
\begin{figure}[h!]
    \centering
\begin{tikzpicture}
    % Define points
    \coordinate (A0) at (0,0);
    \coordinate (A1) at (4,0);
    \coordinate (A2) at (2,3);
    \coordinate (A3) at (6,3);
    % Draw filled triangles (2-simplices) but missing some edges
    \fill[blue!20, opacity=0.5] (A0) -- (A1) -- (A2) -- cycle;
    \fill[green!20, opacity=0.5] (A1) -- (A2) -- (A3) -- cycle;
    % Draw edges (1-simplices)
    \draw[thick] (A0) -- (A1);  % Part of Triangle 1
    \draw[thick] (A1) -- (A2);  % Common edge
    \draw[thick] (A2) -- (A0);  % Part of Triangle 1
    \draw[thick] (A1) -- (A3);  % Part of Triangle 2
    % Draw an extra edge that breaks the simplicial complex condition
    \draw[thick, red, dashed] (A0) -- (A3);  % Arbitrary edge not part of any triangle
    % Draw vertices (0-simplices)
    \fill[black] (A0) circle (2pt) node[below left] {$\btext{a}_0$};
    \fill[black] (A1) circle (2pt) node[below right] {$\btext{a}_1$};
    \fill[black] (A2) circle (2pt) node[above left] {$\btext{a}_2$};
    \fill[black] (A3) circle (2pt) node[above right] {$\btext{a}_3$};
\end{tikzpicture}
   \caption{Non-simplicial Complex}
    \label{fig:Non-simplicial Complex}
\end{figure}
\end{solution*}

Since it is not an easy task to verify whether a given collection is simplicial complex or not, let us look at the next important Lemma \eqref{l:Necessary and sufficient conditions for any collection to be simplicial complex} which gives necessary and sufficient conditions for any collection to be simplicial complex which ultimately also helps in verification.

\begin{lemma}\label{l:Necessary and sufficient conditions for any collection to be simplicial complex}
A collection $\mcal{K}$ of simplicies is a simplicial complex if and only if following conditions hold:
\begin{enumerate}
\item \textbf{Closure under faces:} Every face of a simplex of $\mcal{K}$ is in $\mcal{K}$.
\item \textbf{Disjoint Interiors for distinct simplicies:} Every pair of distinct simplicies of $\mcal{K}$ have disjoint interiors. 
\end{enumerate}
\end{lemma}
\begin{proof}
Let us assume that $\mcal{K}$ is simplicial complex, then we would like to prove that $\mcal{K}$ satisfies both the given conditions.

The first condition is clearly satisfied by $\mcal{K}$ since it is the same as the definition of the simplicial complex. 

Now the second condition, if $\mcal{K}$ is simplicial complex, then we would like to prove that every pair of distinct simplices of $\mcal{K}$ has disjoint interiors i.e. if $\sigma, \tau \in \mcal{K}$ such that $\sigma \neq \tau$, then $\mrm{Int}(\sigma) \cap \mrm{Int}(\tau) = \emptyset$. We will prove its contrapositive statement which is: If $\btext{x} \in \mrm{Int}(\sigma) \cap \mrm{Int}(\tau)$ is arbitrary, then $\sigma = \tau$. As $\btext{x} \in \mrm{Int}(\sigma) \cap \mrm{Int}(\tau)$ is given, this implies that $\btext{x} \in \mrm{Int}(\sigma)$ i.e. $\btext{x} \in \sigma$. Let $\rho = \sigma \cap \tau$. Let $\rho \subset \sigma$ be a proper face of $\sigma$, then $\btext{x} \in \operatorname{Bd} \sigma$, but $\btext{x} \notin \operatorname{Bd} \sigma$ since $\btext{x} \in \mrm{Int}(\sigma)$. Therefore, it is not possible for $\rho$ to be proper face so $\rho = \sigma$. Similarly, we can prove that $\rho = \tau$. Thus,  $\rho = \sigma = \tau$. 

Conversely, if the collection $\mcal{K}$ of simplices holds both conditions $(1)$ and $(2)$, then we would like to prove that $\mcal{K}$ is simplicial complex. This will be proved directly from the definition of simplicial complex. This first condition of the simple complex clearly matches $(1)$ because both are the same. We will now prove that the intersection of any two simplexes  $ \sigma, \tau \in \mcal{K}$ is either empty of a face of each of them.

Let us suppose $\sigma \cap \tau \neq \emptyset$. Then we claim $\sigma \cap \tau = \rho$ is the face of $\sigma$ that is spanned by those vertices $\btext{b}_0, \btext{b}_1, \ldots, \btext{b}_m$ of $\sigma$ that lies in $\tau$. Since $\sigma \cap \tau$ is convex and contains the vertices $\btext{b}_0, \btext{b}_1, \ldots, \btext{b}_m$, which means that $\rho \subset \sigma \cap \tau$. Now let $\btext{x} \in \sigma \cap \tau$ be arbitrary, this means that $\btext{x} \in \mrm{Int}(\sigma) \cap \mrm{Int}(\tau)$ for some faces $\sigma'$ of $\sigma$ and $\tau'$ of $\tau$. But according to the given condition $(2)$, this implies that $\sigma' = \tau'$. Therefore, the vertices of $\sigma'$ are contained in $\tau'$ and they are elements of $\{\btext{b}_0, \btext{b}_1, \ldots, \btext{b}_m\}$. Thus, $\sigma'$ is a face of $\rho$ which means $\btext{x} \in \rho$ and this implies $\sigma \cap \tau \subset \rho$. Therefore, $\sigma \cap \tau = \rho$.
\end{proof}
Let us see following observation from the Lemma \eqref{l:Necessary and sufficient conditions for any collection to be simplicial complex}.

\begin{corollary}
If $\sigma$ is any simplex, the collection consisting of $\sigma$ and its proper faces is a simplicial complex. 
\end{corollary}

\begin{proof}
The condition $(1)$ of Lemma \eqref{l:Necessary and sufficient conditions for any collection to be simplicial complex} is obvious. Condition $(2)$ of Lemma \eqref{l:Necessary and sufficient conditions for any collection to be simplicial complex} satisfied  since for each element $\btext{x} \in \sigma$, there is exactly one face $\sigma'$ of $\sigma$ exist such that $\btext{x} \in \mrm{Int}(\sigma')$. 
\end{proof}

\begin{example}\label{eg:Verification of Lemma}
Let $A = \{\btext{a}_0, \btext{a}_1, \btext{a}_2\}$ is geometrical independent set in $\R^2$ and let the collection of all simplices (2-simplies, 1-simplices and 0-simplices) in $\R^2$ is the simplicial complex given by $\mcal{K}_1$
\[\mcal{K}_1 = \{\{\btext{a}_0\}, \{\btext{a}_1\}, \{\btext{a}_2\}, \{\btext{a}_0, \btext{a}_1\},     \{\btext{a}_1, \btext{a}_2\}, \{\btext{a}_2, \btext{a}_0\}, \{\btext{a}_0, \btext{a}_1, \btext{a}_2\} \}    \] 
is simplicial complex. Verify this example by the Lemma \eqref{l:Necessary and sufficient conditions for any collection to be simplicial complex}.
\end{example}
\begin{solution*}
According to Lemma \eqref{l:Necessary and sufficient conditions for any collection to be simplicial complex} we will verify both conditions as follows: 
\begin{enumerate}
\item \textbf{Closure under faces:}A simplicial complex must contain all faces (lower-dimensional simplices) of each simplex it contains. The 2-simplex $\{\btext{a}_0, \btext{a}_1, \btext{a}_2\}$ has the faces; 1-simplicies as $\{\btext{a}_0, \btext{a}_1\}, \{\btext{a}_1, \btext{a}_2\}, \{\btext{a}_2, \btext{a}_0\}$ and 1-simplicies as $\{\btext{a}_0\}, \{\btext{a}_2\}, \{\btext{a}_2\}$. All of these are explicitly listed in 
$\mcal{K}$, proving that $\mcal{K}$ is closed under taking faces.  
\item \textbf{Disjoint Interiors for distinct simplicies:}The interior of a simplex refers to the points inside it, excluding its lower-dimensional faces. The interiors of the edges (1-simplicies) as $\{\btext{a}_0, \btext{a}_1\}, \{\btext{a}_1, \btext{a}_2\}, \{\btext{a}_2, \btext{a}_0\}$  
do not contain any vertices or other edges. The interior of the triangle 2-simplex $\{\btext{a}_0, \btext{a}_1, \btext{a}_2\}$
does not include any edges (1-simplicies) or vertices (0-simplicies); it only includes the points inside the triangle. Since the interiors of different simplices are disjoint, this condition holds.
\end{enumerate}
Since both necessary conditions hold, we conclude that $\mcal{K}_1$ is simplicial complexes.
\end{solution*}

\begin{question}
Is the simplicial complex is topological space ? 
\end{question}
\subsection{Structures derived from Simplicial Complexes}\label{ss:Structures derived from Simplicial Complexes}
Now we are going to describe some important structures which are derived from simplicial complexes like sub-complexes, skeletons, vertices, underlying spaces, polytopes, polyhedrons.

\begin{definition}[Sub-complex of Simplicial Complex ]\label{d:Sub-complex of Simplicial Complex}
Let $\mcal{K}$ be a simplicial complex. A sub-complex $\mcal{L}$ of $\mcal{K}$ is a subset of  $\mcal{K}$ that is itself a simplicial complex.  
\end{definition}

\begin{example}[Subcomplex of a Simplicial Complex]\label{eg:Example of Sub-complex}
Let \(\mcal{K}\) be the simplicial complex in \(\mathbb{R}^2\) defined by
\[
\mcal{K} = \{\{\btext{a}_0\}, \{\btext{a}_1\}, \{\btext{a}_2\}, \{\btext{a}_0, \btext{a}_1\}, \{\btext{a}_1, \btext{a}_2\}, \{\btext{a}_2, \btext{a}_0\}, \{\btext{a}_0, \btext{a}_1, \btext{a}_2\}\}.
\]
Is the subset \(\mcal{L}\) of \(\mcal{K}\),
\[
\mcal{L} = \{\{\btext{a}_0\}, \{\btext{a}_1\}, \{\btext{a}_0, \btext{a}_1\}\},
\]
a subcomplex of \(\mcal{K}\)?
\end{example}

\begin{solution*}
A simplicial complex $\mcal{K}$ in $\R^2$ is
\[\mcal{K} = \{\{\btext{a}_0\}, \{\btext{a}_1\}, \{\btext{a}_2\}, \{\btext{a}_0, \btext{a}_1\},     \{\btext{a}_1, \btext{a}_2\}, \{\btext{a}_2, \btext{a}_0\}, \{\btext{a}_0, \btext{a}_1, \btext{a}_2\} \}    \]
and subset $\mcal{L}$ of $\mcal{K}$ is
\[\mcal{L} = \{\{\btext{a}_0\}, \{\btext{a}_1\}, \{\btext{a}_0, \btext{a}_1\} \}    \] 
that is itself a simplicial complex. Therefore, $\mcal{L}$ is sub-simplicial complex of $\mcal{K}$.
\end{solution*}

\begin{definition}[$p$-Skeleton of Simplicial Complex]\label{d:p keleton of Simplicial Complex}
The $p$-skeleton of a simplicial complex $\mcal{K}$, denoted as 
$\mcal{K}^{(p)}$, is the sub-complex of $\mcal{K}$ consisting of all simplices in $\mcal{K}$ whose dimension is at most $p$. Symbolically, 
\begin{equation}\label{eq :pskeleton}
\mcal{K}^{(p)} = \{\sigma \in \mcal{K} \colon \dim (\sigma) \leq p\}
\end{equation}
\end{definition}

\begin{definition}[Vertices of Simplicial Complex]\label{d:Vertices of Simplicial Complex}
The elements of collection $\mcal{K}^{(0)}$  are called vertices of $\mcal{K}$.
\end{definition}

\begin{example}[The $p$-Skeleton of a Simplicial Complex]\label{eg:Example of p skelton}
What are the skeletons of the simplicial complex \(\mcal{K}\) in \(\mathbb{R}^2\),
\[
\mcal{K} = \{\{\btext{a}_0\}, \{\btext{a}_1\}, \{\btext{a}_2\}, \{\btext{a}_0, \btext{a}_1\}, \{\btext{a}_1, \btext{a}_2\}, \{\btext{a}_2, \btext{a}_0\}, \{\btext{a}_0, \btext{a}_1, \btext{a}_2\}\}?
\]

Can we describe the 0-, 1-, and 2-skeletons \(\mcal{K}^{(0)}, \mcal{K}^{(1)}, \mcal{K}^{(2)}\), and identify the vertices of \(\mcal{K}\)?
\end{example}

\begin{solution*}
The given simplicial complex $\mcal{K}$ is
\[\mcal{K} = \{\{\btext{a}_0\}, \{\btext{a}_1\}, \{\btext{a}_2\}, \{\btext{a}_0, \btext{a}_1\},     \{\btext{a}_1, \btext{a}_2\}, \{\btext{a}_2, \btext{a}_0\}, \{\btext{a}_0, \btext{a}_1, \btext{a}_2\} \}   \]

According to Definition \eqref{d:p keleton of Simplicial Complex}, $\mcal{K}^{(p)} = \{\sigma \in \mcal{K} \colon \dim (\sigma) \leq p\}$. Therefore, 
\[\mcal{K}^{(0)} = \{\{\btext{a}_0\}, \{\btext{a}_1\}, \{\btext{a}_2\} \},
\mcal{K}^{(1)} = \{\{\btext{a}_0, \btext{a}_1\},     \{\btext{a}_1, \btext{a}_2\}, \{\btext{a}_2, \btext{a}_0\}\},
\mcal{K}^{(2)} = \mcal{K}.\]

According to Definition \eqref{d:Vertices of Simplicial Complex} elements of 0-skeleton $\mcal{K}^{(0)}$ is the vertices of $\mcal{K}$, therefore, the vertices of the simplicial complex $\mcal{K}$ are $\{\btext{a}_0\}, \{\btext{a}_1\}, \{\btext{a}_2\}$.
\end{solution*}

\subsection{Geometric Realization of Simplicial Complex}\label{ss:Geometric Realization of Simplicial Complex}
As we have already studied about simplicial complexs which are combinatorial structures as a collection of vertices and how they are grouped into simplices. Since simplicial complexes are discrete in nature, the question arises: how can we study the continuous properties of such structures? The answer is the geometric realization of simplicial complexes that serve as a bridge between discrete models and continuous shapes. The geometric realization of a simplicial complex gives a continuous topological space corresponding to its combinatorial structure. It bridges the gap between discrete and continuous mathematics, enabling the application of topological, geometric, and analytical techniques. By establishing a concrete space where points are convex combinations of vertices, it allows for the study of homotopy, homology, and topological invariants in algebraic topology. Additionally, it facilitates metric geometry, differential geometry, and computational methods such as finite element analysis and topological data analysis. The realization provides a natural embedding in Euclidean space, making abstract complexes more tangible and useful for approximating smooth manifolds and complex topological spaces. Overall, geometric realization serves as a crucial tool for analyzing and computing properties of spaces that originate from purely combinatorial simplicial structures.

The geometric realization of simplicial complexes has many popular or synonymous names such as underlying space, polyhedron, and simplicial space, but in our study we use only the geometric realization. If $\mcal{K}$ is a simplicial complex then its geometric realization is denoted by $|\mcal{K}|$.

\begin{definition}[Geometric Realization of Simplicial Complex]\label{d:Geometric Realization of Simplicial Complex}
The geometric realization of a simplicial complex $\mcal{K}$ is the subset of $\R^N$ that is union of simplices corresponding to each simplex in $\mcal{K}$, glued together along their common faces. Symoblically, 
\begin{equation}\label{eq:Underlying Space of Simplicial Complex}
|\mcal{K}| = \bigcup_{\sigma_n \in \mcal{K}} \sigma_n
\end{equation}
where
\begin{equation}\label{eq:n-simplex}
 \sigma_n = \Big \{\btext{x} \in \R^N \colon \btext{x} = \sum_{i = 0}^{n}t_{i}\btext{a}_{i}, \, \text{where} \, \sum_{i = 0}^{n}t_{i} = 1 \, \text{and} \, t_{i} \geq 0, \, \forall i \Big\} 
\end{equation}
\end{definition}
The geometric realization $|\mcal{K}|$ is the actual geometric realization of $\mcal{K}$ in $\R^N$, capturing its shape in a continuous way rather than just as an abstract combinatorial structure. 

\begin{remark}
The geometric realization of any simplicial complex also called \textbf{polytope}  or \textbf{geometric realization} of the simplicial complex. A space that is polytope of a simplicial complex is called \textbf{polyhedron}. Some times polyhedron is use instead of polytope of finite simplicial complex.  
\end{remark}

\begin{example}[Geometric Realization of a Simplicial Complex]\label{eg:Example of Underlying Space of Simplicial Complex}
Is the geometric realization of the simplicial complex \(\mcal{K}\) in \(\mathbb{R}^2\), 
\[
\mcal{K} = \{\{\btext{a}_0\}, \{\btext{a}_1\}, \{\btext{a}_2\}, \{\btext{a}_0, \btext{a}_1\}, \{\btext{a}_1, \btext{a}_2\}, \{\btext{a}_2, \btext{a}_0\}, \{\btext{a}_0, \btext{a}_1, \btext{a}_2\}\},
\]
a geometrical 2-simplex, i.e. a filled or solid triangle including edges and vertices with vertices \(\btext{a}_0, \btext{a}_1, \btext{a}_2 \in \mathbb{R}^2\)? Further, is it topologically homeomorphic to a closed 2-dimensional disk?
\end{example}

\begin{solution*}
The simplicial complex $\mcal{K}$ in $\R^2$ is given by
\[\mcal{K} = \{\{\btext{a}_0\}, \{\btext{a}_1\}, \{\btext{a}_2\}, \{\btext{a}_0, \btext{a}_1\},     \{\btext{a}_1, \btext{a}_2\}, \{\btext{a}_2, \btext{a}_0\}, \{\btext{a}_0, \btext{a}_1, \btext{a}_2\} \}   \]

The simplicies of $\mcal{K}$ are as follows:
\begin{enumerate}
\item \textbf{0-simplices:} $\sigma_{0}^{1} = \{\btext{a}_{0}\}, \sigma_{0}^{2} = \{\btext{a}_{1}\}$ and $\sigma_{0}^{3} = \{\btext{a}_{1}\}$.
\item \textbf{1-simplices:} $\sigma_{1}^{1} = \{\btext{a}_{0}, \btext{a}_{1}\}, \sigma_{1}^{2} = \{\btext{a}_{1}, \btext{a}_{2}\}$ and $\sigma_{1}^{3} = \{\btext{a}_{2}, \btext{a}_{0}\}$.
\item \textbf{2-simplex:} $\sigma_{2}^{1} = \{\btext{a}_{0}, \btext{a}_{1}, \btext{a}_{2}\}$
\end{enumerate}

According to Definition \eqref{d:Geometric Realization of Simplicial Complex}, the geometric realization $|\mcal{K}|$ of $\mcal{K}$ is given by
\begin{align*}
|\mcal{K}| & = \bigcup_{\sigma_n \in \mcal{K}} \sigma_n \\
& = \sigma_{0}^{1} \cup \sigma_{0}^{2} \cup \sigma_{0}^{3} \cup \sigma_{1}^{1} \cup \sigma_{1}^{2} \cup \sigma_{1}^{3} \cup \sigma_{2}^{1} \\
& = \{\btext{a}_0\} \cup \{\btext{a}_1\} \cup \{\btext{a}_2\} \cup \{\btext{a}_0, \btext{a}_1\}\cup     \{\btext{a}_1, \btext{a}_2\}\cup \{\btext{a}_2, \btext{a}_0\}\cup \{\btext{a}_0, \btext{a}_1, \btext{a}_2\} \\
& = \{\btext{a}_0, \btext{a}_1, \btext{a}_2\} 
\end{align*}

Thus $|\mcal{K}| = \{\btext{a}_0, \btext{a}_1, \btext{a}_2\}$ which represent the geometrically 2-simplex i.e. filled or solid triangle (including edges and vertices) with vertices $\btext{a}_0, \btext{a}_1, \btext{a}_2 \in \R^2$. Topologically, it is homeomorphic to a closed 2-dimensional disk.  
\begin{figure}[h!]
    \centering
\begin{tikzpicture}
    % Define points
    \coordinate (A0) at (0,0);
    \coordinate (A1) at (4,0);
    \coordinate (A2) at (2,3);
    % Fill the triangular region
    \fill[blue!30, opacity=0.5] (A0) -- (A1) -- (A2) -- cycle;
    % Draw edges
    \draw[thick] (A0) -- (A1);
    \draw[thick] (A1) -- (A2);
    \draw[thick] (A2) -- (A0);
    % Draw vertices
    \filldraw[black] (A0) circle (2pt) node[below left] {$\mathbf{a}_0$};
    \filldraw[black] (A1) circle (2pt) node[below right] {$\mathbf{a}_1$};
    \filldraw[black] (A2) circle (2pt) node[above] {$\mathbf{a}_2$};
\end{tikzpicture}    
    \caption{Geometric Realization $|\mcal{K}|$ of Simplicial Complex $\mcal{K}$ in $\R^2$}
    \label{fig:Underlying Space of Simplicial Complex in R2}
\end{figure}
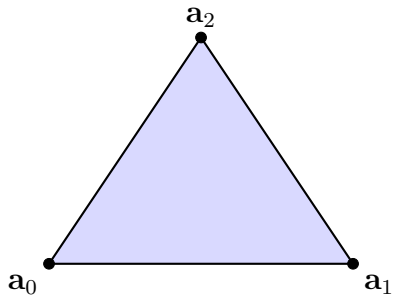
We can interpreted geometrically as follows. The geometric realization $|\mcal{K}|$ consists of the filled triangle formed by the three vertices $\btext{a}_0, \btext{a}_1, \btext{a}_2$. It include three edges (1-simplices) as line segments, The three vertices (0-simplices) as points.The interior of the triangle (the 2-simplex), making it a solid shape rather than just a collection of edges. Thus, the $|\mcal{K}|$  is the filled triangular region in $\R^2$.
\end{solution*}

As we saw that simplicial complex is not a topological space but its underlying space is a topological space which will be proved in the next result.

\begin{theorem}[Underlying space as Topological Space]\label{d:Underlying space as Topological Space}
Let $|\mcal{K}|$ is underlying space of the simplicial complex $\mcal{K}$, then the collection 
\begin{equation}\label{eq:Topology on Underlying space}
\mcal{T} = \{U \subset |\mcal{K}| \colon U \cap \sigma \, \text{is open in} \, \sigma \, \forall \sigma \in \mcal{K}  \}    
\end{equation}
The topology $\mcal{T}$ on the underlying space $|\mcal{K}|$ is the weakest (coarse) topology that ensures that the natural inclusion map $\iota \colon \sigma \to |\mcal{K}|$ defined by $\iota(\btext{x}) = \btext{x}$ for all $\btext{x} \in \sigma$, is continuous for every $\sigma \in \mcal{K}$ under the topology $\mcal{T}$.
\end{theorem}

\begin{proof}
As $|\mcal{K}|$ is the underlying space of the simple complex $\mcal{K}$ and a collection $\mcal{T}$ is defined as
\begin{equation}\label{eq:Topology on Underlying space}
\mcal{T} = \{U \subset |\mcal{K}| \colon U \cap \sigma \, \text{is open in} \, \sigma \, \forall \sigma \in \mcal{K}  \}    
\end{equation} 
We also know that $\sigma \in \R^N$ is topological space under subspace topology. Now we will prove 
\begin{enumerate}
\item The collection $\mcal{T}$ is topology on $|\mcal{K}|$. To prove this we show:
\begin{enumerate}
\item \textbf{$\emptyset$ and $|\mcal{K}| \in \mcal{T}$:} Since $\emptyset \cap \sigma = \sigma$ is open in every simplex $\sigma$, we have $\emptyset \in \mcal{T}$. Since $|\mcal{K}| \cap \emptyset = \emptyset$, which is open in itself, we have $|\mcal{K}| \in \mcal{T}$. 
\item \textbf{Closure under arbitrary unions:}
Let $\{U\}_{\alpha \in I}$ be an arbitrary collection of sets in $\mcal{T}$, then our aim is to show that $U = \cup_{\alpha \in I}U_{\alpha} \in \mcal{T}$. For every $\sigma \in \mcal{K}$, 
\begin{align*}
U \cap \sigma & = \Big ( \bigcup_{\alpha \in I} U_{\alpha}\Big)  \cap \sigma \\
& = \bigcup_{\alpha \in I} (U_{\alpha} \cap \sigma) \quad (\text{By set property})
\end{align*}
Since $U_{\alpha} \cap \sigma$ is open in $\sigma$ for all $\alpha \in I$ (because $U_{\alpha} \in \mcal{T}$), their union is also open in $\sigma$ (because $\sigma$ is itself topological space with subspace topology). Thus $U \cap \sigma$ is open in $\sigma$ for all $\sigma$, which means $U \in \mcal{T}$. So, $\mcal{T}$ is closed under arbitrary union. 
\item \textbf{Closure under finite intersections:}
Let $U, V \mcal{T}$, then our aim is to show that $U \cap V \in \mcal{T}$. For each $\sigma \in \mcal{K}$, 
\[(U \cap V) \cap \sigma = (U \cap \sigma) \cap (V \cap \sigma)\]
Since $(U \cap \sigma)$ and $(V \cap \sigma)$ are open in $\sigma$, their intersection is also open in $\sigma$. Thus, $(U \cap V) \cap \sigma$ is open for all $\sigma$, which means $U \cap V \in \mcal{T}$. So, $\mcal{T}$ is closed under finite intersection.  
\end{enumerate}

\item The inclusion map $\iota \colon \sigma \to |\mcal{K}|$ for all $\sigma \in \mcal{K}$ remain continuous. Since $\sigma$ is subspace of the space $|\mcal{K}|$ under the subspace topology and the so by stander result the inclusion map $\iota \colon \sigma \to |\mcal{K}|$ defined by $\iota(\btext{x}) = \btext{x}$ is continuous. 

\item The topology $\mcal{T}$ is weakest topology. Let $\mcal{T}'$ is another topology on $|\mcal{K}|$ such $\iota_{\sigma}$ is continuous. Now our aim is to show that $\mcal{T} \subseteq \mcal{T}'$. Let $U \in \mcal{T}$ be arbitrary, then by Definition of $\mcal{T}$, $U \cap \sigma$ is open in $\sigma$ for every $\sigma$. Since $\iota_{\sigma}$ is continuous for $\mcal{T}'$ therefore $\iota^{-1}_{\sigma}(U) = U \cap \sigma$ must be open in $\sigma$. So, by Definition of $\mcal{T}'$, $U \in \mcal{T}'$. Thus, $\mcal{T} \subseteq \mcal{T}'$.
\end{enumerate}
\end{proof}

The following two questions arise when defining the topology $\mcal{T}$ on the underlying space $|\mcal{K}|$.
\begin{question}
Why did we choose the topology $\mcal{T}$ as the weakest topology on the underlying space $|\mcal{K}|$?
\end{question}
\begin{answer}
If we took a weaker topology (with even fewer open sets), some open sets inside the simplices 
$\sigma$ would no longer be open in $|\mcal{K}|$, breaking the continuity of $\iota$. If we took a stronger topology (with more open sets), it would still make the inclusions continuous, but we want the smallest such topology. Thus, this topology $\mcal{T}$ is the weakest (coarsest) topology ensuring that every inclusion map $\iota$ is continuous.
\end{answer}

\begin{question}
Why is continuity of inclusion map $\iota \colon \sigma \to |\mcal{K}|$ for all $\sigma \in \mcal{K}$ is important?
\end{question}

\begin{answer}
A simplicial complex is built by gluing simplices together. If the topology is too weak, the gluing might not behave well under continuous mappings.
Ensuring that the inclusion 
$\iota \colon \sigma \to |\mcal{K}|$ is continuous means that any function defined on individual simplices can be naturally extended to the whole space $|\mcal{K}|$.
\end{answer}

\begin{theorem}\label{t:Topology of underlying space finer than subspace topology}
In general the topology of the underlying space $|\mcal{K}|$ is finer than the topology $|\mcal{K}|$ inherits as a subspace of $\R^N$. 
\end{theorem}

\begin{proof}
As given, let us denote the topology on $|\mcal{K}|$ by $\mcal{T}$ and the topology is obtained as a subspace of $\R^N$ i.e. by the subspace topology $\mcal{T}'$. Both topologies are defined as 
\begin{gather}
\mcal{T} = \{U \subset |\mcal{K}| \colon U \cap \sigma \, \text{is open in} \, \sigma,  \, \forall \sigma \in \mcal{K}  \} \\
\mcal{T}' = \{U \subset |\mcal{K}| \colon U = V \cap |\mcal{K}|,  \, V \, \text{is open in} \, \R^N \}
\end{gather}
Our aim is to show that $\mcal{T}' \subseteq \mcal{T}$. Let $U \in \mcal{T}'$ be arbitrary, then there exists open $V \subseteq \R^N$ such that  $U = V \cap |\mcal{K}|$. Since each simplex $\sigma$ is itself a subset of $\R^N$, then $U \cap \sigma = (V \cap |\mcal{K}|) \cap \sigma = V \cap \sigma$. Since $V$ is open in $\R^N$, its intersection with $\sigma$ (a topological subspace of $\R^N$ ) is open in $\sigma$ (by the definition of the subspace topology). Thus, for every simplex $\sigma$, $U \cap \sigma$ is open in $\sigma$, meaning that $U \in \mcal{T}$. Thus,  $\mcal{T}' \subseteq \mcal{T}$.
\end{proof}

\begin{theorem}\label{t:Two topologies on underlying space}
In general, the two topologies are different on the underlying space $|\mcal{K}|$ but the same for finite simplicial  spaces $\mcal{K}$. In other words, the topology $\mcal{T}$ of the underlying space $|\mcal{K}|$ is strictly finer than the subspace topology $\mcal{T}'$ from $\R^N$. However, if $\mcal{K}$ is finite, then these two topologies coincide.
\end{theorem}

\begin{proof}
There are two case arise: 
\begin{enumerate}
\item \textbf{$\mcal{K}$ is an infinite simplicial complex:} In this case we show that $\mcal{T}' \subset \mcal{T}$ i.e. that there exist open set in $\mcal{T}$ that is not open in the subspace topology $\mcal{T}'$ inherited from $\R^N$. 

\item[] In this proof, we consider the case where the simplicial complex $ \mathcal{K} $ is infinite and demonstrate that the topology $ \mathcal{T} $, which is the weakest topology making all simplex inclusions continuous, is strictly finer than the subspace topology inherited from $ \mathbb{R}^N $. We begin by understanding the given simplicial complex $ \mathcal{K} $, which consists of an infinite collection of vertices and edges. Mathematically, we define it as 
\[\mathcal{K} = \{ \{a_n\} \mid n \in \mathbb{N} \} \cup \{ \{a_n, a_{n+1}\} \mid n \in \mathbb{N} \},\]
meaning that the underlying space $ |\mathcal{K}| $ is the union of all these simplices, forming an infinite sequence of edges. Each 1-simplex $ [a_n, a_{n+1}] $ is a closed interval in $ \mathbb{R}^N $, and together they form an infinite one-dimensional geometric structure.

To prove that $ \mathcal{T} $ is strictly finer than the subspace topology, we must show that there exists at least one open set in $ \mathcal{T} $ that is not open in the subspace topology. Consider the set 
\[ U = |\mathcal{K}| \setminus \{a_0\} = \bigcup_{n \geq 1} [a_n, a_{n+1}],\]
which consists of all edges except the removal of the vertex $ a_0 $. In the topology $ \mathcal{T} $, a set is open if its intersection with each simplex is open in the standard topology of that simplex. Since $ U $ consists of full edges $ [a_n, a_{n+1}] $ for $ n \geq 1 $, and the part of $ [a_0, a_1] $ remaining in $ U $ is $ (a_0, a_1] $, the set $ U $ is open in $ \mathcal{T} $.

We now check whether $ U $ is open in the subspace topology inherited from $ \mathbb{R}^N $. In the subspace topology, a set is open if it is the intersection of an open set in $ \mathbb{R}^N $ with $ |\mathcal{K}| $, meaning 
\[   U = W \cap |\mathcal{K}|\]
for some open $ W \subseteq \mathbb{R}^N $. However, this is not possible because the point $ a_0 $ is a limit point of $ U $ in $ \mathbb{R}^N $, and any Euclidean open set containing $ U $ must also contain a small ball around $ a_0 $. Consequently, $ W \cap |\mathcal{K}| $ would necessarily include some neighborhood around $ a_0 $, contradicting $ U = W \cap |\mathcal{K}| $. Furthermore, in the Euclidean topology, a single point $ a_0 $ cannot be removed without also removing an open neighborhood around it, which $ U $ does not possess.

Since we have found an example of a set that is open in $ \mathcal{T} $ but not in the subspace topology, we conclude that 
\[    \mathcal{T} \supsetneq \text{(subspace topology from } \mathbb{R}^N\text{)}.\]
Thus, the topology $ \mathcal{T} $ is strictly finer than the subspace topology from $ \mathbb{R}^N $ when $ \mathcal{K} $ is infinite. This completes the proof.

\item \textbf{$\mcal{K}$ is a finite simplicial complex:} In this case we show that $\mcal{T}' = \mcal{T}$. Now, we consider the case where the simplicial complex $ \mathcal{K} $ is finite and show that the topology $ \mathcal{T} $ coincides with the subspace topology inherited from $ \mathbb{R}^N $. Given that $ \mathcal{K} $ is finite, the underlying space $ |\mathcal{K}| $ consists of a finite number of simplices, including vertices, edges, and possibly higher-dimensional simplices. Each simplex in $ \mathcal{K} $ is a compact subset of $ \mathbb{R}^N $, and their finite union is also compact in the subspace topology.

To prove the equality of the topologies, we need to show that any set open in $ \mathcal{T} $ is also open in the subspace topology and vice versa. Since $ \mathcal{T} $ is the weakest topology making all inclusion maps continuous, it follows that any open set in the subspace topology must also be open in $ \mathcal{T} $. Conversely, in a finite simplicial complex, the intersection of any open set with a simplex remains open in the Euclidean topology of that simplex. Since there are only finitely many simplices, this ensures that any open set in $ \mathcal{T} $ is also open in the subspace topology.

Furthermore, compactness plays a crucial role in finite simplicial complexes. Since $ |\mathcal{K}| $ is a compact subset of $ \mathbb{R}^N $, the subspace topology on $ |\mathcal{K}| $ is locally compact and coincides with $ \mathcal{T} $. This follows because in finite-dimensional spaces, compact and locally compact topologies coincide with their subspace topologies.

Thus, we conclude that for finite $ \mathcal{K} $, the two topologies are identical, i.e.,
\[    \mathcal{T} = \text{(subspace topology from } \mathbb{R}^N\text{)}.\]
\end{enumerate}
\end{proof}

\subsection{Topological Properties of Geometric Realization of Simplicial Complex}\label{ss:Topological Properties of Geometric Realization of Simplicial Complex}
Let us see some elementary topological properties of polyhedron. 

\begin{lemma}
If $\mcal{L}$ is sub-complex of the simplicial space $\mcal{K}$, then $|\mcal{L}|$ is closed subspace of the space $\mcal{K}$. In particular, $\sigma \in \mcal{K}$, then $\sigma$ is a closed subspace of $|\mcal{K}|$.
\end{lemma}
\begin{proof}
The geometric realization of a simplicial complex $\mathcal{K}$ is given by $|\mathcal{K}| = \bigcup_{\sigma \in \mathcal{K}} |\sigma|$, and similarly, the realization of a subcomplex $\mathcal{L}$ is $|\mathcal{L}| = \bigcup_{\sigma \in \mathcal{L}} |\sigma|$. Each simplex $|\sigma|$ is homeomorphic to a standard Euclidean simplex and is therefore a compact subset of some $\mathbb{R}^n$. The space $|\mathcal{K}|$ is equipped with the weak topology induced by its simplices, meaning that a subset $U \subseteq |\mathcal{K}|$ is closed if and only if $U \cap |\sigma|$ is closed in $|\sigma|$ for each $\sigma \in \mathcal{K}$.  

Since $\mathcal{L}$ is a subcomplex of $\mathcal{K}$, it follows that $|\mathcal{L}| = \bigcup_{\sigma \in \mathcal{L}} |\sigma|$. Each simplex $|\sigma|$ is closed in $|\mathcal{K}|$ because simplices in a geometric realization are compact, and compact sets are closed in their ambient space. Since $|\mathcal{L}|$ is a union of such closed sets, it follows that $|\mathcal{L}|$ is closed in $|\mathcal{K}|$, as unions of closed sets remain closed under the weak topology.  

In particular, if $\sigma \in \mathcal{K}$, then $|\sigma|$ itself is closed in $|\mathcal{K}|$. This follows from the fact that simplices form compact subspaces in Euclidean topology, and compact sets remain closed in the weak topology. Therefore, we conclude that $|\mathcal{L}|$ is a closed subspace of $|\mathcal{K}|$ and that each $|\sigma|$ is a closed subspace of $|\mathcal{K}|$.
\end{proof}

\begin{lemma}
A map $f \colon |\mcal{K}| \to X$ is continuous if and only if $f|_{\sigma} \colon \sigma \to X$ is continuous for each $\sigma \in \mcal{K}$.
\end{lemma}

\begin{proof}
\textbf{($\Rightarrow$) Suppose $f: |\mathcal{K}| \to X$ is continuous.}  
By definition of continuity, for any open set $U \subseteq X$, the preimage $f^{-1}(U)$ is open in $|\mathcal{K}|$. Since $|\mathcal{K}|$ is endowed with the weak topology determined by its simplices, a set $A \subseteq |\mathcal{K}|$ is open if and only if $A \cap |\sigma|$ is open in $|\sigma|$ for every $\sigma \in \mathcal{K}$. Applying this to $f^{-1}(U)$, we get:  
\[
f^{-1}(U) \text{ is open in } |\mathcal{K}| \implies f^{-1}(U) \cap |\sigma| \text{ is open in } |\sigma|, \quad \forall \sigma \in \mathcal{K}.
\]
Since $f|_{\sigma} = f|_{|\sigma|}$ is simply the restriction of $f$ to $|\sigma|$, we conclude that:  
\[
f^{-1}(U) \cap |\sigma| = (f|_{\sigma})^{-1}(U) \text{ is open in } |\sigma|, \quad \forall \sigma \in \mathcal{K}.
\]
Thus, $f|_{\sigma}$ is continuous for each $\sigma \in \mathcal{K}$, since continuity is defined by the preimage of open sets being open.

\textbf{($\Leftarrow$)Suppose $f|_{\sigma}: |\sigma| \to X$ is continuous for each $\sigma \in \mathcal{K}$.}  
To show that $f$ is continuous, we must show that for any open set $U \subseteq X$, the preimage $f^{-1}(U)$ is open in $|\mathcal{K}|$. Since $|\mathcal{K}|$ has the weak topology with respect to its simplices, a subset $A \subseteq |\mathcal{K}|$ is open if and only if $A \cap |\sigma|$ is open in $|\sigma|$ for all $\sigma \in \mathcal{K}$. Given that $f|_{\sigma}$ is continuous, we know that:  
\[
(f|_{\sigma})^{-1}(U) = f^{-1}(U) \cap |\sigma| \text{ is open in } |\sigma|, \quad \forall \sigma \in \mathcal{K}.
\]
Since this holds for every simplex $\sigma$, and the weak topology ensures that openness of $f^{-1}(U) \cap |\sigma|$ for all $\sigma$ implies openness of $f^{-1}(U)$ in $|\mathcal{K}|$, we conclude that $f^{-1}(U)$ is open in $|\mathcal{K}|$. Therefore, $f$ is continuous.
\end{proof}

\begin{example}\label{eg:Continuity Theorem on Simplicial Complex}
Consider the simplicial complex 
\[
\mcal{K} = \{\{\btext{a}_0\}, \{\btext{a}_1\}, \{\btext{a}_2\}, \sigma_1 =\{\btext{a}_0, \btext{a}_1\}, \sigma_2 = \{\btext{a}_1, \btext{a}_2\}, \sigma_3 = \{\btext{a}_0, \btext{a}_1, \btext{a}_2\}\}.
\]

The geometric realization \( |\mathcal{K}| \) is a filled-in triangle with edges and vertices included. Define a function \( f : |\mathcal{K}| \to \mathbb{R} \) by \( f(\btext{a}_0) = 0, f(\btext{a}_1) = 1, f(\btext{a}_2) = 2 \) and extend \( f \) linearly on each edge and inside the triangle. 

Is \( f \) continuous on \( |\mathcal{K}| \) if and only if its restriction \( f|_{\sigma} \) is continuous for each simplex \( \sigma \in \mathcal{K} \)?
\end{example}

\begin{solution*}
To verify the theorem, we check continuity both locally (on each simplex) and globally (on \( |\mathcal{K}| \)).

\begin{enumerate}
    \item \textbf{Checking Continuity on Each Simplex:} Analyze the function \( f \) on each simplex separately:
    \begin{itemize}
        \item On each vertex \( \btext{a}_i \), \( f \) is trivially continuous.
        \item On each edge \( \sigma_1 \) and \( \sigma_2 \): Since \( f \) is linearly interpolated along the edges, it is continuous as a function on a line segment (which is homeomorphic to \([0,1]\)).
        \item On the triangle \( \sigma_3 \): Since \( f \) is defined linearly inside the triangle, it is continuous in the Euclidean topology on the triangle.
    \end{itemize}
    Thus, for each simplex \( \sigma \), the restriction \( f|_{\sigma} \) is continuous.
    
    \item \textbf{Checking Global Continuity on \( |\mathcal{K}| \):}
    Since \( |\mathcal{K}| \) has the weak topology, a function is continuous if its restriction to each simplex is continuous. We have already verified that \( f|_{\sigma} \) is continuous for every simplex \( \sigma \). Therefore, by the continuity theorem, \( f \) must be continuous on \( |\mathcal{K}| \).
\end{enumerate}
\end{solution*}

The main objective of introducing the coherent topology (or final topology) is to construct a topology on a space $X$ based on a collection of subspaces $\{X_i\}$ in a way that ensures the continuity of maps from each $X_i$  into $X$. The coherent topology allows us to glue together different subspaces while preserving their topological structure. The simplicial complexes use the coherent topology to define their global topology from local simplices or cells. The geometric realization of a simplicial complex is given the coherent topology relative to its simplices.

\begin{definition}[Coherent Topology]\label{d:Coherent Topology}
Let $ X $ be a set and $ \{X_i\}_{i \in I} $ be a family of subspaces of $ X $ with topologies $ \tau_i $. The coherent topology (or final topology) on $ X $ with respect to $ \{X_i\}_{i \in I} $ is the topology $ \tau $ defined by
\[
\tau = \{ U \subseteq X \colon U \cap X_i \in \tau_i \text{ for all } i \in I \}
\]    
\end{definition}

\begin{remark}
Equivalently, $ \tau $ is the finest topology on $ X $ such that each inclusion map  $\iota_i\colon X_i \to X$ is continuous.  This means that a function $ f\colon X \to Y $ into a topological space $ Y $ is continuous if and only if each restriction  $f|_{X_i}\colon X_i \to Y$ is continuous for all $ i \in I $.
\end{remark}

\begin{theorem}[Coherent Topology on $|\mcal{K}|$]\label{t:Coherent Topology on |K|}
The topology of $|\mcal{K}|$ is coherent with the collection of subspace $\sigma$, for $\sigma \in \mcal{K}$.  
\end{theorem}

\begin{proof}
Let $ |\mathcal{K}| $ be the geometric realization of a simplicial complex $ \mathcal{K} $, where each simplex $ \sigma \in \mathcal{K} $ is given the subspace topology inherited from $ |\mathcal{K}| $. We want to show that the topology of $ |\mathcal{K}| $, denoted by $ \tau $, is coherent with the collection of subspaces $ \{\sigma \colon \sigma \in \mathcal{K} \} $, i.e. 
\[
\tau = \{ U \subseteq |\mathcal{K}| \colon U \cap \sigma \text{ is open in } \sigma, \forall \sigma \in \mathcal{K} \}.
\]

First, we show that $ \tau $ is at least as fine as the coherent topology. Suppose $ U $ is open in $ |\mathcal{K}| $, i.e., $ U \in \tau $. By the definition of the subspace topology, we have 
\[
U \cap \sigma \text{ is open in } \sigma, \quad \forall \sigma \in \mathcal{K}.
\]
Since this is precisely the condition for $ U $ to be in the coherent topology, it follows that $ U $ belongs to the coherent topology. Hence,
\[
U \in \tau \implies U \in \text{coherent topology}.
\]

Next, we show that the coherent topology is at most as fine as $ \tau $. Suppose $ U $ is open in the coherent topology, meaning that 
\[
U \cap \sigma \text{ is open in } \sigma, \quad \forall \sigma \in \mathcal{K}.
\]
By definition, the topology on $ |\mathcal{K}| $ is the weakest topology that makes all inclusion maps $ \iota_\sigma: \sigma \to |\mathcal{K}| $ continuous. This means that a set is open in $ |\mathcal{K}| $ if and only if its restriction to each simplex is open in that simplex. Since $ U \cap \sigma $ is open in $ \sigma $ for all $ \sigma \in \mathcal{K} $, it follows that $ U $ is open in $ |\mathcal{K}| $, i.e.,
\[
U \in \text{coherent topology} \implies U \in \tau.
\]
Since both inclusions hold, we conclude that $ \tau $ is precisely the coherent topology, proving that the topology of $ |\mathcal{K}| $ is coherent with the collection of subspaces $ \sigma $, for $ \sigma \in \mathcal{K} $. 
\end{proof}

\begin{definition}[Barycentric Coordinate]\label{d:Barycentric Coordinate}
Let $ \sigma = [\mathbf{a}_0, \mathbf{a}_1, \dots, \mathbf{a}_n] $ be an $ n $-dimensional simplex in a simplicial complex $ \mathcal{K} $. The barycentric coordinate is function that is associated with a vertex $ \mathbf{a}_j \in \sigma $ and defined as:
\[
\lambda_j: |\mathcal{K}| \to [0,1]
\]
such that for each $ \mathbf{x} \in |\mathcal{K}| $, if $ \mathbf{x} $ belongs to some simplex $ \sigma $, then $ \mathbf{x} $ can be uniquely expressed as a convex combination:
\[
\mathbf{x} = \sum_{i=0}^{n} \lambda_i(\mathbf{x}) \mathbf{a}_i, \quad \text{where } \sum_{i=0}^{n} \lambda_i(\mathbf{x}) = 1 \text{ and } \lambda_i(\mathbf{x}) \geq 0 \text{ for all } i.
\]

The function $ \lambda_j $ assigns to each $ \mathbf{x} $ its barycentric coordinate with respect to $ \mathbf{a}_j $, i.e.,  

\[
\lambda_j(\mathbf{x}) = \text{coefficient of } \mathbf{a}_j \text{ in the convex combination of } \mathbf{x}.
\]
This function is continuous on $ |\mathcal{K}| $ and varies smoothly within each simplex. the barycentric coordinate function 
$\lambda_j(\mathbf{x})$ provides a natural way to describe the position of $\btext{x}$ relative to the vertices of a simplex.
\end{definition}

\begin{example}
Consider a triangle in $ \mathbb{R}^2 $ with vertices:
\[\mathbf{a}_0 = (1,1), \quad \mathbf{a}_1 = (4,2), \quad \mathbf{a}_2 = (2,5).
\]
Find the barycentric coordinates $ \lambda_0, \lambda_1, \lambda_2 $ of the point $\btext{x} = (3,3)$.
\end{example}
\begin{solution*}
Since $ \mathbf{x} $ is a convex combination of the vertices:
\[
\mathbf{x} = \lambda_0 \mathbf{a}_0 + \lambda_1 \mathbf{a}_1 + \lambda_2 \mathbf{a}_2.
\]

This gives the system:
\begin{gather}
3 = \lambda_0 (1) + \lambda_1 (4) + \lambda_2 (2),\\
3 = \lambda_0 (1) + \lambda_1 (2) + \lambda_2 (5),\\
\lambda_0 + \lambda_1 + \lambda_2 = 1.
\end{gather}

Rewriting in matrix form:
\[
\begin{bmatrix} 
1 & 4 & 2 \\ 
1 & 2 & 5 \\ 
1 & 1 & 1 
\end{bmatrix}
\begin{bmatrix} 
\lambda_0 \\ 
\lambda_1 \\ 
\lambda_2 
\end{bmatrix}
=
\begin{bmatrix} 
3 \\ 
3 \\ 
1 
\end{bmatrix}.
\]
On solving above system of simultaneous  linear equations we get barycentric coordinates of $ \mathbf{x} = (3,3) $
\[
\lambda_0 = \frac{1}{11}, \quad \lambda_1 = \frac{6}{11}, \quad \lambda_2 = \frac{4}{11}.
\]
Since all coordinates are nonnegative and sum to 1, the point $ (3,3) $ lies inside the triangle.
\end{solution*}

\begin{lemma}[Hausdorffness of Simplicial Complex]\label{l:Hausdorffness of Simplicial Complex}
The underlying space $|\mcal{K}|$ of a simplicial complex $\mcal{K}$ is a Hausdorff space.
\end{lemma}

\begin{proof}
To prove that the underlying space $ |\mathcal{K}| $ of a simplicial complex $ \mathcal{K} $ is a Hausdorff space, we must show that for any two distinct points $ \mathbf{x}, \mathbf{y} \in |\mathcal{K}| $, there exist disjoint open neighborhoods $ U $ and $ V $ such that $ \mathbf{x} \in U $ and $ \mathbf{y} \in V $.  

The underlying space $ |\mathcal{K}| $ is the union of the simplices in $ \mathcal{K} $, each endowed with the subspace topology from $ \mathbb{R}^n $ i.e. $|\mathcal{K}| = \bigcup_{\sigma \in \mathcal{K}} \sigma$. Each simplex $ \sigma $ is a subset of some Euclidean space $ \mathbb{R}^n $, and the topology on $ |\mathcal{K}| $ is the coherent topology induced by the simplices. The key observation is that each individual simplex is Hausdorff, since it is a subset of Euclidean space, which is Hausdorff.

\begin{enumerate}
\item \textbf{When $ \mathbf{x} $ and $ \mathbf{y} $ belong to different simplices:} Let $ \mathbf{x}  \in \sigma_1$ and $ \mathbf{y}  \in \sigma_2$, where $\sigma_1 \cap \sigma_2 = \emptyset$.  Since each simplex $ \sigma_i $ is an open set in its ambient Euclidean space and simplices are disjoint, we can choose small open neighborhoods $ U_1 \subset \sigma_1 $ and $ U_2 \subset \sigma_2 $ such that $ \mathbf{x} \in U_1 $ and $ \mathbf{y} \in U_2 $, with $ U_1 \cap U_2 = \emptyset $. Thus, the points are separated by disjoint open sets, proving the Hausdorff property in this case.

\item \textbf{When $ \mathbf{x} $ and $ \mathbf{y} $ belong to the same simplex:} Let $ \mathbf{x}, \mathbf{y} \in \sigma $ for some simplex $ \sigma $. Since every simplex is a subset of Euclidean space, it inherits the standard topology of $ \mathbb{R}^n $, which is Hausdorff. That is, for any two distinct points in $ \sigma $, there exist disjoint open sets separating them within $ \sigma $, say $ U, V \subset \sigma $, such that  $U \cap V = \emptyset, \quad \mathbf{x} \in U, \quad \mathbf{y} \in V$. Since the topology of $ |\mathcal{K}| $ is coherent with the simplices, these sets $ U $ and $ V $ are open in $ |\mathcal{K}| $ as well, ensuring the separation condition.

\item \textbf{When $ \mathbf{x} $ and $ \mathbf{y} $ belong to different simplices but share a common face:} Now, let $ \mathbf{x} \in \sigma_1$ and $ \mathbf{y}  \in \sigma_2$, but the simplices share a common face $ \tau $, i.e., $\sigma_1 \cap \sigma_2 = \tau \neq \emptyset$. If either $ \mathbf{x} $ or $ \mathbf{y} $ belongs to $ \tau $, then this reduces to the previous case of being within a single simplex, which we have already proven. Otherwise, they lie in different relative interiors of $ \sigma_1 $ and $ \sigma_2 $. Since relative interiors of disjoint simplices do not intersect, we can find small disjoint open neighborhoods $ U_1 \subset \sigma_1 $ and $ U_2 \subset \sigma_2 $ that are open in $ |\mathcal{K}| $, ensuring separation.
\end{enumerate}
We have shown that for any two distinct points $ \mathbf{x}, \mathbf{y} \in |\mathcal{K}| $, whether they belong to the same simplex, disjoint simplices, or simplices with a shared face, we can always find disjoint open neighborhoods that separate them. Thus, the underlying space $ |\mathcal{K}| $ satisfies the Hausdorff separation property.
\end{proof}

\begin{lemma}[Compactness of Simplicial Complex]\label{l:Compactness of Simplicial Complex}
The underlying space $|\mcal{K}|$ of a finite simplicial complex $\mcal{K}$ is compact. Conversely, if a subset $A$ of $|\mcal{K}|$ is compact, then $A \subset |\mcal{K}_0|$ for some finite sub-complex $\mcal{K}_0$ of $\mcal{K}$.
\end{lemma}

\begin{proof}
Let $ \mathcal{K} $ be a finite simplicial complex. The underlying space $ |\mathcal{K}| $ is given by the finite union of its simplices as $|\mathcal{K}| = \bigcup_{i=1}^{m} \sigma_i$, where each $ \sigma_i $ is a simplex in $ \mathbb{R}^n $.

Each simplex $ \sigma_i $ is the convex hull of a finite set of points in $ \mathbb{R}^n $, meaning it is a bounded and closed subset of $ \mathbb{R}^n $. By the Heine-Borel theorem, every closed and bounded subset of $ \mathbb{R}^n $ is compact i.e. $\sigma_i$ is compact for all $i = 1,2,\dots,m$. Since compactness is preserved under finite unions, the finite union of compact sets remains compact i.e. $|\mathcal{K}| = \bigcup_{i=1}^{m} \sigma_i$  is compact. 
\end{proof}

To describe the local properties of the underlying space $|\mcal{K}|$, let us look at some special subspaces of $|\mcal{K}|$ which will definitely help us in studying the local properties of $|\mcal{K}|$.

\begin{definition}[Star of vertex in Simplicial Complex]\label{l:Star of vertex in Simplicial Complex}
Let $\btext{a}$ is the vertex of the simplicial complex $\mcal{K}$, then the star of the $\btext{a}$ in $\mcal{K}$ is denoted by $\operatorname{St}(\btext{a}, \mcal{K})$ and it is defined as the collection of all simplices in $\mcal{K}$ that contain $\btext{a}$ i.e. 
\begin{equation}\label{eq:Star of vertex in Simplicial Complex}
\operatorname{St}(\btext{a}, \mcal{K}) = \{\sigma \in \mcal{K} \colon \btext{a} \in \sigma\}    
\end{equation}
The geometric realization of the star, denoted as $|\mrm{st}(\btext{a}, \mcal{K})|$ and defined as 
\begin{equation}\label{eq:Geometric Realization of Star of vertex in Simplicial Complex}
|\operatorname{St}(\btext{a}, \mcal{K})| = \bigcup_{\sigma \in \mcal{K}, \btext{a} \in \sigma} \operatorname{Int}(\sigma)    
\end{equation}
where $\operatorname{Int}(\sigma)$ represents the topological interior of the simplex $\sigma$ within its affine hull.
\end{definition}

\begin{remark}\hfill
\begin{enumerate}
\item It provides a localized view around $\btext{a}$, ensuring that every point in the star belongs to the interior of some simplex containing $\btext{a}$.
\item It plays an essential role in topological properties like local connectivity and defining subdivisions in simplicial complexes.
\end{enumerate}
\end{remark}

\begin{definition}[Closure of Star of vertex in Simplicial Complex]\label{l:Closure of Star of vertex in Simplicial Complex}
The closure of the star of $\btext{a}$, denoted as $\operatorname{Cl}(\operatorname{St}(\btext{a},\mcal{K}))$, is the smallest sub-complex of the simplicial complex $\mcal{K}$ that contains the star of $\btext{a}$ i.e.
\begin{equation}\label{eq:Closure of Star of vertex in Simplicial Complex}
\operatorname{Cl}(\operatorname{St}(\btext{a},\mcal{K})) =  \bigcup_{\sigma \in \mcal{K}, \btext{a} \in \sigma} \sigma      
\end{equation}
The geometric realization of the $\operatorname{Cl}(\operatorname{St}(\btext{a},\mcal{K})$, denoted as $|\operatorname{Cl}(\operatorname{St}(\btext{a},\mcal{K})|$ and defined as 
\begin{equation}\label{eq:Geometric Realization of Closure of Star of vertex in Simplicial Complex}
|\operatorname{Cl}(\operatorname{St}(\btext{a},\mcal{K})| = \bigcup_{\sigma \in \mcal{K}, \btext{a} \in \sigma} |\sigma|      
\end{equation}
\end{definition}

\begin{definition}[Link of vertex in Simplicial Complex]\label{l:Link of vertex in Simplicial Complex}
The link of $\btext{a}$ in the simplicial complex $\mcal{K}$, denoted as 
$\operatorname{Lk}(\btext{a},\mcal{K})$, is the sub-complex consisting of all faces of simplices in $\operatorname{Cl}(\operatorname{St}(\btext{a},\mcal{K}))$ that do not contain $\btext{a}$ i.e.
\begin{equation}\label{eq:Link of vertex in Simplicial Complex}
\operatorname{Lk}(\btext{a},\mcal{K}) = \{\tau \in \operatorname{Cl}(\operatorname{St}(\btext{a},\mcal{K})) \colon \btext{a} \notin \tau \}    
\end{equation}
The geometric realization of the $\operatorname{Lk}(\btext{a},\mcal{K})$, denoted as $|\operatorname{Lk}(\btext{a},\mcal{K})|$ and defined as 
\begin{equation}\label{eq:Geometric Realization of Link of vertex in Simplicial Complex}
|\operatorname{Lk}(\btext{a},\mcal{K})| = \bigcup_{\sigma \in \mcal{K}, \btext{a} \in \sigma} \operatorname{conv} (\sigma \backslash \{\btext{a}\})   
\end{equation}
\end{definition}

\begin{example}
Let $ \mathcal{K} $ be a simplicial complex with the vertex set:
\[
A = \{\mathbf{a}_0, \mathbf{a}_1, \mathbf{a}_2, \mathbf{a}_3, \mathbf{a}_4, \mathbf{a}_5, \mathbf{a}_6, \mathbf{a}_7\}
\]

The complex consists of the following simplices:
\begin{itemize}
   \item \textbf{Vertices (0-simplices):} 
   \[\mathbf{a}_0, \mathbf{a}_1, \mathbf{a}_2, \mathbf{a}_3, \mathbf{a}_4, \mathbf{a}_5, \mathbf{a}_6, \mathbf{a}_7\]
    \item \textbf{Edges (1-simplices):} 
    \[
    (\mathbf{a}_0, \mathbf{a}_1), (\mathbf{a}_0, \mathbf{a}_2), (\mathbf{a}_0, \mathbf{a}_3), (\mathbf{a}_1, \mathbf{a}_4), (\mathbf{a}_2, \mathbf{a}_5), (\mathbf{a}_3, \mathbf{a}_6), (\mathbf{a}_4, \mathbf{a}_5), (\mathbf{a}_5, \mathbf{a}_6), (\mathbf{a}_6, \mathbf{a}_7)
    \]
    
    \item \textbf{Triangles (2-simplices):} 
    \[
    (\mathbf{a}_0, \mathbf{a}_1, \mathbf{a}_2), (\mathbf{a}_1, \mathbf{a}_2, \mathbf{a}_4), (\mathbf{a}_2, \mathbf{a}_5, \mathbf{a}_6), (\mathbf{a}_3, \mathbf{a}_6, \mathbf{a}_7)
    \]
\end{itemize}

This forms a connected simplicial complex with vertices, edges, triangles, and some higher-dimensional structure.
\end{example}

\begin{enumerate}
\item \textbf{Star of $ \mathbf{a}_2 $:} The star of $ \mathbf{a}_2 $ is the collection of all simplices that contain $ \mathbf{a}_2 $:
\begin{align*}
\operatorname{St}(\mathbf{a}_2, \mathcal{K}) & = \{ \sigma \in \mathcal{K} \colon \mathbf{a}_2 \in \sigma \} \\     
& = \{ (\mathbf{a}_0, \mathbf{a}_2), (\mathbf{a}_1, \mathbf{a}_2), (\mathbf{a}_2, \mathbf{a}_5), (\mathbf{a}_0, \mathbf{a}_1, \mathbf{a}_2), (\mathbf{a}_1, \mathbf{a}_2, \mathbf{a}_4), (\mathbf{a}_2, \mathbf{a}_5, \mathbf{a}_6) \}
\end{align*}
\item \textbf{Closure of the Star of $ \mathbf{a}_2 $:} The closure of the star is the smallest subcomplex containing the star, including all its faces:
\[
\operatorname{Cl}(\operatorname{St}(\mathbf{a}_{2}, \mcal{K})) = \bigcup_{\sigma \in \operatorname{St}(\mathbf{a}_2)} \sigma
\]
This includes all edges and triangles containing $ \mathbf{a}_2 $, forming a closed neighborhood. This means we must include:

\begin{itemize}
    \item \textbf{All simplices in $ \operatorname{St}(\mathbf{a}_2, \mcal{K}) $}:
    \[
    (\mathbf{a}_0, \mathbf{a}_2), (\mathbf{a}_1, \mathbf{a}_2), (\mathbf{a}_2, \mathbf{a}_5), (\mathbf{a}_0, \mathbf{a}_1, \mathbf{a}_2), (\mathbf{a}_1, \mathbf{a}_2, \mathbf{a}_4), (\mathbf{a}_2, \mathbf{a}_5, \mathbf{a}_6)
    \]
    
    \item \textbf{All faces (subsets) of these simplices, including vertices and edges}:
    \[
    \{\mathbf{a}_0\}, \{\mathbf{a}_1\}, \{\mathbf{a}_2\}, \{\mathbf{a}_4\}, \{\mathbf{a}_5\}, \{\mathbf{a}_6\}
    \]
    \[
    (\mathbf{a}_0, \mathbf{a}_1), (\mathbf{a}_0, \mathbf{a}_2), (\mathbf{a}_1, \mathbf{a}_2), (\mathbf{a}_1, \mathbf{a}_4), (\mathbf{a}_2, \mathbf{a}_5), (\mathbf{a}_5, \mathbf{a}_6)
    \]
    
    \item \textbf{The full triangles from $ \operatorname{St}(\mathbf{a}_2, \mcal{K}) $}:
    \[
    (\mathbf{a}_0, \mathbf{a}_1, \mathbf{a}_2), (\mathbf{a}_1, \mathbf{a}_2, \mathbf{a}_4), (\mathbf{a}_2, \mathbf{a}_5, \mathbf{a}_6)
    \]
\end{itemize}
\item[] Finally 
\begin{multline*}
\operatorname{Cl}(\operatorname{St}(\mathbf{a}_2, \mcal{K})) =
\{\mathbf{a}_0, \mathbf{a}_1, \mathbf{a}_2, \mathbf{a}_4, \mathbf{a}_5, \mathbf{a}_6\} \cup \{ (\mathbf{a}_0, \mathbf{a}_1), (\mathbf{a}_0, \mathbf{a}_2), (\mathbf{a}_1, \mathbf{a}_2), (\mathbf{a}_1, \mathbf{a}_4), (\mathbf{a}_2, \mathbf{a}_5), (\mathbf{a}_5, \mathbf{a}_6) \} \\  \cup \{ (\mathbf{a}_0, \mathbf{a}_1, \mathbf{a}_2), (\mathbf{a}_1, \mathbf{a}_2, \mathbf{a}_4), (\mathbf{a}_2, \mathbf{a}_5, \mathbf{a}_6) \}   
\end{multline*}

\item \textbf{Link of $ \mathbf{a}_2 $:} The link consists of faces of simplices in $ \operatorname{Cl}(\operatorname{St}(\mathbf{a}_2)) $ that do not contain $ \mathbf{a}_2 $:
\begin{align*}
\operatorname{Lk}(\mathbf{a}_2, \mcal{K}) & = \{ \tau \colon \tau \subset \sigma, \mathbf{a}_2 \in \sigma, \mathbf{a}_2 \notin \tau \} \\
& = \{ (\mathbf{a}_0, \mathbf{a}_1), (\mathbf{a}_1, \mathbf{a}_4), (\mathbf{a}_5, \mathbf{a}_6) \}
\end{align*}
\end{enumerate}

\begin{remark}\hfill
\begin{enumerate}
\item  The $\operatorname{St}(\btext{a}, \mcal{K})$ is open in $|\mcal{K}|$.
\item Both spaces $\operatorname{St}(\btext{a}, \mcal{K})$ and $\operatorname{Cl}(\operatorname{St}(\btext{a}, \mcal{K}))$ are path connected but $\operatorname{Lk}(\btext{a}, \mcal{K})$ is not. 
\end{enumerate}
\end{remark}

\begin{lemma}[Existence of compact subspace in Simplicial Complex]\label{l:Existence of compact subspace in Simplicial Complex}
If $A$  is subset of the underlying space $|\mcal{K}|$  of the simplicial complex $\mcal{K}$ is compact, then $A \subset |\mcal{K}_0|$ for some finite sub-complex $\mcal{K}_0$ of $\mcal{K}$.
\end{lemma}

\begin{proof}
The realization $ |\mathcal{K}| $ is covered by the interiors of its simplices i.e. $ |\mathcal{K}| = \bigcup_{\sigma \in \mathcal{K}} \operatorname{Int}(\sigma)$. Since $ A \subset |\mathcal{K}| $, we can write $A \subset \bigcup_{\sigma \in \mathcal{K}} \operatorname{Int}(\sigma)$.

The sets $ \operatorname{Int}(\sigma) $ form an open cover of $ A $. Since $ A $ is compact, there exists a finite subcover, meaning there exist finitely many simplices $ \sigma_1, \sigma_2, \dots, \sigma_n $ such that $ A \subset \bigcup_{i=1}^{n} \operatorname{Int}(\sigma_i)$.

Let $ \mathcal{K}_0 $ be the finite subcomplex of $ \mathcal{K} $ generated by the simplices $ \{ \sigma_1, \sigma_2, \dots, \sigma_n \} $, meaning $\mathcal{K}_0 = \{ \tau \in \mathcal{K} \mid \tau \subseteq \sigma_i \text{ for some } i \}$.  Since each simplex $ \sigma_i $ belongs to $ \mathcal{K} $, all its faces are also in $ \mathcal{K} $, ensuring that $ \mathcal{K}_0 $ is a valid simplicial complex. The geometric realization of this subcomplex is $|\mathcal{K}_0| = \bigcup_{i=1}^{n} |\sigma_i|$. Since $ A \subset \bigcup_{i=1}^{n} \operatorname{Int}(\sigma_i) $ and each simplex's closure is contained in $ |\sigma_i| $, we get $A \subset |\mathcal{K}_0|$.
\end{proof}

\begin{definition}[Locally Finite Simplicial Complex]\label{d:Locally Finite Simplicial Complex}
A  simplicial complex $\mcal{K}$ is said to be locally finite if every vertex $\btext{a} \in \mcal{K}$, the set $\{\sigma \in \mcal{K} \colon \btext{a} \in \sigma\}$ is finite i.e. each vertex is contained in only finitely many simplices. 

Equivalently, $\mcal{K}$ is locally finite is for each vertex $\btext{a} \in \mcal{K}$, its star $\operatorname{St}(\btext{a}, \mcal{K}) = \bigcup_{\sigma \in \mcal{K}, \btext{a} \in \sigma} \sigma$ is composed of only finitely may simplices. 
\end{definition}

\begin{theorem}\label{t:N and S condition for locally finiteness of simplicial complex in term of star of vertex}
A simplicial complex $\mcal{K}$ is locally finite if and only if each closed star $\mrm{st}(\btext{a}, \mcal{K})$ is polytope of a finite sub-complex of $\mcal{K}$.
\end{theorem}

\begin{proof}
Suppose $ \mathcal{K} $ is a locally finite simplicial complex. By definition, this means that for every vertex $ \btext{a} \in \mathcal{K} $, the set of simplices containing $ \btext{a} $ is finite i.e. $\{ \sigma \in \mathcal{K} \colon \btext{a} \in \sigma \}$  is finite. The closed star of a vertex $ \btext{a} $ is given by $\operatorname{St}(\btext{a}, \mathcal{K}) = \bigcup_{\sigma \in \mathcal{K}, \btext{a} \in \sigma} \sigma$. Since $ \mathcal{K} $ is locally finite, the number of simplices containing $ \btext{a} $ is finite. Thus, the union in the definition of $ \operatorname{St}(\btext{a}, \mathcal{K}) $ is taken over a finite collection of simplices. Let $ \mathcal{K}_{\btext{a}} $ be the smallest subcomplex of $ \mathcal{K} $ that contains all simplices in $ \operatorname{St}(\btext{a}, \mathcal{K}) $. Since this collection is finite, $ \mathcal{K}_{\btext{a}} $ is a finite subcomplex of $ \mathcal{K} $. The geometric realization $ |\operatorname{St}(\btext{a}, \mathcal{K})| $ is a bounded region in Euclidean space formed by the simplices in $ \mathcal{K}_{\btext{a}} $. Since $ \mathcal{K}_{\btext{a}} $ is a finite subcomplex, its geometric realization forms a polytope. Thus we have shown that $ \operatorname{St}(\btext{a}, \mathcal{K}) $ is the polytope of a finite subcomplex of $ \mathcal{K} $.

Conversely, assume that for every vertex $ \btext{a} \in \mathcal{K} $, the closed star $ \operatorname{St}(\btext{a}, \mathcal{K}) $ is the polytope of a finite subcomplex $ \mathcal{K}_{\btext{a}} \subset \mathcal{K} $. By assumption, $ \mathcal{K}_{\btext{a}} $ is a finite subcomplex of $ \mathcal{K} $. This means that it contains only finitely many simplices. Since $ \mathcal{K}_{\btext{a}} $ contains all simplices that contribute to $ \operatorname{St}(\btext{a}, \mathcal{K}) $, the number of simplices in $ \mathcal{K} $ that contain $ \btext{a} $ is at most the number of simplices in $ \mathcal{K}_{\btext{a}} $, which is finite. Since we have established that for every vertex $ \btext{a} $, the set $\{ \sigma \in \mathcal{K} \colon \btext{a} \in \sigma \}$ is finite, it follows that $ \mathcal{K} $ satisfies the local finiteness condition. Thus, $ \mathcal{K} $ is a locally finite simplicial complex.

\end{proof}

\begin{lemma}\label{l:N and S condition for locally finiteness of simplicial complex in term of locally compact}
The simplicial complex $\mcal{K}$ is locally finite if and only if the their underlying space $|\mcal{K}|$ is locally compact. 
\end{lemma}

\begin{proof}
Suppose that $ \mathcal{K} $ is a locally finite simplicial complex, meaning that every vertex $ \btext{a} $ of $ \mathcal{K} $ is contained in only finitely many simplices. Our aim is to show that $\mcal{K}$ is locally compact\footnote{A topological space $ X $ is locally compact if every point $ x \in X $ has a compact neighborhood, i.e., for every $ x \in X $, there exists an open set $ U $ containing $ x $ such that the closure $ \overline{U} $ is compact}. We need  to show that every point in $ |\mathcal{K}| $ has a compact neighborhood and for this we construct it as follows. Consider a point $ \btext{x} \in |\mathcal{K}| $, the geometric realization of the simplicial complex. There are two cases:
\begin{itemize}
    \item \textbf{Case 1:} If $ \btext{x} $ is a vertex of $ \mathcal{K} $, say $ \btext{x} = \btext{a} $, then the closed star $ \operatorname{St}(\btext{a}, \mathcal{K}) $ is defined as
    $\operatorname{St}(\btext{a}, \mathcal{K}) = \bigcup_{\sigma \in \mathcal{K}, \btext{a} \in \sigma} \sigma$.
    Since $ \mathcal{K} $ is locally finite, only finitely many simplices contain $ \btext{a} $. Hence, the realization $ |\operatorname{St}(\btext{a}, \mathcal{K})| $ is a finite union of compact simplices (each of which is compact in the subspace topology of Euclidean space). A finite union of compact sets is compact, so $ |\operatorname{St}(\btext{a}, \mathcal{K})| $ is a compact neighborhood of $ \btext{a} $.

    \item \textbf{Case 2:} If $ \btext{x} $ is an interior point of a higher-dimensional simplex, it still belongs to finitely many simplices due to local finiteness. A small enough open neighborhood of $ \btext{x} $ will be contained in the closed star of some nearby vertex, ensuring compactness.
\end{itemize}
Since every point in $ |\mathcal{K}| $ has a compact neighborhood, $ |\mathcal{K}| $ is locally compact.

Conversely, assume that $ |\mathcal{K}| $ is locally compact, meaning that for every point $ x \in |\mathcal{K}| $, there exists an open neighborhood $ U $ such that $ \overline{U} $ is compact.

For any vertex $ \btext{a} $ in $ \mathcal{K} $, take a small open neighborhood $ U $ around $ \btext{a} $ such that $ \overline{U} $ is compact. Since $ |\mathcal{K}| $ is a simplicial complex, $ \overline{U} $ must be covered by finitely many simplices. Since $ \overline{U} $ is compact and each simplex in $ \mathcal{K} $ is a compact subset of Euclidean space, only finitely many simplices can intersect $ \overline{U} $. But every simplex containing $ \btext{a} $ must be entirely contained in $ \overline{U} $, implying that only finitely many simplices contain $ \btext{a} $. Since this holds for every vertex $ \btext{a} $, we conclude that $ \mathcal{K} $ is locally finite.
\end{proof}

\begin{comment}

For more detail see  \cite{Munkres1996ElementsAlgebraicTopology}, \cite{Hatcher2002AlgebraicTopology}, \cite{Ghrist2014ElementaryAppliedTopology}, \cite{10942887}, \cite{Is2025}, \cite{Salepci2025}, \cite{Vaupel2025}, \cite{Chan1994}, \cite{Grieve2025}, \cite{Stanley2025}

\end{comment}

Recent developments in algebraic and applied topology provide important context for the present work on simplicial complexes and their realizations \cite{Hatcher2002AlgebraicTopology,Munkres1996ElementsAlgebraicTopology,Ghrist2014ElementaryAppliedTopology}. Building on these foundations, a number of contemporary studies explore both theoretical and computational aspects of simplicial complexes and related topological structures. Grande et al.\ use topological methods on simplicial complexes to classify trajectories and infer landmarks on surfaces, illustrating how higher-order connectivity can improve data and signal analysis \cite{10942887}. Is compares several notions of topological complexity, clarifying their relationships and computational behavior in different settings \cite{Is2025}, while Stanley and Ruiz develop double homology and study wedge-decomposable simplicial complexes to better understand their algebraic structure \cite{Stanley2025}. Random and probabilistic viewpoints appear in the analysis of asymptotic Betti numbers of random subcomplexes by Salepci and Welschinger, linking random models with classical homological invariants \cite{Salepci2025}. Section complexes of simplicial height functions, initiated by Chan and extended in a modern homotopical framework by Vaupel, Hermansen, and Trygsland, provide tools for analyzing level sets and filtrations on simplicial spaces \cite{Chan1994,Vaupel2025}. From a computational perspective, Grieve’s implementation of abstract simplicial complexes in Macaulay2 offers a software environment for experimenting with these objects and computing their algebraic invariants \cite{Grieve2025}.

%\section{Appendix}\label{s:Appendix}

%***************************************************************************************************************************
%                                           Bibliography
%***************************************************************************************************************************
\bibliographystyle{unsrtnat}
\bibliography{MainBibFile}
\end{document}

%% file: ShortcutSymbolsListForLatex.tex
\newcommand{\mcal}[1]{\mathcal{#1}}
\newcommand{\mbb}[1]{\mathbb{#1}}

\newcommand{\mfr}[1]{\mathfrak{#1}}
\newcommand{\mrm}[1]{\mathrm{#1}}
\newcommand{\btext}[1]{\mathbf{#1}}

\newcommand{\R}{{\mathbb R}}

\newcommand{\N}{{\mathbb N}}